\documentclass[english]{article}
\usepackage[T1]{fontenc}
\usepackage[latin9]{inputenc}
\usepackage{mathtools}
\usepackage{amsthm}
\usepackage{amssymb}
\usepackage{stackrel}
\usepackage{xcolor}
\usepackage[margin = 3 cm,marginpar= 2cm]{geometry}

\makeatletter
\newtheorem{thm}{Theorem}[section]
\newtheorem{defn}[thm]{Definition}
\newtheorem{prop}[thm]{Proposition}
\newtheorem{lem}[thm]{Lemma}
\newtheorem{cor}[thm]{Corollary}

\newtheorem{rem}[thm]{Remark}

\numberwithin{equation}{section}

\@ifundefined{date}{}{\date{}}
\makeatother

\usepackage{babel}

\begin{document}
\title{Singular flows with time-varying weights}
\author{
Immanuel Ben-Porat%
     \thanks{Universit\"at Basel
 Spiegelgasse 1
 CH-4051 Basel, Switzerland  {Email: immanuel.ben-porath@unibas.ch}} %
\and
    Jos\'e A. Carrillo%
     \thanks{Mathematical Institute, University of Oxford, Oxford OX2 6GG, UK.  {Email: carrillo@maths.ox.ac.uk}} %
     \and
    Pierre-Emmanuel Jabin%
     \thanks{Pennsylvania State University, 109 McAllister University Park, PA 16802 USA,  {Email: pejabin@psu.edu}} %
}

\maketitle

\begin{abstract}
We study the mean field limit for singular dynamics with time evolving weights. Our results are an extension of the work of Serfaty \cite{duerinckx2020mean} and Bresch-Jabin-Wang \cite{bresch2019modulated}, which consider singular Coulomb flows with weights which are constant time. The inclusion of time dependent weights necessitates the commutator estimates of \cite{duerinckx2020mean,bresch2019modulated}, as well as a new functional inequality. The well-posedness of the mean field PDE and the associated system of trajectories is also proved. 
\end{abstract}


\section{Introduction}

We aim to study the well posedness and the mean field limit of the following doubly non-local transport PDE 
\begin{equation}
\partial_{t}\mu-\mathrm{div}(\mu\,\mathbf{a}\star\mu)=h\left[\mu\right],\ \mu(0,x)=\mu_{0}\,,\label{eq:-13}
\end{equation}
where $\mathbf{a}:\mathbb{R}^{d}\rightarrow\mathbb{R}^{d}$ is a given
vector field and $h\left[\mu\right]$ is the signed measure given by 
\begin{equation}
h\left[\mu\right]\coloneqq\mu(x)\int_{\mathbb{R}^{d}}S(x-y)\mu(y)dy\,,\label{eq:-14}
\end{equation}
for some function $S:\mathbb{R}^{d}\rightarrow\mathbb{R}$.
We use $\mu(x) dx$ as the notation for integrating against the measure $\mu$ independently of being absolutely continuous or not with respect to the Lebesgue measure. Moreover, we shall denote by $\mu(x)$ its density with respect to the Lebesgue measure in case it is absolutely continuous. Here, both $\mathbf{a}$ and $S$ satisfy some appropriate structural hypothesis specified in Section 2. The initial data $\mu_{0}$ belongs to $\mathcal{P}(\mathbb{R}^{d})$, where $\mathcal{P}(\mathbb{R}^{d})$ is the set of  probability measures on $\mathbb{R}^{d}$. The solutions $\mu(t,\cdot)\in \mathcal{P}(\mathbb{R}^{d})$ sought after are curves of probability measures, and thus the source term $h\left[\mu\right]$ has to have zero average. 

The existence and uniqueness theory of PDEs of the type (\ref{eq:-13}) or variants thereof has been
handled in several works: in \cite{piccoli2014generalized} the existence and uniqueness
of measure valued solutions has been established for non-negative
source terms $h\left[\mu\right]$ which satisfy a boundeness and Lipschitz
condition with respect to the Wasserstein distance. In \cite{piccoli2018measure},
a variant of the Wasserstein distance for signed measures has been
introduced and applied in order to remove the positivity assumption
on $h$. The non-positivity assumption is important if one wishes to study models in which $h\left[\mu\right]$ has zero mean. Both of these results do not cover the important case where $h\left[\mu\right]$ is given by the formula (\ref{eq:-14}), since such source terms typically satisfy the boundeness and Lipschitz conditions required only locally. This obstacle has been successfully overcome
in \cite{duteil2021mean}, in which the derivation of the PDE \eqref{eq:-13} as a mean field limit is also studied. 

Let us denote by $\mathbb{R}^{dN}\setminus \triangle_{N}$ the set of non-overlapping configurations, that is, the set of $\mathbf{x}_{N}:=(x_1,\dots,x_N)\in \mathbb{R}^{dN}$ such that $x_{i}\neq x_{j}$ for all $i\neq j$, and let us denote by $ \mathbb{M}^{N}$ the set of admissible weights, that is, the set of $\mathbf{m}_{N}:=(m_1,\dots,m_N)$ such that $m_i\geq 0$, with $\sum_{i=1}^{N} m_i=N$. 
We can define the empirical measure associated to this configuration as
\begin{align}
\mu_{N}\coloneqq \frac{1}{N}\stackrel[i=1]{N}{\sum}m_{i}\delta_{x_{i}}. \label{empirical intro}    
\end{align}
The following system of ODEs governs the particle dynamics:
\begin{equation}
\left\{ \begin{array}{lc}
\dot{x}_{i}^{N}(t)=-\frac{1}{N}\stackrel[j=1]{N}{\sum}m_{j}^{N}(t)\mathbf{a}(x_{i}^{N}(t)-x_{j}^{N}(t)),\ x_{i}^{N}(0)=x_{i}^{0,N}\\
\dot{m}_{i}^{N}(t)=\frac{1}{N}\stackrel[j=1]{N}{\sum}m_{i}^{N}(t)m_{j}^{N}(t)S(x_{i}^{N}(t)-x_{j}^{N}(t)),\ m_{i}^{N}(0)=m_{i}^{0,N}.
\end{array}\right.\label{eq:-23}
\end{equation}
Subject to several technical assumptions on the functions $\mathbf{a}$ and $S$, the ODE system (\ref{eq:-23}) has a globally well defined flow $(\mathbf{x}_{N}(t),\mathbf{m}_{N}(t))$, see \cite[Theorem 3]{ayi2021mean}, where $\mathbf{x}_{N}(t):=(x_1^N(t),\dots,x_N^N(t))\in \mathbb{R}^{dN}$ and $\mathbf{m}_{N}(t):=(m_1^N(t),\dots,m_N^N(t))\in \mathbb{M}^{N}$. We will denote by $\mu_{N}(t)$ the empirical measure associated to the system (\ref{eq:-23}), with a slight abuse of notation, as the empirical measure given by \eqref{empirical intro} associated to the configuration $(\mathbf{x}_{N}(t),\mathbf{m}_{N}(t))$.
The empirical measure is readily checked to be a distributional
solution of the PDE (\ref{eq:-13}). We say that \eqref{eq:-13} is the mean field limit of the interacting particle system \eqref{eq:-23} if 
\[
W_{1}(\mu_{N}(0),\mu_{0})\underset{N\rightarrow\infty}{\rightarrow}0\Longrightarrow W_{1}(\mu_{N}(t),\mu(t))\underset{N\rightarrow\infty}{\rightarrow}0,\ t\in[0,T],
\]
where $\mu(t)$ is the solution to the PDE \eqref{eq:-13} with initial
data $\mu_{0}$ at time $t$ and $W_1$ is the Wasserstein (or Monge-Kantorovich) distance. Before we proceed with the mathematical discussion, let us say a few words about the real life phenomena that the above models idealizes. Given $N$ agents interacting pairwise via a given interaction $\mathbf{a}$ with time evolving opinions $(x_{1}^N(t),\cdots,x_{N}^N(t))$ and  weights of influence $(m_{1}^N(t),\cdots,m_{N}^N(t))$ the system \eqref{eq:-23} describes the evolution both of the opinions and the weights of influence in time. As it can be seen from \eqref{eq:-23},  the $i$-th opinion and the $i$-th weight evolve according to a law which takes into consideration all the other involved opinions and influences. 

A notable example of a model of this type is the Kuramoto model and its variants \cite{aoki2009co,cestnik2024continuum}. Similar models appear also in crowd dynamics and have received mathematical treatment in \cite{mcquade2019social} and \cite{piccoli2018measure}. The scientific literature seems to be so far mostly limited to treat continuous interaction kernels, although Hegselmann-Krause type models \cite{ceragioli2021generalized} includes also jump discontinuities. To the best of our knowledge, the blow up singularities which are  considered in the present work have not yet been studied when the weights evolve in time.

Viewed as a system of opinion dynamics, the mean field limit can be interpreted as a large population limit. This limit has been proven
in \cite{duteil2021mean}, with an approach which resembles the celebrated Dobrushin
theorem \cite{dobrushin1979vlasov}. In \cite{ayi2021mean} a weak version of the mean field limit has been recovered through a graph limit approach. By ``weak'' we mean that the result is limited to a special choice of initial configurations (at the microscopic level) and a special
choice of initial data for the Cauchy problem \eqref{eq:-13}. Among other works which consider the graph limit and its link with the mean field limit we mention \cite{biccari2019dynamics,paul2022microscopic} as well as  \cite{ayi2024large} for a general overview. The case where the weights form a non-symmetric matrix, a scenario which arises for instance in Neuron dynamics, is also very interesting and has been analyzed in \cite{jabin2024mean}.  All of these results considered measure valued solutions,
and typically one has to require that the functions $\mathbf{a}$
and $S$ at least satisfy some Lipschitz continuity. However, the
case where the function $\mathbf{a}$ admits discontinuities is also
relevant and requires a different mathematical treatment. In the recent
work \cite{ben2024mean}, the well posedness and mean field limit has been
established for the one dimensional attractive Coulomb kernel, which corresponds
to the choice  $\mathbf{a}(x)=\mathrm{sgn}(x)$ in one dimension. The argument proposed
there relies on recasting the limit PDE as a Burgers type equation only valid in one dimension,
and invoking Kruzkhov theory of conservation laws in order to prove
stability estimates for the resulting Burgers equation. It should be mentioned that singularities emerging from $S$ are also relevant and have been considered in \cite{ben2024graph,mcquade2019social}. 

The purpose of this work is twofold: to study the well posedness and
mean field limit for \eqref{eq:-13} for arbitrary dimension $d\geq3$ in the case where $\mathbf{a}$ exhibits a Coulomb type singularity, thereby addressing a question left open in \cite{ben2024mean}.
More precisely, we assume the following hypotheses (\textbf{H1}) on the functions $\mathbf{a}$ and $S$ in \eqref{eq:-13}:
\begin{itemize}
\item (\textbf{H1}-i): $\mathbf{a}=\mathbb{J}\nabla V$, with $V(x)$ the $d-$dimensional repulsive Coulomb interaction ($d\geq 3$), $V(x)=-\frac{|x|^{2-d}}{2-d}$. Here $\mathbb{J}$ is either a $d\times d$ anti-symmetric matrix or the identity matrix. 

\item (\textbf{H1}-ii): $S\in \mathcal{S}(\mathbb{R}^{d})$ is an odd ($S(-x)=-S(x)$) function.

\end{itemize}
Note that $\mathrm{div}(\mathbb{J}\nabla V)=0$ when $\mathbb{J}$ is anti-symmetric and $\mathrm{div}(\mathbb{J}\nabla V)=-\delta_{0}$ when $\mathbb{J}$ is the identity. In comparison to \cite{ben2024mean}, the difficulties which must be addressed
in the new singular settings are reflected on several levels:
First, one has to justify why  \eqref{eq:-23} has a globally
well defined flow. This result is classical when the weights are time independent and follows from the observation that initial separation of opinions is propagated in time, i.e. 
$$
x_{i}^{0,N}\neq x_{j}^{0,N}\Longrightarrow\underset{i\neq j}{\min}\left|x_{i}^N(t)-x_{j}^N(t)\right|>0.
$$
The inclusion of weights which evolve in time necessitates imposing a condition on $S$ which would prevent a Ricatti type blow-up, and this is why we impose the parity condition $S(x)=-S(-x)$. The main difficulty for obtaining the well-posedness at the macroscopic level lies in the fact that when handling discontinuous $V$ there is no reason to expect that the term $\mathbb{J}\nabla V\star\mu$ verifies a Lipschitz condition on the space of probability \textit{measures} with respect to the Wasserstein distance. This fact eliminates the possibility
of proving the stability estimate with respect to the Wasserstein distance. We are therefore led to restrict
the solutions considered to more well behaved function spaces, namely
$L^{p}$ or Sobolev spaces. Thirdly, since the stability estimate at our disposal
is restricted to $L^{p}$ spaces, we are unable to apply it directly
for the empirical measure in order to obtain the long time convergence.
The most difficult part of this work boils down to overcoming this
issue. Our main idea  is to make use of the functional inequalities
discovered in \cite{duerinckx2020mean,bresch2019modulated} as well as proving new functional inequalities which are necessary due to the inclusion of a source term. The functional inequalities in \cite{duerinckx2020mean} reflect  a particularly exciting development in the theory of mean field limits, as they allow for the first time to rigorously identify the mean field limit of Coulomb flows and have already found numerous applications both in classical mean field limits \cite{rosenzweig2020justification,han2021newton,menard2024mean} as well as in quantum many body systems \cite{rosenzweig2021quantum,golse2022mean}. Before stating our two main results, namely the well posedness and the mean field limit, we introduce the following hypotheses:

\vskip 6pt

(\textbf{H2}) The initial data $\mu_{0}\in\mathcal{P}(\mathbb{R}^{d})\cap W^{1,\infty}(\mathbb{R}^{d})$ is such that there is some $R>0$ with
$\mathrm{supp}(\mu_{0})\subset B(0,R)$.

\vskip 6pt

\begin{thm} \label{existence uniqueness intro}
Let assumptions (\textbf{H1})-(\textbf{H2}) hold. Then, there exists a unique solution to the problem \eqref{eq:-13}
with initial data $\mu_{0}$. Moreover, this solution satisfies $\mu \in C\left([0,T];L^{p}(\mathbb{R}^{d})\right)$ for all $1\leq p \leq \infty$ and $\mu \in L^{\infty}\left([0,T];W^{1,p}(\mathbb{R}^{d})\right)$ for all $1\leq p< \infty$. 
Furthermore, $\mu(t,\cdot)$ is compactly supported and its support satisfies 
\[
\mathrm{supp}(\mu(t,\cdot))\subset B(0,\overline{R})\,, \mbox{ for all } t\in[0,T],
\]
where $\overline{R}=\overline{R}(\left\Vert \mu_{0}\right\Vert _{L^\infty },R,T,\left\Vert S\right\Vert _{L^\infty})$.   
\end{thm}
\begin{thm}
\label{main thm intro} Let assumptions (\textbf{H1})-(\textbf{H2}) hold and let $\mu(t,\cdot)$ be the unique solution with initial data $\mu_{0}$  ensured via Theorem \ref{existence uniqueness intro}. Let $\mathbf{m}_{N}^{0}\in \mathbb{M}^{N}$ such that $m_{i}^{0,N}\leq M$ for some $M>0$ and  $\mathbf{x}_{N}^{0}\in \mathbb{R}^{dN}\setminus \triangle_{N}$. Denote by $(\mathbf{x}_{N}(t),\mathbf{m}_{N}(t))\in C([0,T];\mathbb{R}^{dN}\times\mathbb{R}^{N})$ the solution to the system of ODEs \eqref{eq:opinion dynamics sec 2} with initial data $(\mathbf{x}_{N}^{0},\mathbf{m}_{N}^{0})$ ensured via Theorem \ref{well posedness -1}. Then 

 \begin{align*}
   \mu_{N}(t,\cdot)\underset{N\rightarrow \infty}{ \rightarrow}\mu(t,\cdot)  \ \mbox{in the weak sense locally uniformly in time}
 \end{align*}  
 provided that $\mathcal{E}_{N}(0)\underset{N\rightarrow\infty}{\rightarrow}0$. 
\end{thm}

The paper unfolds as follows. In Section 2 we recap basic
structural properties of the model and state the main results. Afterwards, we include an important discussion in which we elaborate on the novelty of our method. Section
3  is devoted to proving the existence of solutions to \eqref{eq:-13} and the
stability estimate, from which the uniqueness follows at once. Section 4 is devoted to proving the global well-posedness of the ODE (\ref{eq:-23}). Finally, Section
5 is devoted to modifying the modulated energy approach introduced
in \cite{duerinckx2020mean} and extended in \cite{bresch2019modulated}, thus yielding the mean field limit.

\section{Preliminaries }

\subsection{The equation for the trajectories}

The dynamics that we consider are governed by the following system of
$(d+1)\times N$ ODEs 
\begin{equation}
\left\{ \begin{array}{lc}
\dot{x}_{i}^{N}(t)=-\frac{1}{N}\stackrel[j=1]{N}{\sum}m_{j}^{N}(t)\mathbb{J}\nabla V(x_{i}^{N}(t)-x_{j}^{N}(t)),\ x_{i}^{N}(0)=x_{i}^{0,N}\\
\dot{m}_{i}^{N}(t)=\frac{1}{N}\stackrel[j=1]{N}{\sum}m_{i}^{N}(t)m_{j}^{N}(t)S(x_{i}^{N}(t)-x_{j}^{N}(t)),\ m_{i}^{N}(0)=m_{i}^{0,N}
\end{array}\right..\label{eq:opinion dynamics sec 2}
\end{equation}
 The notation is as follows: the unknowns are $x_{i}^N\in\mathbb{R}^{d}$
and $m_{i}^N\in\mathbb{R}$ are referred to as opinions and weights
respectively, and are supplemented with initial data $x_{i}^{0,N},m_{i}^{0,N}$.

\begin{defn}
We say that $(\mathbf{x}_{N}(t),\mathbf{m}_{N}(t))\in C([0,T];\mathbb{R}^{dN}\times\mathbb{R}^{N})$
is a solution to the system of ODEs \eqref{eq:opinion dynamics sec 2} on $[0,T)$ ($T\leq\infty$)
if 
\[
\underset{i\neq j}{\min}\left|x_{i}
^{N}(t)-x_{j}^{N}(t)\right|>0,\ t\in[0,T)
\]
and for all $t\in[0,T)$ and all $1\leq i \leq N$ it holds that

\[
\left\{ \begin{array}{lc}
x_{i}^{N}(t)=x_{i}^{0,N}-\frac{1}{N}\stackrel[j=1]{N}{\sum}\displaystyle\int_{0}^{t}m_{j}^{N}(\tau)\mathbb{J}\nabla V(x_{i}^{N}(\tau)-x_{j}^{N}(\tau))d\tau\\
m_{i}^{N}(t)=m_{i}^{0,N}+\frac{1}{N}\stackrel[j=1]{N}{\sum}\displaystyle\int_{0}^{t}m_{i}^{N}(\tau)m_{j}^{N}(\tau)S(x_{i}^{N}(\tau)-x_{j}^{N}(\tau))d\tau.
\end{array}\right.
\]
We say that $(\mathbf{x}_{N}(t),\mathbf{m}_{N}(t))$ is a solution
to the system of ODEs \eqref{eq:opinion dynamics sec 2} with maximal life span $T>0$ if it is a solution on $[0,T)$ to \eqref{eq:opinion dynamics sec 2}, but is not a solution on $[0,T]$. 
\end{defn}
The following theorem establishes that there is a well-defined flow for the system of ODEs (\ref{eq:opinion dynamics sec 2}). The proof is postponed to Section \ref{WELL POSEDNESS SEC}.

\begin{thm}
\label{well posedness -1} Let hypothesis (\textbf{H1}) hold. Suppose that for any $N\in \mathbb{N}$ it holds that $\mathbf{x}_{N}^{0}\in \mathbb{R}^{dN}\setminus\triangle_{N}$, $\mathbf{m}_{N}^{0}\in \mathbb{M}_{N}$ and there is some $M>0$ such that for all $1\leq i\leq N$ it holds that 
\begin{equation*}
0\leq m_{i}^{0,N}\leq M.
\end{equation*} 
Then, the system of ODEs \eqref{eq:opinion dynamics sec 2} has a unique global solution $(\mathbf{x}_{N}(t),\mathbf{m}_{N}(t))\in C^{1}([0,\infty);\mathbb{R}^{dN}\times\mathbb{R}^{N})$
solution with initial data $(\mathbf{x}_{N}^{0},\mathbf{m}_{N}^{0})$. In particular $x_{i}^{N}(t)\neq x_{j}^{N}(t)$ for all $t \in [0,\infty)$ and all $i\neq j$. 
\end{thm} 
\begin{rem} 
Since $V$ is the Coulomb interaction
there is some constant $\mathbf{C}>0$ such that 
\begin{equation*}
\left\Vert \nabla^{2}V\star \mu \right\Vert_{2} \leq \mathbf{C}\left\Vert \mu \right\Vert_{2}     
\end{equation*}
and  
\begin{equation*}
 0\leq \widehat{V}(x)\leq \frac{\mathbf{C}}{\left\vert x \right\vert^{2}},    
\end{equation*}
where $\widehat{V}$ is the Fourier transform of $V$. The first inequality is the Calderon-Zygmund inequality (see \cite[Theorem 4.12]{duoandikoetxea2024fourier}), and the second inequality is due to the fact that $-\Delta V=\delta_{0}$.  
\end{rem}

\begin{rem}
The assumption that $S$ is odd implies that any solution $(\mathbf{x}_{N}(t),\mathbf{m}_{N}(t))$ of the system of ODEs \eqref{eq:opinion dynamics sec 2} has $m_{i}^{N}(t) \geq 0$ for all $t\in [0,T]$ and  conserves the total weight, i.e. 
\begin{align*}
    \frac{1}{N} \stackrel[i=1]{N}\sum m_{i}^{N}(t)=1
\end{align*}
for all $t\in[0,T]$. In addition there is a constant $\overline{M}=\overline{M}(M,\left\Vert S\right\Vert_{\infty})$ such that for all $t\in [0,T]$ it holds that 
$
    m_{i}^{N}(t)\leq \overline{M}, \ i=1,\dots,N.
$
\label{conservation of total weight}
\end{rem}

\subsection{The mean field equation }

Recall that the mean field limit PDE that is expected to be derived
from (\ref{eq:opinion dynamics sec 2}) is 
\begin{equation}
\partial_{t}\mu-\mathrm{div}(\mu\mathbb{J}\nabla V\star\mu)=h\left[\mu\right],\ \mu(0,x)=\mu
_{0}.\label{mean field equation sec2}
\end{equation}
We introduce the notation $\mathbf{S}[\mu]\coloneqq\int_{\mathbb{R}^{d}}S(x-y)\mu(y)dy$ and $\mathbf{A}[\mu]\coloneqq\mathbb{J}\nabla V\star \mu$.  
We endow the space $C([0,T];\mathcal{P}(\mathbb{R}^{d}))$ with the metric 
$$
D(\mu,\nu)\coloneqq \underset{t\in[0,T]}{\sup} W_{1}(\mu(t,\cdot),\nu(t,\cdot))
$$
where $W_{1}$ is the Wasserstein distance (see \cite{villani2009optimal}). Our notion of weak solution is as follows.  

\begin{defn} \label{def of weak sol}
Assume $V$ and $S$ satisfy (\textbf{H1}) and let $\mu_{0}\in L^{1}(\mathbb{R}^{d})\cap L^{\infty}(\mathbb{R}^{d})$. A time dependent $\mu\in C([0,T]; L^{1}(\mathbb{R}^{d})) \cap L^\infty(0,T;L^{\infty}(\mathbb{R}^{d}))$
is said to be a weak solution of \eqref{mean field equation sec2}
if for all $\varphi\in C_{0}^{\infty}(\mathbb{R}^{d})$, 
it holds that 
\[
\frac{d}{dt}\int_{\mathbb{R}^{d}}\varphi(x)\mu(t,x)dx+\int_{\mathbb{R}^{d}}\nabla\varphi(x)\mathbf{A}[\mu](t,x)\mu(t,x)dx=\int_{\mathbb{R}^{d}}h[\mu](t,x)\varphi(x)dx,\ \mu(0,x)=\mu_0,
\]
in the distributional sense.
\end{defn}

Notice that the second term in the definition of weak solution to \eqref{mean field equation sec2} makes sense by a separation into short and long range terms in $\nabla V\star \mu$ using that $\mu(t,\cdot)\in L^{1}(\mathbb{R}^{d})\cap L^{\infty}(\mathbb{R}^{d})$ for all $t\in [0,T]$.

\begin{rem} 
As remarked in \cite{duteil2021mean}, an equivalent definition to Definition \ref{def of weak sol} is to demand that for each  $\varphi\in C^{\infty}((0,T)\times\mathbb{R}^{d}))$ it holds that 
\[
\int_{0}^{T}\int_{\mathbb{R}^{d}}\left[\partial_{t}\varphi(t,x)\mu(t,x)+\nabla\varphi(t,x)\mathbf{A}\left[\mu \right](t,x)\mu(t,x)-\varphi(t,x)h\left[\mu\right](t,x)\right] dxdt=0.
\]
\end{rem}

\begin{rem}\label{aux2}
In case $\nabla V$ is Lipschitz, the definition of weak solution to \eqref{mean field equation sec2} can be done assuming only $\mu_{0}\in L^{1}(\mathbb{R}^d)$.
\end{rem}

We will frequently need several structural properties of the source term which are not influenced by the regularity imposed on $S$. We summarize these properties in the following result. 

\begin{prop}
\textup{\cite[Proposition 9]{duteil2021mean}} For all $\mu\in\mathcal{P}(\mathbb{R}^{d})$, it holds that $\mathrm{supp}(h\left[\mu\right])=\mathrm{supp}(\mu)$, $\left|h\left[\mu\right]\right|\leq\left\Vert S\right\Vert _{\infty}\mu$, and 
$$
\int_{\mathbb{R}^{d}}h\left[\mu\right](x)dx=0.
$$
\end{prop}
\subsection{Comments on the method}
The most technically involved part of this work is the derivation of the PDE (\ref{mean field equation sec2}) as a mean field limit from the dynamics (\ref{eq:opinion dynamics sec 2}). Therefore we turn the reader's attention to several key points in the proof. First, let us consider the case where $S=0$. In this case the PDE (\ref{mean field equation sec2}) is homogeneous, and the weights are constant in time - for simplicity we can take $m_{i}^{N}=\frac{1}{N}$ for all $1\leq i \leq N$. The method in \cite{duerinckx2020mean} rests upon a clever renormalization argument of a weak strong stability principle for the limit PDE. Let us introduce first the weak strong stability principle at the level of the limit PDE followed by a brief outline of the renormalization procedure carried out in \cite{duerinckx2020mean}.

\textbf{The weak strong stability principle}. Given solutions $\mu_{1}(t,\cdot),\mu_{2}(t,\cdot)$ to the PDE (\ref{mean field equation sec2}),  consider the modulated energy 
\begin{align*}
 \mathcal{E}(t)\coloneqq\int_{\mathbb{R}^{d}}\left(\mu_{1}-\mu_{2}\right)(t,x)V\star\left(\mu_{1}-\mu_{2}\right)(t,x)dx.
 \end{align*}
A standard calculation reveals that 
\begin{align}
\frac{d}{dt}\mathcal{E}(t)&\leq-\int_{\mathbb{R}^{d}\times\mathbb{R}^{d}}\left(\mathbb{J}\nabla V\star\mu_{2}(t,x)-\mathbb{J}\nabla V\star\mu_{2}(t,y)\right)\nabla V(x-y)\left(\mu_{1}(t,\cdot)-\mu_{2}(t,\cdot)\right)^{\otimes2}dxdy\nonumber\\
&\leq 2\underset{t\in[0,T]}{\sup}\left\Vert \nabla^{2} V\star\mu_{2}(t,\cdot) \right\Vert_{\infty}\int_{\mathbb{R}^{d}} \left\vert \nabla V\star\left(\mu_{1}(t,\cdot)-\mu_{2}(t,\cdot)\right)\right\vert^{2}dx\nonumber\\
&\leq 2\underset{t\in[0,T]}{\sup}\left\Vert \nabla^{2} V\star\mu_{2}(t,\cdot) \right\Vert_{\infty}\mathcal{E}(t) ,
\label{stability sketch }
\end{align}
where we used the identity 
\begin{equation}
\int_{\mathbb{R}^{d}\times\mathbb{R}^{d}} V(x-y)\mu(x)\mu(y)dxdy=\int_{\mathbb{R}^{d}\times\mathbb{R}^{d}}\left\vert\nabla V\star \mu\right\vert^{2}(x)dx 
\label{interaction in Fourier}    
\end{equation}
for any $\mu\in \dot{H}^{-1}(\mathbb{R}^{d})$, $d\geq 3$, which can be readily seen in Fourier. Assuming that $\mu_{2}(0,\cdot)\in W^{1,p}(\mathbb{R}^{d})$  (for $p>1$ sufficiently large) and that $\mu_{2}$ enjoys propagation of Sobolev regularity, it can be shown that $\underset{t\in[0,T]}{\sup}\left\Vert \nabla^{2} V\star\mu_{2}(t,\cdot) \right\Vert_{\infty}\lesssim \left\Vert \mu_{2}(0,\cdot) \right\Vert_{W^{1,p}}$. Thus the inequality (\ref{stability sketch })
yields the stability estimate 
\begin{align*}
\mathcal{E}(t)\leq e^{Ct} \mathcal{E}(0)  \end{align*}
 for some constant $C=C(\left\Vert \mu_{2}(0,\cdot) \right\Vert_{W^{1,p}})$. It is remarkable that this argument necessitates only regularity on one of the solutions involved. 
 
 \textbf{Renormalization}. We now define the renormalized modulated energy associated to $\mu$ and $\mu_N$ as
\begin{align}\label{renomenergyNconf}
\mathcal{E}_{N}(\mu,\mu_N)&:=\int_{x\neq y}V(x-y)\left(\mu_{N}(\cdot)-\mu(\cdot)\right)^{\otimes2}dxdy.
\end{align}
Given a solution $\mu(t,\cdot)$ to the PDE \eqref{mean field equation sec2} and a solution $\mathbf{x}_{N}(t)$ to the system of ODEs \eqref{eq:opinion dynamics sec 2}, we define the renormalized modulated energy associated to them as
\begin{align}\label{renomenergyN}
\mathcal{E}_{N}(t):=\mathcal{E}_{N}(\mu(t,\cdot),\mu_N(t,\cdot)).
\end{align}
The definitions \eqref{renomenergyNconf} and \eqref{renomenergyN} extend in the obvious manner in the case of weighted empirical measures. A standard calculation reveals that 
\begin{align*}
\frac{d}{dt}\mathcal{E}_{N}(t)\leq&-\int_{x\neq y}\left(\mathbb{J}\nabla V\star\mu(t,x)-\mathbb{J}\nabla V\star\mu(t,y)\right)\nabla V(x-y)\left(\mu_{N}(t,\cdot)-\mu(t,\cdot)\right)^{\otimes2}dxdy.
\end{align*}
Note the removal of the diagonal in the integrals above, which is necessary due to the singularity of $V$ at the origin. The functional inequality which allows to close the estimate, and which reflects the outstanding novelty of  \cite{duerinckx2020mean} is that for any given configuration $\mathbf{x}_{N}\in \mathbb{R}^{dN}\setminus \triangle_{N}$,  bounded probability density $\mu$ and bounded Lipschitz vector field $u$ there holds the inequality 
\begin{align}
\left\vert\int_{x\neq y}\left(u(x)-u(y)\right)\nabla V(x-y)\left(\mu_{N}-\mu\right)^{\otimes2}dxdy\right\vert\leq C\int_{x\neq y}V(x-y)\left(\mu_{N}-\mu\right)^{\otimes2}dxdy+o_{N}(1),
\label{commutator sketch}
\end{align}
where $C>0$ is some harmless constant. 
We will refer to \eqref{commutator sketch} as the commutator estimate, see \cite{rosenzweig2021quantum}.
The representation (\ref{interaction in Fourier}) is inapplicable to the renormalized energy $\mathcal{E}_{N}(t)$ due to the removal of the diagonal. A key observation of \cite{duerinckx2020mean} is that the identity (\ref{interaction in Fourier}) admits a renormalized version. More precisely, denote by $\delta^{(\eta)}_{x}$ the uniform measure of mass $1$ on $\partial B(x,\eta)$. Given $\vec{\eta}=\left(\eta_{1},\ldots,\eta_{N}\right)$ set $\mu_{N}^{(\vec{\eta})}\coloneqq \frac{1}{N}\stackrel[i=1]{N}{\sum}\delta_{x_{i}}^{(\eta_{i})}$. Then, letting $r_{i}\coloneqq \frac{1}{4}\underset{i\neq j}{\min}{\left\vert x_{i}-x_{j}\right\vert}$  and $\eta_{i}\leq r_{i}$,  Serfaty proves the identity 
\begin{align}
\mathcal{E}_{N}(\mu,\mu_{N})= \frac{1}{c_d}\left(\int_{\mathbb{R}^{d}\times\mathbb{R}^{d}}\left\vert\nabla V\star (\mu^{(\eta)}_{N}-\mu)\right\vert^{2}(x)dx-\frac{c_{d}}{N}\stackrel[i=1]{N}{\sum}V(\eta_{i})\right)+\mathrm{negligible\ terms}
. \label{Representation}
\end{align}

\textbf{The weak strong stability principle with source term}. Coming back to the more general scenario  at stake, where a source term $h[\mu]$ is included, it can be shown by direct calculation 
that in this case $\mathcal{E}(t)$ verifies
\begin{align}
\frac{d}{dt}\mathcal{E}(t)\leq&-\int_{\mathbb{R}^{d}\times\mathbb{R}^{d}}\left(\mathbb{J}\nabla V\star\mu(t,x)-\mathbb{J}\nabla V\star\mu(t,y)\right)\nabla V(x-y)\left(\mu_{1}(t,\cdot)-\mu_{2}(t,\cdot)\right)^{\otimes2}dxdy \label{stability sketch with h} \\
&+2\int_{\mathbb{R}^{d}\times\mathbb{R}^{d}}V(x-y)\left(h\left[\mu_{1}\right](t,x)-h\left[\mu_{2}\right](t,x)\right)\left(\mu_{1}(t,y)-\mu_{2}(t,y)\right)dxdy\coloneqq \mathcal{D}^{1}(t)+ \mathcal{D}^{2}(t).     \notag
\end{align}
The first term $\mathcal{D}^{1}$ in the right hand side of \eqref{stability sketch with h} can be handled precisely as before.  Thus, the inclusion of the source term is manifested through the second term $\mathcal{D}^{2}$. In this context, our key observation is that $h$ is a Lipschitz operator on $\dot{H}^{-1}(\mathbb{R}^{d})$, namely we will show  (see Corollary \ref{h Lip coro}) an inequality of the form 
\begin{equation}
\left\Vert h[\mu_{1}]-h[\mu_{2}]\right\Vert _{\dot H^{-1}(\mathbb{R}^d)}\leq C\left\Vert \mu_{1}-\mu_{2}\right\Vert _{\dot H^{-1}(\mathbb{R}^d)}. \label{h lip pre}   
\end{equation}
Equipped with Inequality (\ref{h lip pre}) a bound on $\mathcal{D}^{2}(t)$ by means of $\mathcal{E}(t)$ can be obtained. Indeed,  appealing to Fourier we see that  
\begin{align*}
\mathcal{D}^{2}(t)&=2\int_{\mathbb{R}^{d}}V\star\left(h\left[\mu_{1}\right]-h\left[\mu_{2}\right]\right)(t,x)\left(\mu_{1}(t,x)-\mu_{2}(t,x)\right)dx\\&=2\int_{\mathbb{R}^{d}}\frac{\widehat{\left(h\left[\mu_{1}\right]-h\left[\mu_{2}\right]\right)}(t,x)}{\left\vert x \right\vert}\frac{\widehat{\left(\mu_{1}-\mu_{2}\right)}(t,x)}{\left\vert x \right\vert}dx\\&\leq 2\left\Vert h[\mu_{1}(t,\cdot)]-h[\mu_{2}(t,\cdot)]\right\Vert _{\dot H^{-1}(\mathbb{R}^d)}\left\Vert \mu_{1}(t,\cdot)-\mu_{2}(t,\cdot)\right\Vert _{\dot H^{-1}(\mathbb{R}^d)},    
\end{align*}
by Cauchy-Schwarz. This combined together with \eqref{h lip pre} shows that 
\begin{align*}
 \mathcal{D}^{2}(t) \lesssim  \left\Vert \mu_{1}(t,\cdot)-\mu_{2}(t,\cdot)\right\Vert _{\dot H^{-1}(\mathbb{R}^d)}^{2}=\mathcal{E}(t).  
\end{align*}
which in turn closes the Gr\"onwall estimate.

\textbf{Renormalization with source term}. We will show by a direct calculation (see Theorem \ref{Time derivative of energy}) the inequality 
\begin{align}
\frac{d}{dt}\mathcal{E}_{N}(t)\leq&-\int_{x\neq y}\left(\mathbb{J}\nabla V\star\mu(t,x)-\mathbb{J}\nabla V\star\mu(t,y)\right)\nabla V(x-y)\left(\mu_{N}(t,\cdot)-\mu(t,\cdot)\right)^{\otimes2}dxdy \notag\\
&+2\int_{x\neq y}V(x-y)\left(h\left[\mu_{N}(t,\cdot)\right](x)-h\left[\mu(t,\cdot)\right](x)\right)\left(\mu_{N}(t,y)-\mu(t,y)\right)dxdy\notag\\\coloneqq& \mathcal{D}_{N}^{1}(t)+ \mathcal{D}_{N}^{2}(t). \label{sketch sec 2}
\end{align}
The first term $\mathcal{D}_{N}^{1}$ in the right hand side of Inequality \eqref{sketch sec 2} can be handled precisely via the commutator estimate \eqref{commutator sketch} in the same manner outlined before. Thus, the main challenge is to renormalize the argument outlined in the previous paragraph in order to control $\mathcal{D}_{N}^{2}(t)$ by means of  $\mathcal{E}_{N}(t)$. A first possibility is to attempt proving an identity in the spirit of  (\ref{Representation}) for 
$$
\int_{x\neq y}V(x-y)(h[\mu_{N}]-h[\mu])(x)(h[\mu_{N}]-h[\mu])(y)dxdy. 
$$
However, it is not clear if the fine cancellations which produce this formula can be extended in this manner, since the operation $\mu \mapsto h[\mu]$ is non linear. We therefore pursue a different path inspired by the work of Bresch-Jabin-Wang  \cite{bresch2019modulated}, which avoids making use of this representation. The merit of this approach is that it does not necessitate any algebraic identities, but only inequalities. With this approach we will prove a functional inequality of the form 
\begin{align}
\int_{x\neq y}\!\!\!\!\!\!V(x-y)\left(\mu_{N}-\mu\right)(x)\left(h\left[\mu_{N}\right]-h\left[\mu\right]\right)(y)dxdy\leq 
C\int_{x\neq y}\!\!\!\!\!\!V(x-y)\left(\mu_{N}-\mu\right)^{\otimes2}dxdy+o_{N}(1),
\label{functional ine sec 2}
\end{align}
for some constant $C>0$. 

\section{The well posedness of the mean field equation }
\subsection{Existence }
We start by recalling existence and uniqueness for the limit PDE as
established in \cite{duteil2021mean}.  
\begin{prop}
\label{Cauchy sequence} \textup{(\cite[Theorem 2]{duteil2021mean})} Assume that $\nabla V$ is Lipschitz, let $S\in W^{1,\infty}(\mathbb{R}^{d})$ be odd and let the initial data $\mu_{0}$ satisfy  (\textbf{H2}). Then, there exist a unique solution $\mu\in C([0,T];\mathcal{P}(\mathbb{R}^{d}))$
of \eqref{mean field equation sec2} in the sense of Definition \ref{def of weak sol} and Remark \ref{aux2}. Moreover, $\mu(t,\cdot)$ is compactly supported for all $t\in [0,T]$. 
\end{prop}
\begin{rem} \label{classical solutions}
If one further assumes $V\in \mathcal{S}(\mathbb{R}^{d})$, $S \in \mathcal{S}(\mathbb{R}^{d})$ and $\mu_{0}\in C_{0}^{\infty}(\mathbb{R}^{d})$ in Proposition \ref{Cauchy sequence}, then it is possible to show that the solution $\mu$ provided by  Proposition \ref{Cauchy sequence} is a smooth compactly supported function. Consequently, subject to these assumption the solution is classical.     
\end{rem}
The main theorem of this section concerns the existence of a solution to \eqref{mean field equation sec2} and is based on a stability argument using suitably mollified problems. An important component of the proof is the propagation of Sobolev norms. Sobolev regularity is important
in order to be able to use the commutator estimate \eqref{commutator sketch} against Lipschitz vector fields. In this context, we will need the following near-boundedness of the Calderon-Zygmund operator on $L^{\infty}$.   
\begin{prop} 
\label{near boundedness}  
\textup{(\cite[Proposition 7.7]{bahouri2011fourier})}
Let $s>0$  and let $a\in[1,\infty)$ and $b\in[1,\infty]$. Then, there is some $C>0$ such that 
\begin{align*}
\left\Vert \nabla^2 V\star \mu \right\Vert_{\infty}\leq 
C\left(\mathrm{min} \left(\left\Vert\nabla V\star \mu\right\Vert_{b},\left\Vert \mu\right\Vert_{a}\right)+\left\Vert \mu\right\Vert_{\infty}(1+\left|\log\left\Vert \mu\right\Vert_{\infty}\right|\right)\log\left(\max\left(e,\left\Vert \mu\right\Vert_{\infty}\right)+\left\Vert \mu \right\Vert_{C^{0,s}}\right).
\end{align*} 
\end{prop}
\begin{thm}
Assume that $S$ and $V$ are as in (\textbf{H1}) and let the initial data $\mu_{0}$ satisfy (\textbf{H2}).   Then, there exist a solution in the sense of Definition \ref{def of weak sol} to \eqref{mean field equation sec2}
with initial data $\mu_{0}$. Moreover, this solution satisfies $\mu \in C\left([0,T];L^{p}(\mathbb{R}^{d})\right)$ for all $1\leq p \leq \infty$ and $\mu \in L^{\infty}\left([0,T];W^{1,p}(\mathbb{R}^{d})\right)$ for all $1\leq p< \infty$. 
Furthermore, $\mu(t,\cdot)$ is compactly supported and its support satisfies 
\[
\mathrm{supp}(\mu(t,\cdot))\subset B(0,\overline{R})\,, \mbox{ for all } t\in[0,T],
\]
where $\overline{R}=\overline{R}(\left\Vert \mu_{0}\right\Vert _{L^\infty },R,T,\left\Vert S\right\Vert _{L^\infty})$. 
\label{thm final existence}
\end{thm}
\textit{Proof}. \textbf{Step 1. Propagation of $L^p$ norms}.  Let $\chi_{\varepsilon}$ be a standard mollifier and let $V_{\varepsilon} \coloneqq \chi_{\varepsilon}\star V $ and $\mathbf{A}_{\varepsilon}[\mu]\coloneqq \mathbb{J}\nabla V_{\varepsilon} \star \mu$.
Let $\mu_{\varepsilon}$ be the solution
of the mollified equation 

\begin{equation}
\partial_{t}\mu_{\varepsilon}-\mathrm{div}\left(\mu_{\varepsilon}\mathbf{A}_{\varepsilon}[\mu_{\varepsilon}]\right)=h\left[\mu_{\varepsilon}\right],\ \mu_{\varepsilon}(0,x)=\chi_{\varepsilon}\star\mu_{0}.\label{mollifiedeq}
\end{equation}
This solution is classical and is ensured  via Proposition \ref{Cauchy sequence} and Remark \ref{classical solutions}. We wish to quantify the growth, size of support and regularity of this solution. We compute the time derivative
of $\left\Vert \mu_{\varepsilon}(t,\cdot)\right\Vert _{p}^{p}$ . 
\begin{align*}
\frac{d}{dt}\left\Vert \mu_{\varepsilon}(t,\cdot)\right\Vert _{p}^{p}&=p\int_{\mathbb{R}^{d}}\mu_{\varepsilon}^{p-1}(t,x)\partial_{t}\mu_{\varepsilon}(t,x)dx\\
&=p\int_{\mathbb{R}^{d}}\mu_{\varepsilon}^{p-1}(t,x)\mathrm{div}\left(\mu_{\varepsilon}(t,x)\mathbf{A}_{\varepsilon}[\mu_{\varepsilon}](t,x)\right)dx+p\int_{\mathbb{R}^{d}}\mu_{\varepsilon}^{p-1}(t,x)h\left[\mu_{\varepsilon}\right](t,x)dx\coloneqq J_{1}+J_{2}.    
\end{align*}
To bound $J_{1}$, note that integration by parts yields 
\begin{align}
J_{1}&=p\int_{\mathbb{R}^{d}}\mu_{\varepsilon}^{p-1}(t,x)\nabla\mu_{\varepsilon}(t,x)\mathbf{A}_{\varepsilon}[\mu_{\varepsilon}](t,x)dx+ p\int_{\mathbb{R}^{d}}\mu_{\varepsilon}^{p}(t,x)\mathrm{div}\left(\mathbf{A}_{\varepsilon}[\mu_{\varepsilon}](t,x)\right)dx \notag\\
&=
(p-1)\int_{\mathbb{R}^{d}}\mu_{\varepsilon}^{p}(t,x)\mathrm{div}\left(\mathbb{J}\nabla V_{\varepsilon}\star\mu_{\varepsilon}(t,x)\right)dx.\label{eqJ1}
\end{align}
Note that $\mathrm{div}(\mathbb{J}\nabla V_{\varepsilon}\star\mu_{\varepsilon})=0$ if $\mathbb{J}$
is antisymmetric and  $\mathrm{div}(\mathbb{J}\nabla V_{\varepsilon}\star\mu_{\varepsilon})=-\chi_{\varepsilon}\star \mu_{\varepsilon}\leq0$ if $\mathbb{J}$ is the identity.  Thus, \eqref{eqJ1} proves that $J_{1}\leq0$. Furthermore we have 
\begin{equation*}
\left\vert J_{2}\right\vert \leq p\int_{\mathbb{R}^{d}}\mu_{\varepsilon}^{p}(t,x)\left\vert \mathbf{S}[\mu_{\varepsilon}]\right\vert(t,x)dx\leq p\left\Vert \mathbf{S}[\mu_{\varepsilon}](t,\cdot)\right\Vert _{\infty}\left\Vert \mu_{\varepsilon}(t,\cdot)\right\Vert _{p}^{p}\leq p\left\Vert S\right\Vert _{\infty}\left\Vert \mu_{\varepsilon}(t,\cdot)\right\Vert _{p}^{p}.
\end{equation*}
Therefore, we conclude
that 
\[
\frac{d}{dt}\left\Vert \mu_{\varepsilon}(t,\cdot)\right\Vert _{p}^{p}\leq p\left\Vert S\right\Vert _{\infty}\left\Vert \mu_{\varepsilon}(t,\cdot)\right\Vert _{p}^{p},
\]
which entails 
\[
\left\Vert \mu_{\varepsilon}(t,\cdot)\right\Vert _{p}^{p}\leq e^{p\left\Vert S\right\Vert _{\infty}t}\left\Vert \mu_{\varepsilon}(0,\cdot)\right\Vert _{p}^{p}=e^{p\left\Vert S\right\Vert _{\infty}t}\left\Vert \chi_{\varepsilon}\star\mu_{0}\right\Vert _{p}^{p}\leq e^{p\left\Vert S\right\Vert _{\infty}t}\left\Vert \mu_{0}\right\Vert _{p}^{p},
\]
or 

\[
\left\Vert \mu_{\varepsilon}(t,\cdot)\right\Vert _{p}\leq e^{\left\Vert S\right\Vert _{\infty}t}\left\Vert \mu_{0}\right\Vert _{p},\ 1\leq p<\infty.
\]
Since according to Proposition \ref{Cauchy sequence} $\mu_{\varepsilon}(t,\cdot)$ is compactly
supported we can pass to the limit as $p\rightarrow\infty$ in the
last inequality in order to deduce 
\begin{align}
\left\Vert \mu_{\varepsilon}(t,\cdot)\right\Vert _{p}\leq e^{\left\Vert S\right\Vert _{\infty}t}\left\Vert \mu_{0}\right\Vert _{p}\ 1\leq p\leq\infty. \label{propagation of L^p eq}   
\end{align}

\textbf{Step 2. Propagation of support}. Denote by $R_{\varepsilon}(t)$ the size of the support of $\mu_{\varepsilon}(t,\cdot)$, i.e. $R_{\varepsilon}(t)\coloneqq \left\vert \mathrm{supp}(\mu_{\varepsilon}(t,\cdot))\right\vert$. For each $\eta>0$ consider the function 
\begin{align*}
\varphi_{\eta}(t,x)\coloneqq \frac{1}{\eta+\mu_{\varepsilon}(t,x)}.     
\end{align*}
Multiplying \eqref{mollifiedeq} by $\varphi_{\eta}$ and integrating on $[0,t]\times \mathbb{R}^{d}$ we arrive at  the equation 
\begin{align}
\int_{\mathbb{R}^{d}}\mu_{\varepsilon}(t,x)\varphi_{\eta}(t,x)=&\int_{\mathbb{R}^{d}}\mu_{\varepsilon}(0,x)\varphi_{\eta}(0,x)+\int_{0}^{t}\int_{\mathbb{R}^{d}}\partial_{\tau}\varphi_{\eta}(\tau,x)\mu_{\varepsilon}(\tau,x)dxd\tau \notag\\
&+\int_{0}^{t}\int_{\mathbb{R}^{d}}\varphi_{\eta}(\tau,x)\mathrm{div}(\mu_{\varepsilon}(\tau,x)\mathbf{A}_{\varepsilon}[\mu_{\varepsilon}](\tau,x))dxd\tau \notag\\
&+\int_{0}^{t}\int_{\mathbb{R}^{d}}\varphi_{\eta}(\tau,x)h[\mu_{\varepsilon}](\tau,x)dxd\tau.
\label{eq tested}
\end{align}
We proceed by manipulating the inner integrals in the right-hand side of \eqref{eq tested}.
\begin{align}
\int_{\mathbb{R}^{d}}\partial_{\tau}\varphi_{\eta}(\tau,x)\mu_{\varepsilon}(\tau,x)dx=&-\int_{\mathbb{R}^{d}}\frac{\mu_{\varepsilon}(\tau,x)\partial_{\tau}\mu_{\varepsilon}(\tau,x)}{(\eta+\mu_{\varepsilon}(\tau,x))^{2}}dx\notag\\
=&-\int_{\mathbb{R}^{d}}\frac{\mu_{\varepsilon}(\tau,x)\mathrm{div}(\mu_{\varepsilon}(\tau,x)\mathbf{A}_{\varepsilon}[\mu_{\varepsilon}](\tau,x))}{(\eta+\mu_{\varepsilon}(\tau,x))^{2}}dx-\int_{\mathbb{R}^{d}}\frac{\mu_{\varepsilon}(\tau,x)h[\mu_{\varepsilon}](\tau,x)}{(\eta+\mu_{\varepsilon}(\tau,x))^{2}}dx \notag\\
=&-\int_{\mathbb{R}^{d}}\frac{\mu_{\varepsilon}(\tau,x)\nabla\mu_{\varepsilon}(\tau,x)\mathbf{A}_{\varepsilon}[\mu_{\varepsilon}](\tau,x)}{(\eta+\mu_{\varepsilon}(\tau,x))^{2}}dx \notag\\&-\int_{\mathbb{R}^{d}}\frac{\mu_{\varepsilon}^{2}(\tau,x)\mathrm{div}(\mathbf{A}_{\varepsilon}[\mu_{\varepsilon}](\tau,x))}{(\eta+\mu_{\varepsilon}(\tau,x))^{2}}dx
-\int_{\mathbb{R}^{d}}\frac{\mu_{\varepsilon}(\tau,x)h[\mu_{\varepsilon}](\tau,x)}{(\eta+\mu_{\varepsilon}(\tau,x))^{2}}dx. \label{first manipulation}
\end{align}
Integrating by parts the first integral in the right-hand side of \eqref{first manipulation} we get 
\begin{align}
\int_{\mathbb{R}^{d}}\frac{\mu_{\varepsilon}(\tau,x)\nabla\mu_{\varepsilon}(\tau,x)\mathbf{A}_{\varepsilon}[\mu_{\varepsilon}](\tau,x)}{(\eta+\mu_{\varepsilon}(\tau,x))^{2}}dx =&\,-\int_{\mathbb{R}^{d}}\mathbf{A}_{\varepsilon}[\mu_{\varepsilon}](\tau,x)\nabla\left(\frac{1}{\eta+\mu_{\varepsilon}(\tau,x)}\right)\mu_{\varepsilon}(\tau,x)dx\notag \\
=&\,\int_{\mathbb{R}^{d}}\mathrm{div}(\mathbf{A}_{\varepsilon}[\mu_{\varepsilon}])(\tau,x)\frac{\mu_{\varepsilon}(\tau,x)}{\eta+\mu_{\varepsilon}(\tau,x)}dx\notag \\
&\,+\int_{\mathbb{R}^{d}}\mathbf{A}_{\varepsilon}[\mu_{\varepsilon}](\tau,x)\nabla\log\left(\eta+\mu_{\varepsilon}(\tau,x)\right)dx.\label{identity for 1st int}
\end{align}
As for the second integral in \eqref{eq tested}, we observe the identity  
\begin{align}
\int_{\mathbb{R}^{d}}\varphi_{\eta}(\tau,x)\mathrm{div}(\mu_{\varepsilon}(\tau,x)\mathbf{A}_{\varepsilon}[\mu_{\varepsilon}](\tau,x))dx 
= &\,\int_{\mathbb{R}^{d}}\mathbf{A}_{\varepsilon}[\mu_{\varepsilon}](\tau,x)\nabla\log\left(\eta+\mu_{\varepsilon}(\tau,x)\right)dx\notag\\
&\,+\int_{\mathbb{R}^{d}}\mathrm{div}(\mathbf{A}_{\varepsilon}[\mu_{\varepsilon}](\tau,x))\frac{\mu_{\varepsilon}(\tau,x)}{\mu_{\varepsilon}(\tau,x)+\eta}dx. \label{second manipulation}    
\end{align}
The combination of equations \eqref{first manipulation}, \eqref{identity for 1st int} and \eqref{second manipulation} yields 
\begin{align*}
&\int_{\mathbb{R}^{d}}\partial_{\tau}\varphi_{\eta}(\tau,x)\mu_{\varepsilon}(\tau,x)dx+\int_{\mathbb{R}^{d}}\varphi_{\eta}(\tau,x)\mathrm{div}(\mu_{\varepsilon}(\tau,x)\mathbf{A}_{\varepsilon}[\mu_{\varepsilon}](\tau,x))dx+\int_{\mathbb{R}^{d}}\varphi_{\eta}(\tau,x)h[\mu_{\varepsilon}](\tau,x)dx\\
&=-\int_{\mathbb{R}^{d}}\frac{\mu_{\varepsilon}^{2}(\tau,x)\mathrm{div}(\mathbf{A}_{\varepsilon}[\mu_{\varepsilon}](\tau,x))}{(\eta+\mu_{\varepsilon}(\tau,x))^{2}}dx-\int_{\mathbb{R}^{d}}\frac{\mu_{\varepsilon}(\tau,x)h[\mu_{\varepsilon}](\tau,x)}{(\eta+\mu_{\varepsilon}(\tau,x))^{2}}dx+\int_{\mathbb{R}^{d}}\varphi_{\eta}(\tau,x)h[\mu_{\varepsilon}](\tau,x)dx\\
&\coloneqq I_{1}+I_{2}+I_{3}.
\end{align*}
Using that $\mathrm{div}(\mathbf{A}_{\varepsilon}[\mu_{\varepsilon}])=-\chi_{\varepsilon}\star\mu_{\varepsilon}$ if $\mathbb{J}=\mathrm{Id}$ and $\mathrm{div}(\mathbf{A}_{\varepsilon}[\mu_{\varepsilon}])=0$ if $\mathbb{J}$ is antisymmetric we see that 
\begin{align}
\left \vert I_{1} \right\vert \leq\left\Vert \chi_{\varepsilon}\star\mu_{\varepsilon}\right\Vert _{\infty}\int_{\mathbb{R}^{d}}\mu_{\varepsilon}^{2}(\tau,x)\varphi_{\eta}^{2}(\tau,x)dx&\leq\underset{s\in[0,T]}{\sup}\left\Vert \mu_{\varepsilon}(s,\cdot)\right\Vert _{\infty}\int_{\mathbb{R}^{d}}\mu_{\varepsilon}^{2}(\tau,x)\varphi_{\eta}^{2}(\tau,x)dx \notag\\
&\leq e^{T\left\Vert S\right\Vert_{\infty}}\left\Vert \mu_{0}\right\Vert_{\infty}\int_{\mathbb{R}^{d}}\mu_{\varepsilon}^{2}(\tau,x)\varphi_{\eta}^{2}(\tau,x),   \label{ESTIMATE for I1}
\end{align}
where in the last inequality we used \eqref{propagation of L^p eq}. 
Furthermore, we have the inequalities  
\begin{align}
\left\vert I_{2} \right\vert \leq\int_{\mathbb{R}^{d}}\frac{\mu_{\varepsilon}^{2}(\tau,x)\left|S\star\mu_{\varepsilon}\right|(\tau,x)}{(\eta+\mu_{\varepsilon}(\tau,x))^{2}}dx\leq\left\Vert S\right\Vert _{\infty}\int_{\mathbb{R}^{d}}\mu_{\varepsilon}^{2}(\tau,x)\varphi_{\eta}^{2}(\tau,x)dx
\label{ESTIMATE for I2}
\end{align}
and 
\begin{align}
\left\vert I_{3} \right\vert\leq\int_{\mathbb{R}^{d}}\varphi_{\eta}(\tau,x)\mu_{\varepsilon}(\tau,x)\left|S\star\mu_{\varepsilon}\right|(\tau,x)dx\leq \left\Vert S\right\Vert _{\infty}\int_{\mathbb{R}^{d}}\mu_{\varepsilon}(\tau,x)\varphi_{\eta}(\tau,x)dx.
\label{ESTIMATE for I3}
\end{align}
Putting $\overline{M}=\max\left\{ 2\left\Vert S\right\Vert _{\infty},e^{\left\Vert S\right\Vert _{\infty}T}\left\Vert \mu_{0}\right\Vert _{\infty}\right\} $ and substituting \eqref{ESTIMATE for I1}, \eqref{ESTIMATE for I2} and \eqref{ESTIMATE for I3} in \eqref{eq tested} 
we find that 
\begin{align*}
\int_{\mathbb{R}^{d}}\mu_{\varepsilon}(t,x)\varphi_{\eta}(t,x)\leq&\int_{\mathbb{R}^{d}}\mu_{\varepsilon}(0,x)\varphi_{\eta}(0,x)\\
&+2\overline{M}\left(\int_{0}^{t}\int_{\mathbb{R}^{d}}\varphi_{\eta}^{2}(\tau,x)\mu_{\varepsilon}^{2}(\tau,x)dxd\tau+\int_{0}^{t}\int_{\mathbb{R}^{d}}\varphi_{\eta}(\tau,x)\mu_{\varepsilon}(\tau,x)dxd\tau \right).     
\end{align*}
Since $\varphi_{\eta}(t,x)\mu_{\varepsilon}(t,x)\underset{\eta \rightarrow 0}{\rightarrow} \mathbf{1}_{\left\vert \cdot \right\vert\leq R_{\varepsilon}(t)}(x)$, passing to the limit in the last inequality as $\eta \rightarrow 0$ we obtain 
\begin{align*}
R_{\varepsilon}^{d}(t)\leq R_{\varepsilon}^{d}(0)+2\overline{M}\int_{0}^{t}R_{\varepsilon}^{d}(\tau)d\tau.
\end{align*}
Hence, by Gr\"onwall's inequality we deduce that 
$
R_{\varepsilon}(t)\leq R_{\varepsilon}(0)e^{\frac{2\overline{M}t}{d}}\leq (R+\varepsilon)e^{\frac{2\overline{M}T}{d}}.    
$
To conclude, we have proved 
\begin{align}
R_{\varepsilon}(t)\leq \overline{R}, \ t \in [0,T] \label{compact support final est}  \end{align}
for some $\overline{R}=\overline{R}(\left\Vert \mu_{0}\right\Vert _{L^\infty },R,T,\left\Vert S\right\Vert _{L^\infty})$.

\

\textbf{Step 3. Propagation of Sobolev norms}. Expanding the divergence in \eqref{mollifiedeq}, we get
\begin{align*}
\partial_{t}\mu_{\varepsilon}-\nabla\mu_{\varepsilon}\mathbb{J}\nabla V_{\varepsilon}\star\mu_{\varepsilon}-\mu_{\varepsilon}\mathrm{div}(\mathbb{J}\nabla V_{\varepsilon}\star\mu_{\varepsilon})=h\left[\mu_{\varepsilon}\right].    
\end{align*}
Taking its $j$-th derivative, we obtain 
\begin{equation}
\partial_{t}\partial_{x_{j}}\mu_{\varepsilon}-\nabla\partial_{x_{j}}\mu_{\varepsilon}\mathbb{J}\nabla V_{\varepsilon}\star\mu_{\varepsilon}-\partial_{x_{j}}(\mathbb{J}\nabla V_{\varepsilon}\star\mu_{\varepsilon})\nabla\mu_{\varepsilon}-\partial_{x_{j}}\mu_{\varepsilon}\mathrm{div}(\mathbb{J}\nabla V_{\varepsilon}\star\mu_{\varepsilon})-\mu_{\varepsilon}\mathrm{div}(\mathbb{J}\nabla V_{\varepsilon}\star\partial_{x_{j}}\mu_{\varepsilon})= R_{j,{\varepsilon}}(t,x),\label{eq:-40}
\end{equation}
with $R_{j,{\varepsilon}}(t,x)\coloneqq\partial_{x_{j}}h\left[\mu_{\varepsilon}\right]$.
Set $v_{\varepsilon}(t,x)\coloneqq\mathbb{J}\nabla V_{\varepsilon}\star\mu_{\varepsilon}$ and $\omega_{j,\varepsilon}\coloneqq\partial_{x_{j}}\mu_{\varepsilon}$,
so that the equation \eqref{eq:-40} reads 
\begin{align*}
\partial_{t}\omega_{j,{\varepsilon}}-\nabla\omega_{j,{\varepsilon}}v_{\varepsilon}-\underset{k}{\sum}\partial_{x_{j}}v^{k}_{\varepsilon}\omega_{k,\varepsilon}-\omega_{j,\varepsilon}\mathrm{div}(v_{\varepsilon})-\mu_{\varepsilon}\mathrm{div}(\partial_{x_{j}}v_{\varepsilon})=R_{j,{\varepsilon}}(t,x).    
\end{align*}
We compute the time derivative of $\stackrel[j=1]{d}{\sum}\left\Vert \omega_{j,\varepsilon}(t,\cdot)\right\Vert _{p}^{p}$ as follows:  
\begin{align*}
\frac{d}{dt}\underset{j}{\sum}\left\Vert \omega_{j,\varepsilon}(t,\cdot)\right\Vert _{p}^{p}=&p\underset{j}{\sum}\int_{\mathbb{R}^{d}}\left|\omega_{j,\varepsilon}\right|^{p-1}(t,x)\partial_{t}\left|\omega_{j,\varepsilon}\right|(t,x)dx\\
=&p\underset{j}{\sum}\int_{\mathbb{R}^{d}}\left|\omega_{j,\varepsilon}\right|^{p-1}(t,x)\mathrm{sgn}(\omega_{j,\varepsilon}(t,x))\\&\times\left(\nabla\omega_{j,\varepsilon}v_{\varepsilon}+\underset{k}{\sum}(\partial_{x_{j}}v_{\varepsilon}^{k})\omega_{k,\varepsilon}+\omega_{j,\varepsilon}\mathrm{div}(v_{\varepsilon})+\mu_{\varepsilon}\mathrm{div}(\partial_{x_{j}}v_{\varepsilon})\right)(t,x)dx \\
&+p\underset{j}{\sum}\int_{\mathbb{R}^{d}}\left|\omega_{j,\varepsilon}\right|^{p-1}(t,x)\mathrm{sgn}(\omega_{j,\varepsilon}(t,x))R_{j,\varepsilon}(t,x)dx.
\end{align*}
The right-hand side of the last equation is bounded by 
\begin{align*}
&\underset{j}{\sum}\int_{\mathbb{R}^{d}}\left|\omega_{j,\varepsilon}\right|^{p}(t,x)\left|\mathrm{div}(v_\varepsilon)\right|(t,x)dx+p\underset{j,k}{\sum}\int_{\mathbb{R}^{d}}\left|\omega_{j,\varepsilon}\right|^{p-1}(t,x)\left|\partial_{x_{j}}v_\varepsilon^{k}\right|(t,x)\left|\omega_{k,\varepsilon}\right|(t,x)dx\\
&+p\underset{j}{\sum}\int_{\mathbb{R}^{d}}\left|\omega_{j,\varepsilon}\right|^{p}(t,x)\left|\mathrm{div}(v_\varepsilon)\right|(t,x)dx+p\underset{j}{\sum}\int_{\mathbb{R}^{d}}\left|\omega_{j,\varepsilon}\right|^{p-1}(t,x)\left\vert \mu_\varepsilon \partial_{x_{j}}\mathrm{div} (v_\varepsilon)\right\vert (t,x)\\
&+p\underset{j}{\sum}\int_{\mathbb{R}^{d}}\left|\omega_{j,\varepsilon}\right|^{p-1}(t,x)\left|R_{j,\varepsilon}\right|(t,x)dx\\ \leq &(p+1)\underset{j}{\sum}\int_{\mathbb{R}^{d}}\left|\omega_{j,\varepsilon}\right|^{p}\left|\mathrm{div}(v_\varepsilon)\right|(t,x)dx+p\underset{j,k}{\sum}\int_{\mathbb{R}^{d}}\left|\omega_{j,\varepsilon}\right|^{p-1}(t,x)\left|\partial_{x_{j}}v_\varepsilon^{k}\right|(t,x)\left|\omega_{k,\varepsilon}\right|(t,x)dx\\&+pC(T,\left\Vert S \right\Vert_{\infty}, \left\Vert \mu_{0} \right\Vert_{\infty})\underset{j}{\sum}\int_{\mathbb{R}^{d}}\left|\omega_{j,\varepsilon}\right|^{p}(t,x)dx +p\underset{j}{\sum}\int_{\mathbb{R}^{d}}\left|\omega_{j,\varepsilon}\right|^{p-1}(t,x)\left|R_{j,\varepsilon}\right|(t,x)dx.    
\end{align*}
In the last inequality we invoked  \eqref{propagation of L^p eq}, according to which there holds the estimate
\begin{align*}
\left\Vert \mathrm{div}(v_{\varepsilon})(t,\cdot)\right\Vert _{\infty}\leq \left\Vert \mathrm{div}(\mathbb{J}\nabla V\star \mu_{\varepsilon})(t,\cdot)\right\Vert _{\infty} \leq\left\Vert \mu_{\varepsilon}(t,\cdot)\right\Vert _{\infty}\leq C   
\end{align*}
and thus
\begin{equation}
\int_{\mathbb{R}^{d}}\left|\omega_{j,\varepsilon}\right|^{p}(t,x)\left|\mathrm{div}(v_\varepsilon)\right|(t,x)dx\leq C\int_{\mathbb{R}^{d}}\left|\omega_{j,\varepsilon}\right|^{p}(t,x)dx, \label{1st ineq for I}
\end{equation}
where $C=C\left(\left\Vert \mu_{0}\right\Vert _{\infty},\left\Vert S\right\Vert _{\infty},T\right)$.
Invoking Proposition \ref{near boundedness},  \eqref{propagation of L^p eq} and Young's inequality for products yields 
\begin{align*}
p&\int_{\mathbb{R}^{d}}\left|\omega_{j,\varepsilon}\right|^{p-1}(t,x)\left\vert \partial_{x_{j}}v_\varepsilon^{k} \right\vert(t,x)\left\vert \omega_{k,\varepsilon} \right\vert(t,x)dx\\&\leq \left\Vert \partial_{x_{j}}v_\varepsilon^{k}(t,\cdot)\right\Vert_{\infty}p\int_{\mathbb{R}^{d}}\left|\omega_{j,\varepsilon}\right|^{p-1}(t,x)\left|\omega_{k,\varepsilon} \right|(t,x)dx\\
&\leq C\left(d,\left\Vert \mu_{\varepsilon}(t,\cdot)\right\Vert_{\infty}\right)\log\left(\max\left(e,\left\Vert \mu_{\varepsilon}(t,\cdot)\right\Vert_{\infty}\right)+\left\Vert \mu_{\varepsilon}(t,\cdot) \right\Vert_{C^{0,s}}\right)\\
&\quad\times p\left(\frac{p-1}{p}\int_{\mathbb{R}^{d}}\left|\omega_{j,\varepsilon}(t,x)\right|^{p}dx+\frac{1}{p}\int_{\mathbb{R}^{d}}\left|\omega_{k,\varepsilon}(t,x)\right|^{p}dx\right)\\
&\leq C\log\left(C+\left\Vert \mu_{\varepsilon}(t,\cdot) \right\Vert_{C^{0,s}}^{p}\right)\stackrel[j=1]{d}{\sum}\left\Vert \omega_{j,\varepsilon}(t,\cdot)\right\Vert _{p}^{p},
\end{align*}
where $C=C\left(p,\left\Vert S \right\Vert_{\infty},T,d,\left\Vert \mu_{0} \right\Vert_{\infty}\right)>e$. Choosing $s=1-\frac{d}{p}$ we have by Morrey's inequality 
\begin{align*}
\left\Vert \mu_{\varepsilon}(t,\cdot) \right\Vert_{C^{0,s}}^{p}\lesssim \left\Vert \mu_{\varepsilon}(t,\cdot) \right\Vert_{W^{1,p}(\mathbb{R}^d)}^{p}\leq  C\left(T,\left\Vert S \right\Vert_{\infty},\left\Vert \mu_{0}\right\Vert_{\infty},d,p\right)\left(1+\stackrel[j=1]{d}\sum \left\Vert \omega_{j,\varepsilon}
(t,\cdot)\right\Vert_{p}^p\right),    
\end{align*}
hence we have proved 
\begin{equation}
p\int_{\mathbb{R}^{d}}\left|\omega_{j,\varepsilon}\right|^{p-1}(t,x)\left\vert \partial_{x_{j}}v_\varepsilon^{k} \right\vert(t,x)\left\vert \omega_{k,\varepsilon} \right\vert(t,x)dx\leq C\log\left(C+C\stackrel[j=1]{d}\sum \left\Vert \omega_{j,\varepsilon}
(t,\cdot)\right\Vert_{p}^p\right)\stackrel[j=1]{d}\sum \left\Vert \omega_{j,\varepsilon}
(t,\cdot)\right\Vert_{p}^p.    \label{2nd ineq for I} 
\end{equation}
Thanks to H{\"o}lder's inequality, it holds that  
\begin{align}
&p\int_{\mathbb{R}^{d}}\left|\omega_{j,\varepsilon}\right|^{p-1}(t,x)\left|R_{j,\varepsilon}\right|(t,x)dx \notag\\
&=p\int_{\mathbb{R}^{d}}\left|\omega_{j,\varepsilon}\right|^{p-1}(t,x)\left|\omega_{j,\varepsilon}\right|(t,x)\left|\mathbf{S}[\mu_{\varepsilon}]\right|(t,x)dx+p\int_{\mathbb{R}^{d}}\left|\omega_{j,\varepsilon}\right|^{p-1}(t,x)\mu_{\varepsilon}(t,x)\left|\partial_{x_{j}}\mathbf{S}[\mu_{\varepsilon}]\right|(t,x)dx
\notag\\
&\leq p\left\Vert S\right\Vert _{\infty}\int_{\mathbb{R}^{d}}\left|\omega_{j,\varepsilon}\right|^{p}(t,x)dx+p\left(\int_{\mathbb{R}^{d}}\left|\omega_{j,\varepsilon}\right|^{p}(t,x)dx\right)^{\frac{p-1}{p}}\left(\int_{\mathbb{R}^{d}}\left|\mu_{\varepsilon}\partial_{x_{j}}S\star \mu_{\varepsilon}\right|^{p}(t,x)dx\right)^{\frac{1}{p}} \notag\\
&\leq p\left\Vert S\right\Vert _{\infty}\int_{\mathbb{R}^{d}}\left|\omega_{j,\varepsilon}\right|^{p}(t,x)dx+C(p,\left\Vert \mu_{0}\right\Vert_{\infty},T,\left\Vert S\right\Vert_{W^{1,\infty}})\left(\int_{\mathbb{R}^{d}}\left|\omega_{j,\varepsilon}\right|^{p}(t,x)dx\right)^{\frac{p-1}{p}} \notag\\
&\leq C\left( \left\Vert \omega_{j,\varepsilon}(t,\cdot)\right\Vert _{p}^{p}+1\right),\label{eq:-43}
\end{align}
where $C=C\left(T,\left\Vert S\right\Vert _{W^{1,\infty}},\left\Vert \mu_{0}\right\Vert _{\infty},p\right)$. 
Put 
\begin{align*}
    \Omega(t)\coloneqq \sum_{j=1}^{d} \left\Vert \omega_{j,\varepsilon}
(t,\cdot)\right\Vert_{p}^p. 
\end{align*}
Gathering  \eqref{1st ineq for I}, \eqref{2nd ineq for I} 
and \eqref{eq:-43}, we arrive at the inequality 
\begin{align*}
 \frac{d}{dt}\Omega(t)\leq  C\log\left( 1+\Omega(t)\right)\left(1+\Omega(t)\right),  \qquad
\mbox{ or equivalently } \qquad
 \frac{d}{dt} \log\log\left(1+\Omega(t)\right)\leq C, 
 \end{align*}
for a suitable constant $C=C\left(\left\Vert \mu_{0}\right\Vert _{\infty},\left\Vert S\right\Vert _{W^{1,\infty}},T,d,p\right)$.
Solving the above inequality yields 
\begin{align}
\left\Vert \nabla\mu_{\varepsilon}(t,\cdot)\right\Vert_{p}^{p}=\Omega(t)\leq C\left(1+\Omega(0)\right)^{e^{Ct}}\leq C\left(1+\left\Vert \nabla \mu_{0} \right\Vert_{p}^{p}\right)^{e^{Ct}}, \ 1\leq p<\infty.   \label{propagation of Sobolev eq}
\end{align}
\textbf{Step 4. Compactness and extraction of a solution}. We start by showing that $\mu_{\varepsilon}\in C([0,T];L^{p}(\mathbb{R}^{d}))$ uniformly in $\varepsilon$.    By  \eqref{mollifiedeq} we have 
\begin{align*}
\left\Vert \mu_{\varepsilon}(t,\cdot)-\mu_{\varepsilon}(s,\cdot)\right\Vert_{p}\leq \left\vert t-s\right\vert\underset{\tau \in [0,T]}{\sup}\left(\left\Vert \mathrm{div}(\mu_{\varepsilon}(\tau,\cdot)\mathbf{A}_{\varepsilon}[\mu_{\varepsilon}](\tau,\cdot))\right\Vert_{p}+\left\Vert h[\mu_{\varepsilon}](\tau,\cdot) \right\Vert_{p}\right).      
\end{align*}
We expand 
\begin{align*}
\mathrm{div}(\mu_{\varepsilon}(\tau,\cdot)\mathbf{A}_{\varepsilon}[\mu_{\varepsilon}](\tau,\cdot))=\nabla \mu_{\varepsilon}(\tau,\cdot)\mathbf{A}_{\varepsilon}[\mu_{\varepsilon}](\tau,\cdot)+\mu_{\varepsilon}(\tau,\cdot)\mathrm{div}(\mathbf{A}_{\varepsilon}[\mu_{\varepsilon}](\tau,\cdot)).     
\end{align*}
In what follows  $C=C\left(\left\Vert \mu_{0}\right\Vert _{W^{1,\infty}},\left\Vert S\right\Vert _{W^{1,\infty}},T,d,p\right)$ is a constant which may change from line to line. According to  \eqref{propagation of Sobolev eq} there holds the estimate  
\begin{align*}
\left\Vert \nabla \mu_{\varepsilon}(\tau,\cdot)\right\Vert_{p}\leq C      
\end{align*}
while according to inequality \eqref{propagation of L^p eq} we have 
\begin{align}
\left\Vert \mathbf{A}_{\varepsilon}[\mu_{\varepsilon}](\tau,\cdot)\right\Vert_{\infty}\leq \left\Vert \mathbb{J}\nabla V\star \mu_{\varepsilon}(\tau,\cdot)\right\Vert_{\infty}\leq C. \label{est on mollified A}    
\end{align}
Thus, we get the estimate  
\begin{align*}
\left\Vert \nabla \mu_{\varepsilon}(\tau,\cdot)\mathbf{A}_{\varepsilon}[\mu_{\varepsilon}](\tau,\cdot) \right\Vert_{p}\leq \left\Vert \nabla \mu_{\varepsilon}(\tau,\cdot) \right\Vert_{p}\left\Vert \mathbf{A}_{\varepsilon}[\mu_{\varepsilon}](\tau,\cdot)\right\Vert_{\infty}\leq C.       
\end{align*}
From the same consideration  
\begin{align*}
\left\Vert \mu _{\varepsilon}(\tau,\cdot))\right\Vert_{p} \leq C    
\end{align*}
and 
\begin{align*}
\left\Vert \mathrm{div}(\mathbf{A}_{\varepsilon}[\mu_{\varepsilon}](\tau,\cdot))\right\Vert_{\infty}\leq \left\Vert \chi_{\varepsilon}\star \mu_{\varepsilon}(\tau,\cdot)\right\Vert_{\infty}\leq \left\Vert \mu_{\varepsilon}(\tau,\cdot)\right\Vert_{\infty}\leq C.         
\end{align*}
Hence 
\begin{align}
\underset{\tau \in [0,T]}{\sup}\left\Vert \mathrm{div}(\mu_{\varepsilon}(\tau,\cdot)\mathbf{A}_{\varepsilon}[\mu_{\varepsilon}](\tau,\cdot))\right\Vert_{p}\leq C. \label{div bound}   
\end{align}
In addition, notice the bound 
\begin{align}
 \underset{\tau \in [0,T]}{\sup} \left\Vert h[\mu_{\varepsilon}](\tau,\cdot)\right\Vert_{p}\leq \left\Vert S \right\Vert_{\infty }\underset{\tau \in [0,T]}{\sup} \left\Vert \mu_{\varepsilon}(\tau,\cdot)\right\Vert_{p}\leq C.     \label{h[mu] bound}
\end{align}
Therefore, from  \eqref{div bound} and \eqref{h[mu] bound} we obtain 
\begin{align}
\left\Vert \mu_{\varepsilon}(t,\cdot)-\mu_{\varepsilon}(s,\cdot)\right\Vert_{p}\leq C\left\vert t-s\right\vert. \label{compactness}     
\end{align}
By the theorem of Arzela-Ascoli and \eqref{compactness} we may extract a
subsequence $\mu_{\varepsilon_{m}}$ and some 
$$\mu \in C([0,T];L^{p}(\mathbb{R}^{d}))$$
such that 
\begin{align*}
\underset{t \in [0,T]}{\sup}\left\Vert \mu_{\varepsilon_{m}}(t,\cdot)-\mu(t,\cdot)\right\Vert_{p} \underset{m \rightarrow \infty}{\rightarrow}0.    
\end{align*}
By \eqref{compact support final est}, $\mu(t,\cdot)$ has compact support and $\mathrm{supp}(\mu(t,\cdot))\subset \overline{R}$ and by \eqref{propagation of Sobolev eq} and the Banach-Alaoglu theorem $\mu \in L^{\infty}([0,T];W^{1,p}(\mathbb{R}^{d}))$.We proceed by showing that $\mu$ is a weak solution to the equation \eqref{mean field equation sec2}. To achieve this it suffices to check that 
\begin{align}
\underset{t \in [0,T]}{\sup}\left\Vert \mathbf{A}_{\varepsilon_{m}}[\mu_{\varepsilon_{m}}](t,\cdot)\mu_{\varepsilon_{m}}(t,\cdot)- \mathbf{A}[\mu](t,\cdot)\mu(t,\cdot)\right\Vert_{1} \underset{m\rightarrow \infty}{\rightarrow} 0, \ \underset{t\in [0,T]}{\sup}\left\Vert h[\mu_{\varepsilon_{m}}](t,\cdot)-h[\mu](t,\cdot) \right\Vert_{1} \underset{m\rightarrow \infty}{\rightarrow} 0.
\label{limit of A and h}
\end{align}
By the triangle inequality  
\begin{align*}
\left\Vert \mathbf{A}_{\varepsilon_{m}}[\mu_{\varepsilon_{m}}](t,\cdot)\mu_{\varepsilon_{m}}(t,\cdot)- \mathbf{A}[\mu](t,\cdot)\mu(t,\cdot)\right\Vert_{1}\leq& 
\left\Vert \mathbf{A}_{\varepsilon_{m}}[\mu_{\varepsilon_{m}}](t,\cdot)(\mu_{\varepsilon_{m}}(t,\cdot)-\mu(t,\cdot))\right\Vert_{1}\\
&+\left\Vert \mu(t,\cdot)(\mathbf{A}_{\varepsilon_{m}}[\mu_{\varepsilon_{m}}](t,\cdot)-\mathbf{A}[\mu](t,\cdot))\right\Vert_{1}\coloneqq A_{1}+A_{2}.  
\end{align*}
By \eqref{est on mollified A} we see that 
\begin{align}
A_{1}\leq C\left\Vert \mu_{\varepsilon_{m}}(t,\cdot)-\mu(t,\cdot)\right\Vert_{1}. \label{A1 est}  
\end{align}
Furthermore, we have
\begin{align*}
A_{2}\leq& \left\Vert \mu(t,\cdot)(\mathbf{A}_{\varepsilon_{m}}[\mu_{\varepsilon_{m}}](t,\cdot)-\mathbf{A}_{\varepsilon_{m}}[\mu](t,\cdot))\right\Vert_{1}+
\left\Vert \mu(t,\cdot)(\mathbf{A}_{\varepsilon_{m}}[\mu](t,\cdot)-\mathbf{A}[\mu](t,\cdot) )\right\Vert_{1}\\
\leq& C\left(\left\Vert \mathbf{A}_{\varepsilon_{m}}[\mu_{_{\varepsilon_{m}}}](t,\cdot)-\mathbf{A}_{\varepsilon_{m}}[\mu](t,\cdot)\right\Vert_{2}+\left\Vert \mathbf{A}_{\varepsilon_{m}}[\mu](t,\cdot)-\mathbf{A}[\mu](t,\cdot) \right\Vert_{2}\right)\\
\leq& C\left(\left\Vert \mathbb{J}\nabla V\star(\mu_{\varepsilon_{m}}-\mu)(t,\cdot)\right\Vert_{2}+
\left\Vert \chi_{\varepsilon_{m}}\star\mathbf{A}[\mu](t,\cdot)-\mathbf{A}[\mu](t,\cdot)\right\Vert_{2}\right).  
\end{align*}
Clearly by passing to Fourier we have
\begin{align*}
\left\Vert \mathbb{J}\nabla V\star(\mu_{\varepsilon_{m}}-\mu)(t,\cdot)\right\Vert_{2}\leq C\left\Vert \mu_{\varepsilon_{m}}(t,\cdot)-\mu(t,\cdot)\right\Vert_{\dot{H}^{-1}(\mathbb{R}^{d})} \leq C\left\Vert \mu_{\varepsilon_{m}}(t,\cdot)-\mu(t,\cdot)\right\Vert_{p}.    
\end{align*}
In addition, by  \eqref{propagation of Sobolev eq} 
one has 
\begin{align*}
\left\Vert \chi_{\varepsilon_{m}}\star\mathbf{A}[\mu](t,\cdot)-\mathbf{A}[\mu](t,\cdot) \right\Vert_{2}&\leq 
C \left\Vert \chi_{\varepsilon_{m}}\star\mu(t,\cdot)-\mu(t,\cdot)\right\Vert_{\dot{H}^{-1}(\mathbb{R}^{d})}\\ 
&\leq C\left\Vert \chi_{\varepsilon_{m}}\star \mu(t,\cdot)-\mu(t,\cdot)\right\Vert_{p}\leq C\varepsilon_{m}\left\Vert \nabla \mu(t,\cdot)\right\Vert_{p}\leq C\varepsilon_{m},     
\end{align*}
from which we get 
\begin{align}
A_{2}\leq C\left(\left\Vert \mu_{\varepsilon}(t,\cdot)-\mu(t,\cdot)\right\Vert_{p}+\varepsilon_{m}\right) \label{A2 est}.    
\end{align}
Inequalities \eqref{A1 est} and \eqref{A2 est} entail the first convergence in \eqref{limit of A and h}. 
In addition, we estimate 
\begin{align*}
\left\Vert h[\mu_{\varepsilon_{m}}](t,\cdot)-h[\mu](t,\cdot)\right\Vert_{1}&\leq \left\Vert (\mu_{\varepsilon_{m}}-\mu)(t,\cdot)S\star \mu_{\varepsilon_{m}}(t,\cdot)\right\Vert_{1}+ \left\Vert \mu(t,\cdot)(S\star \mu_{\varepsilon_{m}}-S\star\mu)(t,\cdot)\right\Vert_{1}\\
&\leq \left\Vert S\right\Vert_{\infty}\left\Vert \mu_{\varepsilon_{m}}(t,\cdot)-\mu(t,\cdot) \right\Vert_{1}+\left\Vert S \right\Vert_{\infty} 
 \left\Vert\mu_{\varepsilon_{m}}(t,\cdot)-\mu(t,\cdot)\right\Vert_{1}\\
 &=2\left\Vert S \right\Vert_{\infty} \left\Vert\mu_{\varepsilon_{m}}(t,\cdot)-\mu(t,\cdot)\right\Vert_{1},    
\end{align*}
which establishes the second convergence in \eqref{limit of A and h}, thereby proving that $\mu$ is a weak solution in the sense of Definition \ref{def of weak sol}.    
\begin{flushright}
$\square$
\par\end{flushright}

\subsection{Uniqueness }

To prove uniqueness, we study the evolution of the interaction energy
\begin{equation*}
\mathcal{E}(t)=\int_{\mathbb{R}^{d}}\left(\mu_{1}-\mu_{2}\right)(t,x)(V\star\left(\mu_{1}-\mu_{2}\right))(t,x)dx=\left\Vert \mu_{1}(t,\cdot)-\mu_{2}(t,\cdot)\right\Vert _{\dot{H}^{-1}(\mathbb{R}^{d})}^2.
\end{equation*} 
We start by proving that $h[\mu]$ is Lipschitz with respect to the $\dot{H}^{-1}(\mathbb{R}^{d})$ norm. In fact we have the following more general Lemma, from which the Lipschitz continuity of $h[\mu]$ with respect to $\dot{H}^{-1}(\mathbb{R}^{d})$ will follow. In the next lemma, we show some estimates involving a given kernel $W$.  The next lemma will eventually be used for either the Coulomb interaction $V$ or some regularized kernels thereof in the next sections.

\begin{lem}
\label{h Lip estimate } Let $S$ be as in (\textbf{H1}). Suppose further that 

$\bullet$ $\mu,\nu\in\mathcal{S}'(\mathbb{R}^{d})$ are tempered
distributions such that $\widehat{\mu},\widehat{\nu}\in L^{\infty}(\mathbb{R}^{d})$. 

$\bullet$ $W\in L^{1}_{\mathrm{loc}}(\mathbb{R}^{d})$ is such that 
\[\int_{\mathbb{R}^{d}}\left\vert W \right\vert(x)\left|\widehat{\mu-\nu}\right|^{2}(x)dx<\infty.\]

$\bullet$ There is some $c>0$ such that for all $x\in\mathbb{R}^{d}$ it holds that $\left|\widehat{S}\right|\star\left\vert W\right\vert(x)\leq c\left\vert W\right\vert(x).$

\noindent Then, it holds that 
\[
\int_{\mathbb{R}^{d}}\left\vert W \right\vert(x)\left|\widehat{h\left[\mu\right]-h\left[\nu\right]}\right|^{2}(x)dx\leq C\int_{\mathbb{R}^{d}}\left\vert W\right\vert(x)\left|\widehat{\mu-\nu}\right|^{2}(x)dx
\]
where $C=C\left(c,\left\Vert \widehat{S}\right\Vert _{1},\left\Vert \widehat{\mu}\right\Vert _{\infty},\left\Vert \widehat{\nu}\right\Vert _{\infty}\right).$
\end{lem}
\textit{Proof}. We have 
\begin{align}
I:=\int_{\mathbb{R}^{d}}\left\vert W \right\vert(x)\left|\widehat{h\left[\mu\right]-h\left[\nu\right]}\right|^{2}(x)dx=&\int_{\mathbb{R}^{d}}\left\vert W \right\vert(x)(\widehat{\mu S\star\mu-\nu S\star\nu})(x)\left(\widehat{h\left[\mu\right]-h\left[\nu\right]}\right)(x)dx\nonumber\\
=&\int_{\mathbb{R}^{d}}\left\vert W \right\vert(x)(\widehat{(\mu-\nu)(S\star\nu)})(x)\left(\widehat{h\left[\mu\right]-h\left[\nu\right]}\right)(x)dx\nonumber\\
&+\int_{\mathbb{R}^{d}}\left\vert W \right\vert(x)(\widehat{\mu(S\star\mu-S\star\nu)})(x)\left(\widehat{h\left[\mu\right]-h\left[\nu\right]}\right)(x)dx\nonumber\\
:=& I_1+I_2.    
\label{equation}
\end{align}
The first integral in the right-hand side of \eqref{equation} is controlled as follows
\begin{align}
I_1=\int_{\mathbb{R}^{d}}\left\vert W \right\vert(x)&((\widehat{\mu-\nu)(S\star\nu)})(x)\left(\widehat{h\left[\mu\right]-h\left[\nu\right]}\right)(x)dx\nonumber\\&=\int_{\mathbb{R}^{d}}\left\vert W \right\vert(x)\widehat{(\mu-\nu})\star(\widehat{S}\widehat{\nu})(x)\left(\widehat{h\left[\mu\right]-h\left[\nu\right]}\right)(x)dx\nonumber\\
&=\int_{\mathbb{R}^{d}\times\mathbb{R}^{d}}\left\vert W \right\vert(x)\widehat{(\mu-\nu})(y)(\widehat{S}\widehat{\nu})(x-y)\left(\widehat{h\left[\mu\right]-h\left[\nu\right]}\right)(x)dydx\nonumber\\
&\leq \left\Vert \widehat{\nu}\right\Vert _{\infty}\int_{\mathbb{R}^{d}\times\mathbb{R}^{d}}\left\vert W \right\vert(x)\left|\widehat{\mu-\nu}\right|(y)\left|\widehat{S}\right|(x-y)\left|\widehat{h\left[\mu\right]-h\left[\nu\right]}\right|(x)dydx.
\label{aux}
\end{align}
Thanks to the assumption $\left|\widehat{S}\right|\star\left\vert W \right\vert(x)\leq c\left\vert W \right\vert(x)$ and since $\widehat{S}$ is even we have 
\begin{align*}
\int_{\mathbb{R}^{d}\times\mathbb{R}^{d}}\left|\widehat{\mu-\nu}\right|^{2}(y)\left|\widehat{S}\right|(x-y)\left\vert W \right\vert(x)dxdy&=\int_{\mathbb{R}^{d}}\left|\widehat{\mu-\nu}\right|^{2}(y)\left|\widehat{S}\right|\star\left\vert W \right\vert(y)dy\\
&\leq c\int_{\mathbb{R}^{d}}\left|\widehat{\mu-\nu}\right|^{2}(y)\left\vert W \right\vert(y)dy.
\end{align*}
Furthermore, $\widehat{S}\in L^{1}(\mathbb{R}^{d})$
shows that 
\[
\int_{\mathbb{R}^{d}\times\mathbb{R}^{d}}\left|\widehat{h\left[\mu\right]-h\left[\nu\right]}\right|^{2}(x)\left|\widehat{S}\right|(x-y)\left\vert W \right\vert(x)dxdy=\left\Vert \widehat{S}\right\Vert _{1}\int_{\mathbb{R}^{d}}\left|\widehat{h\left[\mu\right]-h\left[\nu\right]}\right|^{2}(x)\left\vert W \right\vert(x)dx.
\]
Collecting the last two estimates and applying Cauchy-Schwarz to \eqref{aux}, we conclude that
$$
I_1\leq C \left(\int_{\mathbb{R}^{d}}\left|\widehat{\mu-\nu}\right|^{2}(y)\left\vert W \right\vert(y)dy\right)^{1/2} I^{1/2}
$$
where $C=C\left(c,\left\Vert \widehat{S}\right\Vert _{1},\left\Vert \widehat{\nu} \right\Vert_{\infty}\right).$ As for the second integral in (\ref{equation}) 

\begin{align*}
I_{2}&=\int_{\mathbb{R}^{d}}\left\vert W \right\vert(x)((\widehat{S}\widehat{(\mu-\nu))}\star\widehat{\mu})(x)\left(\widehat{h\left[\mu\right]-h\left[\nu\right]}\right)(x)dx \notag\\
&=\int_{\mathbb{R}^{d}\times\mathbb{R}^{d}}\left\vert W \right\vert(x)\widehat{S}(x-y)\widehat{(\mu-\nu)}(x-y)\widehat{\mu}(y)\left(\widehat{h\left[\mu\right]-h\left[\nu\right]}\right)(x)dxdy 
\end{align*}
and by the same considerations as before we deduce that 
\begin{align*}
I_{2}\leq C\left(c,\left\Vert \widehat{S}\right\Vert _{1},\left\Vert \widehat{\mu}\right\Vert _{\infty}\right)\left(\int_{\mathbb{R}^{d}}\left|\widehat{\mu-\nu}\right|^{2}(y)\left\vert W \right\vert(y)dy\right)^{1/2}I^{\frac{1}{2}}   
\end{align*}
from which we deduce that  
\[
\int_{\mathbb{R}^{d}}\left\vert W \right\vert(x)\left|\widehat{h\left[\mu\right]-h\left[\nu\right]}\right|^{2}(x)dx\leq C\int_{\mathbb{R}^{d}}\left\vert W \right\vert(x)\left|\widehat{\mu-\nu}\right|^{2}(x)dx
\]
where $C=C\left(c,\left\Vert \widehat{S}\right\Vert _{1},\left\Vert \widehat{\mu}\right\Vert _{\infty},\left\Vert \widehat{\nu}\right\Vert _{\infty}\right)$.
\begin{flushright}
$\square$
\par\end{flushright}

As a corollary, we obtain continuity of $h[\mu]$ with respect to the $\dot H^1(\mathbb{R}^d)$ norm.

\begin{cor} \label{h Lip coro}
Suppose that $\mu,\nu \in L^1 \cap L^\infty (\mathbb{R}^d)$ and that $S$ is as in (\textbf{H1}). Then, we have 
\[\left\Vert h[\mu]-h[\nu]\right\Vert _{\dot H^{-1}(\mathbb{R}^d)}\leq C\left\Vert \mu-\nu\right\Vert _{\dot H^{-1}(\mathbb{R}^d)}\] 
where 
$C=C\left(\left\Vert (1+\left|\cdot\right|^{2})\widehat{S}\right\Vert _{L^1 \cap L^\infty},\left\Vert \widehat{\mu}\right\Vert _{\infty},\left\Vert \widehat{\nu}\right\Vert _{\infty}\right)$.
\end{cor}
\textit{Proof}. We take $W=\widehat{V}=\frac{1}{\left\vert x \right\vert^{2}}$ in Lemma \ref{h Lip estimate }. We need to check the condition 
\[\left|\widehat{S}\right|\star W(x)\leq cW(x).\]
Indeed, we have 
\begin{align*}
\left|x\right|^{2}\left(\left\vert\widehat{S}\right\vert\star\frac{1}{\left|\cdot\right|^{2}}\right)(x)&=\left|x\right|^{2}\int_{\mathbb{R}^{d}}\frac{\left\vert\widehat{S}\right\vert(y)}{\left|x-y\right|^{2}}dy\leq2\int_{\mathbb{R}^{d}}\frac{\left\vert\widehat{S}\right\vert(y)\left|x-y\right|^{2}}{\left|x-y\right|^{2}}dy+2\left\Vert (\left|\cdot\right|^{2}\left\vert\widehat{S}\right\vert)\star\frac{1}{\left|\cdot\right|^{2}}\right\Vert _{\infty}\\
&=2\left\Vert \widehat{S}\right\Vert _{1}+2\left\Vert (\left|\cdot\right|^{2}\left\vert\widehat{S}\right\vert)\star\frac{\mathbf{1}_{\leq1}}{\left|\cdot\right|^{2}}\right\Vert _{\infty}+2\left\Vert (\left|\cdot\right|^{2}\left\vert\widehat{S}\right\vert)\star\frac{\mathbf{1}_{\geq1}}{\left|\cdot\right|^{2}}\right\Vert _{\infty}.
\end{align*}
By Young's inequality, for any $q<\frac{d}{2}$ we have the estimate  
\begin{align*}
\left\Vert (\left|\cdot\right|^{2}\left\vert\widehat{S}\right\vert)\star\frac{\mathbf{1}_{\leq1}}{\left|\cdot\right|^{2}}\right\Vert _{\infty}\leq\left\Vert \frac{\mathbf{1}_{\leq1}}{\left|\cdot\right|^{2}}\right\Vert _{q}\left\Vert \left|\cdot\right|^{2}\widehat{S}\right\Vert _{q'}\lesssim_{d,q}\left\Vert \left|\cdot\right|^{2}\widehat{S}\right\Vert _{q'}
\leq\left\Vert (1+\left|\cdot\right|^{2})\widehat{S}\right\Vert _{L^1 \cap L^\infty}
\end{align*}
and for any $p>\frac{d}{2}$ we have the estimate 
\begin{align*}
\left\Vert (\left|\cdot\right|^{2}\left\vert\widehat{S}\right\vert)\star\frac{\mathbf{1}_{\geq1}}{\left|\cdot\right|^{2}}\right\Vert _{\infty}\leq\left\Vert \frac{\mathbf{1}_{\geq1}}{\left|\cdot\right|^{2}}\right\Vert _{p}\left\Vert \left|\cdot\right|^{2}\widehat{S}\right\Vert _{p'}\lesssim_{d,q}\left\Vert \left|\cdot\right|^{2}\widehat{S}\right\Vert _{p'}\leq\left\Vert (1+\left|\cdot\right|^{2})\widehat{S}\right\Vert _{L^1 \cap L^\infty}.    \end{align*}
Therefore, one finds that  
\begin{align*}
\left\vert\widehat{S}\right\vert\star\frac{1}{\left|\cdot\right|^{2}}\leq\frac{c\left(\left\Vert (1+\left|\cdot\right|^{2})\widehat{S}\right\Vert _{L^1 \cap L^\infty}\right)}{\left|x\right|^{2}}.    
\end{align*}
The announced result follows now by Lemma \ref{h Lip estimate }.
\begin{flushright}
$\square$
\par\end{flushright}

\begin{thm}
Let hypothesis (\textbf{H1})  hold. Let $\mu^{1}(t,\cdot)$ and $\mu^{2}(t,\cdot)$ be weak solutions to \eqref{mean field equation sec2} with initial data $\mu_{0}^{1}, \mu_{0}^{2}$ satisfying (\textbf{H2}). Then, the $\dot{H}^{-1}(\mathbb{R}^{d})$-stability holds, i.e.
\[
\mathcal{E}(t)\leq e^{Ct}\mathcal{E}(0),
\]
where 
\begin{equation*}
 C=C\left(\left\Vert (1+\left|\cdot\right|^{2})\widehat{S}\right\Vert _{L^1 \cap L^\infty},\left\Vert \mu_{0}^{1}\right\Vert _{\infty},T\right).   
\end{equation*}
In particular, uniqueness of weak solutions to \eqref{mean field equation sec2} is obtained. 
\end{thm}
\textit{Proof}. Recall that $\mathrm{div}(\mathbb{J}\nabla V)=0$ when $\mathbb{J}$ is anti-symmetric and $\mathrm{div}(\mathbb{J}\nabla V)=-\delta_{0}$ when $\mathbb{J}$ is the identity. We compute that 
\begin{align*}
\frac{d}{dt}\mathcal{E}(t)=&-2\int_{\mathbb{R}^{d}}\left(\mu^{1}\mathbb{J}\nabla V\star\mu^{1}-\mu^{2}\mathbb{J}\nabla V\star\mu^{2}\right)(t,x)\nabla V\star\left(\mu^{1}-\mu^{2}\right)(t,x)dx\\
&+2\int_{\mathbb{R}^{d}}\left(h\left[\mu^{1}\right]-h\left[\mu^{2}\right]\right)(t,x)V\star\left(\mu^{1}-\mu^{2}\right)(t,x)dx\coloneqq \mathcal{D}^{1}(t)+\mathcal{D}^{2}(t).
\end{align*}
\textbf{Step 1}. \textbf{Estimate on} $\mathcal{D}^{1}(t)$. We can rewrite $\mathcal{D}^{1}(t)$ as
\begin{align*}
\mathcal{D}^{1}(t)=&-2\int_{\mathbb{R}^{d}}\mathbb{J}\nabla V\star\mu^{1}\left(\mu^{1}-\mu^{2}\right)(t,x)\nabla V\star\left(\mu^{1}-\mu^{2}\right)(t,x)dx\\
&-2\int_{\mathbb{R}^{d}}\mu^{2}\left(\mathbb{J}\nabla V\star(\mu^{1}-\mu^{2})\right)(t,x)\nabla V\star\left(\mu^{1}-\mu^{2}\right)(t,x)dx\\
\leq&-2\int_{\mathbb{R}^{d}}\mathbb{J}\nabla V\star\mu^{1}\left(\mu^{1}-\mu^{2}\right)(t,x)\nabla V\star\left(\mu^{1}-\mu^{2}\right)(t,x)dx, 
\end{align*}
where in the last inequality we used $\mathbb{J}\nabla V\star \mu \nabla V\star \mu \geq 0$.   
Therefore, we deduce
\begin{align*}
\mathcal{D}^{1}(t)&\leq -2\int_{\mathbb{R}^{d}}\mathbb{J}\nabla V\star\mu^{1}\mathrm{div}\left(\nabla V\star\left(\mu^{1}-\mu^{2}\right)\right)(t,x)\nabla V\star\left(\mu^{1}-\mu^{2}\right)(t,x)dx\\
&=-\int_{\mathbb{R}^{d}}\mathbb{J}\nabla V\star\mu^{1}(t,x)\nabla\left(\left|\nabla V\star\left(\mu^{1}-\mu^{2}\right)\right|^{2}\right)(t,x)dx\\
&=\int_{\mathbb{R}^{d}}\mathrm{div}(\mathbb{J}\nabla V \star \mu^{1})(t,x)\left|\nabla V\star\left(\mu^{1}-\mu^{2}\right)\right|^{2}(t,x)dx\\
&\leq2\underset{t\in [0,T]}{\sup}\left\Vert \mu^{1}(t,\cdot) \right\Vert_{\infty}\mathcal{E}(t)\leq C(\left\Vert S\right\Vert _{\infty},\left\Vert \mu^{1}_{0}\right\Vert _{\infty},T)\mathcal{E}(t).
\end{align*}
Hence, we conclude
\begin{equation}
\mathcal{D}^{1}(t)\leq C(\left\Vert S\right\Vert _{\infty},\left\Vert \mu^{1}_{0}\right\Vert _{\infty},T)\mathcal{E}(t).\label{inequality for E1}    
\end{equation}

\textbf{Step 2}. \textbf{Estimate on} $\mathcal{D}^{2}(t)$. By Plancherel's
theorem we have 
\begin{align}
\mathcal{D}_{2}(t)&=2\int_{\mathbb{R}^{d}}\widehat{(h\left[\mu^{1}\right]-h\left[\mu^{2}\right])}(t,x)\widehat{V}(x)\widehat{\left(\mu^{1}-\mu^{2}\right)}(t,x)dx\notag\\
&\leq \mathbf{C}\int_{\mathbb{R}^{d}}\frac{\left\vert \widehat{h\left[\mu^{1}\right]-h\left[\mu^{2}\right]} \right\vert }{\left|x\right|}(t,x)\frac{\widehat{\left|\mu^{1}-\mu^{2}\right|}}{\left|x\right|}(t,x)dx\notag\\
&\leq {\mathbf{C}}\left(\int_{\mathbb{R}^{d}}\frac{\left|\widehat{h\left[\mu^{1}\right]-h\left[\mu^{2}\right]}\right|^{2}}{\left|x\right|^{2}}(t,x)dx\right)^{\frac{1}{2}}\left(\int_{\mathbb{R}^{d}}\frac{\left|\widehat{\mu^{1}-\mu^{2}}\right|^{2}}{\left|x\right|^{2}}(t,x)dx\right)^{\frac{1}{2}} \notag\\
&\leq\mathbf{C}\left\Vert h\left[\mu^{1}\right](t,\cdot)-h\left[\mu^{2}\right](t,\cdot)\right\Vert _{\dot{H}^{-1}(\mathbb{R}^{d})}\left\Vert \mu^{1}(t,\cdot)-\mu^{2}(t,\cdot)\right\Vert _{\dot{H}^{-1}(\mathbb{R}^{d})}. \label{inequality for E2}
\end{align}
By Corollary \ref{h Lip coro}  we have 
\[
\left\Vert h\left[\mu^{1}\right](t,\cdot)-h\left[\mu^{2}\right](t,\cdot)\right\Vert _{\dot{H}^{-1}(\mathbb{R}^{d})}\leq C\left\Vert \mu^{1}(t,\cdot)-\mu^{2}(t,\cdot)\right\Vert _{\dot{H}^{-1}(\mathbb{R}^{d})}=C\sqrt{\mathcal{E}(t)},
\]
where $C=C\left(\left\Vert (1+\left|\cdot\right|^{2})\widehat{S}\right\Vert _{L^\infty \cap L^1}\right)$, so that by \eqref{inequality for E2} we get
\[
\mathcal{D}^{2}(t)\leq C\mathcal{E}(t).
\]
To conclude, \eqref{inequality for E1} and \eqref{inequality for E2}  show that
\[
\frac{d}{dt}\mathcal{E}(t)\leq C\left(\left\Vert (1+\left|\cdot\right|^{2})\widehat{S}\right\Vert _{L^\infty \cap L^1},\left\Vert \mu^{1}_{0}\right\Vert _{\infty},T\right)\mathcal{E}(t),
\]
and therefore $\mathcal{E}(t)\leq e^{Ct}\mathcal{E}(0)$.\qed


\section{Well posedness for the ODE system \label{WELL POSEDNESS SEC} }

In this section, we explain how to prove the existence of a well defined
flow for the system 
\begin{equation}
\left\{ \begin{array}{lc}
\dot{x}_{i}^{N}(t)=-\frac{1}{N}\stackrel[j=1]{N}{\sum}m_{j}^{N}(t)\mathbb{J}\nabla V(x_{i}^{N}(t)-x_{j}^{N}(t)),\ x_{i}^{N}(0)=x_{i}^{0,N}\\
\dot{m}_{i}^{N}(t)=\frac{1}{N}\stackrel[j=1]{N}{\sum}m_{i}^{N}(t)m_{j}^{N}(t)S(x_{i}^{N}(t)-x_{j}^{N}(t)),\ m_{i}^{N}(0)=m_{i}^{0,N}.
\end{array}\right.\label{eq:-19}
\end{equation}
We will adapt the proof about the well definition of the flow for time independent weights Riesz potentials as in  \cite[Section 3.2]{rosenzweig2020justification} to
the time dependent setting. Consider the weighted interaction energy given by 
\[
\mathcal{H}_{N}(t)\coloneqq\underset{i\neq j}{\sum}m_{i}(t)m_{j}(t)V(x_{i}(t)-x_{j}(t)). 
\]
We start by proving short time existence. 
\begin{thm}
Let hypothesis (\textbf{H1}) hold. Assume that 

\begin{equation*}
\forall i\neq j:x_{i}^{0,N}\neq x_{j}^{0,N}
\end{equation*}
and there is some $M>0$ such that for all $N\in\mathbb{N}$ and $1\leq i\leq N$
one has 

\begin{equation*}
0\leq m_{i}^{0,N}\leq M.
\end{equation*}
Then, there is some $T_{\ast}>0$ such that the system (\ref{eq:-19}) has a unique $C^{1}([0,T_{\ast});\mathbb{R}^{dN}\times\mathbb{R}^{N})$
solution.     
\label{short time}
\end{thm}
\textit{Proof}. 
Let $\Psi\in C^{\infty}_{0}(\mathbb{R})$ be such that $\Psi(r)=1$ for $\left| r \right|\leq 1$ and $\Psi(r)=0$ for $\left| r \right|\geq 2$. Let
$\psi\in C_{0}^{\infty}(\mathbb{R}^{d})$ be the radial function defined by $\psi(x) \coloneqq \Psi(\left| x \right|)$. Let 
\[
\psi_{\varepsilon}(x)\coloneqq\Psi\left(\frac{\left|x\right|^{2}}{\varepsilon^{2}}\right),
\]
and 
$V_{\varepsilon}(x)\coloneqq\left(1-\psi_{\varepsilon}(x)\right)V(x)$.
In what follows we omit the superscript $N$ from $x_{i,\varepsilon}^{N}$ and $m_{i,\varepsilon}^{N}$. Consider the regularized system 
\begin{equation}
\left\{ \begin{array}{lc}
{\dot{x}_{i,\varepsilon}}(t)=-\frac{1}{N}\stackrel[j=1]{N}{\sum}m_{j,\varepsilon}(t)\mathbb{J}\nabla V_{\varepsilon}(x_{i,\varepsilon}(t)-x_{j,\varepsilon}(t)),\ x_{i,\varepsilon}(0)=x_{i}^{0}\\
\dot{m}_{i,\varepsilon}(t)=\frac{1}{N}\stackrel[j=1]{N}{\sum}m_{i,\varepsilon}(t)m_{j,\varepsilon}(t)S(x_{i,\varepsilon}(t)-x_{j,\varepsilon}(t)),\ m_{i,\varepsilon}(0)=m_{i}^{0}.
\end{array}\right.\label{eq:Regularized trajectories}
\end{equation}
By  \cite[Theorem 3]{ayi2021mean}, system (\ref{eq:Regularized trajectories})
has a unique global solution $(x_{i,\varepsilon}(t),m_{i,\varepsilon}(t))$
on $[0,\infty)$. We study the evolution of $\left|x_{i,\varepsilon}(t)-x_{i}^{0}\right|$ given by
$$
\left|x_{i,\varepsilon}(t)-x_{i}^{0}\right|\leq \int_0^t \left|\frac{1}{N}\stackrel[j=1]{N}{\sum}m_{j,\varepsilon}(s)\cdot\mathbb{J}\nabla V_{\varepsilon}(x_{j,\varepsilon}(s)-x_{i,\varepsilon}(s))\right| \,ds\,.
$$
We can now estimate the $L^\infty$ bound of the velocity field using that
\[
\nabla V_{\varepsilon}(x)=-2\Psi'\left(\frac{\left|x\right|^{2}}{\varepsilon^{2}}\right)\frac{x}{\varepsilon^{2}}V(x)+\left(1-\psi_{\varepsilon}(x)\right)\nabla V(x).
\]
Actually, writing $k=d-2$ notice that 
\begin{align*}
\left|\left(1-\psi_{\varepsilon}(x)\right)\nabla V(x)\right|\lesssim\frac{1}{\varepsilon^{k+1}}    
\qquad
\mbox{and} 
\qquad
\left|2\Psi'\left(\frac{\left|x\right|^{2}}{\varepsilon^{2}}\right)\frac{x}{\varepsilon^{2}}V(x)\right|\lesssim\frac{1}{\varepsilon^{k+1}},
\end{align*}
hence 
\[
\left|\nabla V_{\varepsilon}(x)\right|\lesssim\frac{1}{\varepsilon^{k+1}}.
\]
So, owing to Remark \ref{conservation of total weight}, we infer the inequality  
\begin{equation}
\left|x_{i,\varepsilon}(t)-x_{i}^{0}\right|\lesssim\frac{1}{\varepsilon^{k+1}N}\int_0^t\stackrel[j=1]{N}{\sum}m_{j,\varepsilon}(s) \, ds=\frac{T}{\varepsilon^{k+1}}.
\label{eq:-22}
\end{equation}
Thanks to Inequality (\ref{eq:-22}), we can bound the seperation
in time $t$ as follows 

\[
\left|x_{i,\varepsilon}(t)-x_{j,\varepsilon}(t)\right|\geq\left|x_{i}^{0}-x_{j}^{0}\right|-\left|x_{i}^{0}-x_{i,\varepsilon}(t)\right|-\left|x_{j}^{0}-x_{j,\varepsilon}(t)\right|\geq\left|x_{i}^{0}-x_{j}^{0}\right|-\frac{2C_{k,d}T}{\varepsilon^{k+1}}.
\]
If we choose $\varepsilon_{0}=\frac{1}{16}\underset{i\neq j}{\min}\left|x_{i}^{0}-x_{j}^{0}\right|$ and
$T>0$ such that 
\[
\frac{2C_{k,d}T}{\varepsilon_{0}^{k+1}}\leq\frac{\underset{i\neq j}{\min}\left|x_{i}^{0}-x_{j}^{0}\right|}{2},
\]
then for each $t\in[0,T)$ we get the bound 
\begin{equation}
\left|x_{i,\varepsilon_{0}}(t)-x_{j,\varepsilon_{0}}(t)\right|\geq\frac{\left|x_{i}^{0}-x_{j}^{0}\right|}{2}.\label{eq:-20}
\end{equation}
It follows from (\ref{eq:-20}) that for each $t\in[0,T]$ it holds
that $\nabla V_{\varepsilon}(x_{i,\varepsilon}(t)-x_{j,\varepsilon}(t))=\nabla V(x_{i,\varepsilon}(t)-x_{j,\varepsilon}(t))$,
and therefore for each $t\in[0,T]$
\[
\left\{ \begin{array}{lc}
\dot{x}_{i,\varepsilon_{0}}(t)=-\frac{1}{N}\stackrel[j=1]{N}{\sum}m_{j,\varepsilon_{0}}(t)\mathbb{J}\nabla V(x_{i,\varepsilon_{0}}(t)-x_{j,\varepsilon_{0}}(t)),\ x_{i,\varepsilon_{0}}(0)=x_{i}^{0}\\
\dot{m}_{i,\varepsilon_{0}}(t)=\frac{1}{N}\stackrel[j=1]{N}{\sum}m_{i,\varepsilon_{0}}(t)m_{j,\varepsilon_{0}}(t)S(x_{i,\varepsilon_{0}}(t)-x_{j,\varepsilon_{0}}(t)),\ m_{i,\varepsilon_{0}}(0)=m_{i}^{0}.
\end{array}\right.
\]
\qed\\
Next, we claim that the  quantity $\mathcal{H}_{N}(t)$ is propagated in time, and
already here the oddness of $S$ is essential. 
\begin{lem}
\label{conservation of energy} Let $(\mathbf{x}_N(t),\mathbf{m}_N(t))$ be a solution to the system of ODEs \eqref{eq:-19} with initial data $\mathbf(\mathbf{x}_N(0),\mathbf{m}_N(0))$. Then, the interaction energy satisfies
\[
\mathcal{H}_{N}(t)\leq e^{2\left\Vert S \right\Vert_{\infty}t}\mathcal{H}_{N}(0), \ for\ all \ t\in[0,T).
\]
\end{lem}
\begin{proof}
We compute the
time derivative as follows 
\begin{align*}
\frac{d}{dt}\mathcal{H}_{N}(t)=&\frac{d}{dt}\underset{i\neq j}{\sum}m_{i}(t)m_{j}(t)V(x_{i}(t)-x_{j}(t))\\=&\underset{i\neq j}{\sum}\dot{m_{i}}(t)m_{j}(t)V(x_{i}(t)-x_{j}(t))+\underset{i\neq j}{\sum}m_{i}(t)\dot{m_{j}}(t)V(x_{i}(t)-x_{j}(t))
\end{align*}
\begin{equation}
+\underset{i\neq j}{\sum}m_{i}(t)m_{j}(t)\nabla V(x_{i}(t)-x_{j}(t))\cdot(\dot{x_{i}}(t)-\dot{x_{j}}(t)).   \label{eq51} 
\end{equation}
The first 2 terms in the right hand side of \eqref{eq51} are paired together as follows. 

\[
\underset{i\neq j}{\sum}\dot{m_{i}}(t)m_{j}(t)V(x_{i}(t)-x_{j}(t))=\frac{1}{N}\underset{i\neq j}{\sum}\underset{k}{\sum}m_{i}(t)m_{k}(t)m_{j}(t)S(x_{k}(t)-x_{i}(t))V(x_{i}(t)-x_{j}(t)),
\]
while 
\begin{align*}
\underset{i\neq j}{\sum}m_{i}(t)\dot{m_{j}}(t)V(x_{i}(t)-x_{j}(t))=\frac{1}{N}\underset{i\neq j}{\sum}\underset{k}{\sum}m_{i}(t)m_{j}(t)m_{k}(t)S(x_{k}(t)-x_{j}(t))V(x_{i}(t)-x_{j}(t))\\
=\frac{1}{N}\underset{i\neq j}{\sum}\underset{k}{\sum}m_{i}(t)m_{j}(t)m_{k}(t)S(x_{k}(t)-x_{i}(t))V(x_{j}(t)-x_{i}(t)).
\end{align*}
So, since $V$ is even, the first 2 terms in \eqref{eq51} add up to 
\begin{align}
&\frac{2}{N}\underset{i\neq j,k}{\sum}m_{i}(t)m_{j}(t)m_{k}(t)S(x_{k}(t)-x_{i}(t))V(x_{j}(t)-x_{i}(t)) \notag\\
&\leq\frac{2\left\Vert S \right\Vert_{\infty}}{N}\underset{i\neq j,k}{\sum}m_{i}(t)m_{j}(t)m_{k}(t)V(x_{j}(t)-x_{i}(t))=2\left\Vert S \right\Vert_{\infty}\mathcal{H}_{N}(t), 
\label{cons1}
\end{align}
where in the last equation we used conservation of the total weight (see Remark \ref{conservation of total weight}). The third sum in the right hand side of \eqref{eq51} writes 
\begin{align*}
\underset{i\neq j}{\sum} m_{i}(t)m_{j}(t)\nabla V(x_{i}(t)-x_{j}(t))\cdot&\left(\frac{1}{N}\underset{k\neq j}{\sum} m_{k}(t)\mathbb{J}\nabla V(x_{j}(t)-x_{k}(t))-\frac{1}{N}\underset{k\neq i}{\sum}m_{k}(t)\mathbb{J}\nabla V(x_{i}(t)-x_{k}(t))\right) \notag \\
&=-\frac{2}{N}{\underset{i,j,k}{\widehat\sum}}m_{i}(t)m_{j}(t)m_{k}(t)\nabla V(x_{i}(t)-x_{j}(t))\cdot \mathbb{J}\nabla V(x_{i}(t)-x_{k}(t)) 
\end{align*} 
where the equality is obtained by swapping the indices $i$ and $j$ in the first term and the hat sum refers to a sum over the set of indices $\{(i,j,k) \mbox{ such that } i\neq j, k\neq i\}$. We finally notice that we can rewrite this term as follows
\begin{align}
{\underset{i,j,k}{\widehat\sum}}m_{i}(t)&m_{j}(t)m_{k}(t)\nabla V(x_{i}(t)-x_{j}(t))\cdot \mathbb{J}\nabla V(x_{i}(t)-x_{k}(t)) \nonumber\\&=-\frac{2}{N} \sum_i m_{i}(t)  \left(\sum_{j\neq i} m_{j}(t) \nabla V(x_{i}(t)-x_{j}(t))\right) \cdot \mathbb{J} \left(\sum_{k\neq i} m_{k}(t) \nabla V(x_{i}(t)-x_{k}(t))\right)\nonumber\\&=-\frac{2}{N} \sum_i m_{i}(t)  \left(\sum_{j\neq i} m_{j}(t) \nabla V(x_{i}(t)-x_{j}(t))\right) \cdot \mathbb{J} \left(\sum_{j\neq i} m_{j}(t) \nabla V(x_{i}(t)-x_{j}(t))\right).
\label{rhs}
\end{align} 
If $\mathbb{J}$ is anti-symmetric then the right-hand side of \eqref{rhs} is $0$ while if $\mathbb{J}$ is the identity matrix the right hand side of \eqref{rhs} is 
\begin{align*}
-\frac{2}{N} \sum_i m_{i}(t)  \left\vert\sum_{j\neq i} m_{j}(t) \nabla V(x_{i}(t)-x_{j}(t))\right\vert^{2}\leq 0. 
\end{align*}
Therefore, we conclude that the third sum in the right hand side of \eqref{eq51} satisfies 
\begin{align}
\underset{i\neq j}{\sum}m_{i}(t)m_{j}(t)\nabla V(x_{i}(t)-x_{j}(t))\cdot(\dot{x_{i}}(t)-\dot{x_{j}}(t))\leq 0. 
\label{rhs2} 
\end{align}
Together with \eqref{cons1} this concludes the proof. 
\end{proof}
As a corollary from propagation of energy, we get the following lower bound on the minimal separation
of the opinions, which will be needed in order to prove that maximal
lifespan solutions are in fact global.
We are now well positioned to prove Theorem \ref{well posedness -1}, i.e. the existence of a globally well-defined flow.

\

\textit{Proof of Theorem \ref{well posedness -1}}. \textbf{Step 1}. \textbf{Estimate on the separation}. Set $k\coloneqq d-2$. We claim that if  $(\mathbf{x}_{N}(t),\mathbf{m}_{N}(t))$
is a solution to the system \eqref{eq:-19},  then there holds the estimate  
\begin{align}
\underset{i\neq j}{\min}\underset{t\in[0,T)}{\inf}\left|x_{i}(t)-x_{j}(t)\right|\geq\min\left\{ 1,\frac{1}{e^{2\left\Vert S \right\Vert_{\infty}T}\mathcal{H}_{N}(0)}\right\}^{1/{k}}. \label{minimal dist}    
\end{align}
If $i\neq j$ are such that    $\left|x_{i}(t)-x_{j}(t)\right|\leq1$ then by Lemma \ref{conservation of energy}
we have 

\[
\frac{1}{\left|x_{i}(t)-x_{j}(t)\right|^{k}}\leq\mathcal{H}_{N}(t)\leq e^{2\left\Vert S \right\Vert_{\infty}t}\mathcal{H}_{N}(0),
\]
so that 
\[
\left|x_{i}(t)-x_{j}(t)\right|^{k}\geq\frac{1}{e^{2\left\Vert S \right\Vert_{\infty}t}\mathcal{H}_{N}(0)}.
\]
It follows that for all $i\neq j$ one has 
\[
\left|x_{i}(t)-x_{j}(t)\right|\geq\min\left\{ 1,\frac{1}{e^{2\left\Vert S \right\Vert_{\infty}T}\mathcal{H}_{N}(0)}\right\}^{1/{k}} ,
\]
which establishes \eqref{minimal dist}.

\

\textbf{Step 2}. \textbf{Long time existence and uniqueness}. Let $(\mathbf{x}_{N}(t),\mathbf{m}_{N}(t))$
be a solution to the system of ODEs \eqref{eq:-19} with maximal lifespan $T>0$.
We claim that if $T<\infty$ then 
\[
\underset{t\nearrow T}{\lim}\underset{i\neq j}{\min}\left|x_{i}(t)-x_{j}(t)\right|=0.
\]
Indeed suppose that $T< \infty$ and assume on the contrary that 
\[
\underset{t\nearrow T}{\lim}\underset{i\neq j}{\min}\left|x_{i}(t)-x_{j}(t)\right|=\delta>0.
\]
 Choose $T'<T$ sufficiently close to $T$ such that 
\[
\underset{T'\leq t\leq T}{\inf}\underset{i\neq j}{\min}\left|x_{i}(t)-x_{j}(t)\right|\geq\frac{\delta}{2}.
\]
Assume also that $T-T'<T_{\Delta}$, where $T_{\Delta}$ is the maximal
lifespan solution of the system of ODEs \eqref{eq:-19} with initial data $(x_{i}(T'),m_{i}(T'))$, i.e. the equation 
\begin{equation}
\left\{ \begin{array}{lc}
\dot{z_{i}}(t)=-\frac{1}{N}\stackrel[j=1]{N}{\sum}n_{j}(t)\mathbb{J}\nabla V(z_{i}(t)-z_{j}(t)),\ z_{i}(0)=x_{i}(T')\\
\dot{n}_{i}(t)=\frac{1}{N}\stackrel[j=1]{N}{\sum}n_{i}(t)n_{j}(t)S(z_{i}(t)-z_{j}(t)),\ n_{i}(0)=m_{i}(T').
\end{array}\right.\label{eq:-21}
\end{equation}
 Note that by Theorem \ref{short time}, there exists a maximal life span solution $(\mathbf{z}_{N}(t),\mathbf{n}_{N}(t))$
on $[0,T_{\Delta})$ to the system of ODEs (\ref{eq:-21}). Define 
\[
\mathbf{y}_{N}(t)\coloneqq\left\{ \begin{array}{lc}
\mathbf{x}_{N}(t) & 0\leq t\leq T'\\
\mathbf{z}_{N}(t-T') & T'\leq t< T'+T_{\Delta}
\end{array}\right.,\ \mathbf{l}_{N}(t)\coloneqq\left\{ \begin{array}{lc}
\mathbf{m}_{N}(t) & 0\leq t\leq T'\\
\mathbf{n}_{N}(t-T') & T'\leq t< T'+T_{\Delta}
\end{array}\right..
\]
Note that $\mathbf{y}_{N}(t)$ is continuous and that by \eqref{minimal dist} we have 
\[
\underset{t\in[0,T'+T_{\Delta})}{\inf}\underset{i\neq j}{\min}\left|y_{i}(t)-y_{j}(t)\right|>0.
\]
We leave the reader to check that $(\mathbf{y}_{N}(t),\mathbf{l}_{N}(t))$ is
a solution to the system on $[0,T'+T_{\Delta})$ since we are dealing with an autonomous system. This entails a contradiction to the assumption that $T$ is maximal, because $T'+T_{\Delta}>T$. We conclude that 
\[
T<\infty\Longrightarrow\underset{t\nearrow T}{\lim}\underset{i\neq j}{\min}\left|x_{i}(t)-x_{j}(t)\right|=0.
\]
In view of \eqref{minimal dist} it follows that
$T=\infty$, as desired. 
\begin{flushright}
$\square$
\par\end{flushright}


\section{The mean field limit. }
In Section 5.1 we compute the time derivative of $\mathcal{E}_{N}(t)$, and in Section 5.2 we establish the functional inequality \eqref{functional ine sec 2} which in turn leads to a Gr\"onwall estimate on $\mathcal{E}_{N}(t)$. 

\subsection{Time Derivative of the re-normalized Modulated Energy. }

We recall that the re-normalized interaction energy \eqref{renomenergyN} is defined as
\begin{align*}
\mathcal{E}_{N}(t)&=\int_{x\neq y}V(x-y)\left(\mu_{N}(t,\cdot)-\mu(t,\cdot)\right)^{\otimes2}dxdy\\
&=\frac{1}{N^{2}}\underset{i\neq j}{\sum}m_{i}(t)m_{j}(t)V(x_{i}(t)-x_{j}(t))-\frac{2}{N}\stackrel[i=1]{N}{\sum}m_{i}(t)(V\star\mu)(x_{i}(t))+\int_{\mathbb{R}^{d}}\mu(t,x)(V\star\mu)(t,x)dx.
\end{align*}
The aim of this section is to compute the time derivative of $\mathcal{E}_{N}(t)$, which is given in the following result.

\begin{prop}\label{Time derivative of energy}  
Let hypotheses (\textbf{H1})-(\textbf{H2}) hold. Let $\mu(t,\cdot)$ be the unique solution to the PDE \eqref{mean field equation sec2} with initial data  $\mu_{0}$ provided by Theorem \ref{existence uniqueness intro}. Then, we have 
\begin{align*}
\frac{d}{dt}\mathcal{E}_{N}(t)\leq&-\int_{x\neq y}\left(\mathbb{J}\nabla V\star\mu(t,x)-\mathbb{J}\nabla V\star\mu(t,y)\right)\nabla V(x-y)\left(\mu_{N}(t,\cdot)-\mu(t,\cdot)\right)^{\otimes2}dxdy\\
&+2\int_{x\neq y}V(x-y)\left(h\left[\mu_{N}(t,\cdot)\right](x)-h\left[\mu(t,\cdot)\right](x)\right)\left(\mu_{N}(t,y)-\mu(t,y)\right)dxdy\\\coloneqq& \mathcal{D}_{N}^{1}(t)
+ \mathcal{D}_{N}^{2}(t).
\end{align*}
\end{prop}

\begin{proof}
We recall that from Theorem~\ref{well posedness -1}, we have that $x_i(t)\neq x_j(t)$ for all $t\geq 0$ and all $i\neq j$. This makes it straightforward to justify all calculations in the proof. To make the equations lighter we shall omit the time variable whenever there is no ambiguity. 

\smallskip

\textbf{Step 1}. \textit{Calculation of} $\frac{d}{dt}\frac{1}{N^{2}}\underset{i\neq j}{\sum}m_{i}(t)m_{j}(t)V(x_{i}(t)-x_{j}(t))$.
We have 
\begin{equation}\begin{aligned}
\frac{d}{dt}\frac{1}{N^{2}}\underset{i\neq j}{\sum}m_{i}(t)m_{j}(t)V(x_{i}(t)-x_{j}(t))=&\,\frac{1}{N^{2}}\underset{i\neq j}{\sum}m_{i}(t)m_{j}(t)\nabla V(x_{i}(t)-x_{j}(t))\cdot \left(\dot{x_{i}}(t)-\dot{x_{j}}(t)\right)\\
&+\frac{1}{N^{2}}\underset{i\neq j}{\sum}(\dot{m_{i}}(t)m_{j}(t)+m_{i}(t)\dot{m}_{j}(t))V(x_{i}(t)-x_{j}(t))\\
:=&\,E_1+T_1.  
\end{aligned}
\label{equation for cross term}
\end{equation}
The first term in \eqref{equation for cross term} is non-positive due to \eqref{rhs2} in Lemma \ref{conservation of energy}, i.e.  $E_1\leq 0$.
The second term in (\ref{equation for cross term}) can be rewritten using symmetry as
\begin{align*}
T_1=\frac{2}{N^{2}}\underset{i\neq j}{\sum}\dot{m_{i}}m_{j}V(x_{i}-x_{j})&=\frac{2}{N^{3}}\underset{i\neq j}{\underset{i,j,k}{\sum}}m_{i}m_{j}m_{k}S(x_{i}-x_{k})V(x_{i}-x_{j})\\
&=2\int_{x\neq y}V(x-y)h\left[\mu_{N}(t,\cdot)\right](x)\mu_{N}(t,y)dxdy.    
\end{align*}

\smallskip

 {\bf Step 2.} \textit{Calculations of} $-\frac{d}{dt}\frac{2}{N}\stackrel[i=1]{N}{\sum}m_{i}(t)(V\star\mu)(t,x_{i}(t))$ \textit{and} $\frac{d}{dt}\int_{\mathbb{R}^{d}}(V\star\mu)(t,x)\mu(t,x)$.
Let us start with
\begin{align*}
&-\frac{d}{dt}\frac{2}{N}\stackrel[i=1]{N}{\sum}m_{i}(t)V\star\mu(t,x_{i}(t))\\&=-\frac{2}{N}\stackrel[i=1]{N}{\sum}\dot{m_{i}}V\star\mu(x_{i})-\frac{2}{N}\stackrel[i=1]{N}{\sum}m_{i}V\star\partial_{t}\mu(t,x_{i}(t))-\frac{2}{N}\stackrel[i=1]{N}{\sum}m_{i}\nabla V\star\mu(x_{i})\dot{x}_{i}=T_2+E_2, 
\end{align*}
where we have set  
\begin{align}
& T_2\coloneqq-\frac{2}{N^{2}}\underset{i,k}{\sum}m_{i}m_{k}S(x_{i}-x_{k})V\star\mu(x_{i})
-\frac{2}{N}\stackrel[i=1]{N}{\sum}m_{i}V\star h\left[\mu\right](x_{i}),\notag\\
& E_2\coloneqq
-\frac{2}{N}\stackrel[i=1]{N}{\sum}m_{i}V\star(\mathrm{div}(\mu\mathbb{J}\nabla V\star\mu))(x_{i})-\frac{2}{N^{2}}\underset{i\neq k}{\sum}m_{i}m_{k}\nabla V\star\mu(x_{i})\mathbb{J}\nabla V(x_{k}-x_{i}).\label{E2}
\end{align}
Observe that we can further write that
\begin{equation*}
T_2=-2\int_{\mathbb{R}^{d}}V\star\mu(t,x)h\left[\mu_{N}(t,\cdot)\right]dx-2\int_{\mathbb{R}^{d}}V\star h\left[\mu\right](t,x)\mu_{N}(t,x)dx. \hspace{2.5 cm}
\end{equation*}
With similar calculations, we find that 
\begin{equation}
\begin{aligned}
&\frac{d}{dt}\int_{\mathbb{R}^{d}}V\star\mu(t,x)\mu(t,x)dx=2\int_{\mathbb{R}^{d}}V\star\left(\mathrm{div}(\mu\mathbb{J}\nabla V\star\mu)+h\left[\mu\right]\right)(x)\mu(x)dx\\
&\quad=2\int_{\mathbb{R}^{d}}V\star\left(\mathrm{div}(\mu\mathbb{J}\nabla V\star\mu)\right)(x)\mu(x)dx+2\int_{\mathbb{R}^{d}}V\star h\left[\mu\right](x)\mu(x)dx\coloneqq E_{3}+T_{3}.
\end{aligned}
\label{eq:-16}
\end{equation}
We observe that
\[
T_{1}+T_{2}+T_{3}=2\int_{x\neq y}V(x-y)\left(h\left[\mu_{N}(t,\cdot)\right](x)-h\left[\mu(t,\cdot)\right](x)\right)\left(\mu_{N}(t,y)-\mu(t,y)\right)dxdy,
\]
which implies that
\begin{equation}\label{auxx}
\begin{aligned}
&\frac{d}{dt}\frac{1}{N^{2}}\underset{i\neq j}{\sum}m_{i}(t)m_{j}(t)V(x_{i}(t)-x_{j}(t))=E_1+E_2+E_3 \\
&\quad+2\int_{x\neq y}V(x-y)\left(h\left[\mu_{N}(t,\cdot)\right](x)-h\left[\mu(t,\cdot)\right](x)\right)\left(\mu_{N}(t,y)-\mu(t,y)\right)dxdy,
\end{aligned}
\end{equation}
where $E_1\leq 0$, $E_2$ is given by~\eqref{E2} and $E_3$ by~\eqref{eq:-16}.

\

\textbf{Step 3}. Collecting \eqref{E2} and \eqref{eq:-16}, we get
\begin{align*}
 E_{2}+E_{3}
 =&\,2\int_{\mathbb{R}^{d}}V\star\left(\mathrm{div}(\mu\mathbb{J}\nabla V\star\mu)\right)(x)\mu(x)dx-\frac{2}{N}\stackrel[i=1]{N}{\sum}m_{i}V\star(\mathrm{div}(\mu\mathbb{J}\nabla V\star\mu))(x_{i})\\
 &-\frac{2}{N^{2}}\underset{i,k}{\sum}m_{i}m_{k}\nabla V\star\mu(x_{i})\mathbb{J}\nabla V(x_{k}-x_{i}).
\end{align*}
We now work in each term on the right-hand side as follows. We first rewrite the first term as
\begin{align*}
2\int_{\mathbb{R}^{d}}V\star\left(\mathrm{div}(\mu\mathbb{J}\nabla V\star\mu)\right)(x)\mu(x)dx&=-2\int_{\mathbb{R}^{d}}\mu(x)(\nabla V\star\mu)(x)(\mathbb{J}\nabla V\star\mu)(x)dx
\\&=-2\int_{\mathbb{R}^{d}\times\mathbb{R}^{d}}\nabla V(x-y)\mu(x)\mu(y)(\mathbb{J}\nabla V\star\mu)(x)dxdy,
\end{align*}
by integration by parts, while the second term can also be rewritten as
\begin{align*}
-\frac{2}{N}\stackrel[i=1]{N}{\sum}m_{i}V\star\mathrm{div}(\mu\mathbb{J}\nabla V\star\mu)(x_{i})&=-2\int_{\mathbb{R}^{d}\times\mathbb{R}^{d}}V(x-y)\mathrm{div}(\mu\mathbb{J}\nabla V\star\mu)(y)\mu_{N}(x)dxdy\\
&=-2\int_{\mathbb{R}^{d}\times\mathbb{R}^{d}}\nabla V(x-y)(\mathbb{J}\nabla V\star\mu)(y)\mu(y)\mu_{N}(x)dxdy\\
&=2\int_{\mathbb{R}^{d}\times\mathbb{R}^{d}}\nabla V(x-y)(\mathbb{J}\nabla V\star \mu)(x)\mu(x)\mu_{N}(y)dxdy,
\end{align*}
by integration by parts and symmetrization and the final term is equivalenty written as
\begin{align*}
-&\frac{2}{N^{2}}\underset{i,k}{\sum}m_{i}m_{k}\mathbb{J}\nabla V\star\mu(x_{i})\nabla V(x_{k}-x_{i})
=-2\int_{\mathbb{R}^{d}\times\mathbb{R}^{d}}\nabla V(x-y)\mathbb{J}\nabla V\star \mu(x)\mu^{\otimes 2}_{N}(dxdy).    
\end{align*}
Hence, we find that 
\begin{align*}
E_{2}+E_{3}=&\,-2\int_{\mathbb{R}^{d}\times \mathbb{R}^{d}}\nabla V(x-y)\mathbb{J}\nabla V\star \mu(x)\left(\mu_{N}-\mu\right)^{\otimes 2}(dxdy)\\
&\,-2\int_{\mathbb{R}^{d}}\mathbb{J}\nabla V\star\mu(x)\nabla V\star\mu(x)\mu_{N}(x)dx\\
\leq&\,-2\int_{\mathbb{R}^{d}\times \mathbb{R}^{d}}\nabla V(x-y)\mathbb{J}\nabla V\star \mu(x)\left(\mu_{N}-\mu\right)^{\otimes 2}(dxdy)
\\
=&\,-\int_{x\neq y}\left(\mathbb{J}\nabla V\star\mu(x)-\mathbb{J}\nabla V\star\mu(y)\right)\nabla V(x-y)\left(\mu_{N}-\mu\right)^{\otimes2}(dxdy),
\end{align*}
because $\mathbb{J}\nabla V\star\mu(x)\nabla V\star\mu(x)\geq 0$.
Hence, we conclude that 
\begin{align*}
E_{1}+ E_{2}+E_{3}\leq -\int_{x\neq y}\left(\mathbb{J}\nabla V\star\mu(x)-\mathbb{J}\nabla V\star\mu(y)\right)\nabla V(x-y)\left(\mu_{N}-\mu\right)^{\otimes2}dxdy,     
\end{align*} 
so that together with \eqref{auxx}, we get 
\begin{align*}
\frac{d}{dt}\mathcal{E}_{N}(t)\leq &-\int_{x\neq y}\left(\mathbb{J}\nabla V\star\mu(x)-\mathbb{J}\nabla V\star\mu(y)\right)\nabla V(x-y)\left(\mu_{N}-\mu\right)^{\otimes2}dxdy\\
&+2\int_{x\neq y}V(x-y)\left(h\left[\mu_{N}\right](x)-h\left[\mu\right](x)\right)\left(\mu_{N}(y)-\mu(y)\right)dxdy,
\end{align*}
as desired.
\end{proof}

\subsection{Functional Inequalities. }
This sub-section develops the most subtle part of this work. Indeed, re-normalizing the stability estimate from the previous section requires new arguments in comparison to \cite{bresch2019modulated}, due to the inclusion of a source term. The crucial technical argument is found in Lemma \ref{commutator estimate1} for which we need certain preliminary technical results.  
The following regularization Lemma is an adaptation of
\cite[Lemma 4.1]{bresch2019modulated}, which was proved for the periodic case, to the Euclidean
setting. We will designate by $\left\Vert V \right\Vert_{L^{p}+L^{q}}$ the norm of the space $L^{p}+L^{q}$ defined by 
\begin{align*}
\left\Vert V \right\Vert_{L^{p}+L^{q}}\coloneqq \mathrm{inf}({\left\Vert V_{1}\right\Vert_{p}+\left\Vert V_{2}\right\Vert_{q}})    
\end{align*}
where the infimum is taken over all splittings $V=V_{1}+V_{2}$ such that $V_{1}\in L^{p}(\mathbb{R}^{d})$ and $V_{2}\in L^{q}(\mathbb{R}^{d})$. 
\begin{lem}
\label{Regularization }   Let $V(x)=\frac{1}{\left|x\right|^{k}}$ where $0<k< d-1$. Let $0\leq \zeta \leq1$ be such that $\zeta \in C^{\infty}(\mathbb{R}^{d})$, $\zeta \equiv 0$ on $B_{1}(0)$ and $\zeta  \equiv 1$ on $\mathbb{R}^{d}\setminus B_{2}(0)$. Set $\zeta_{\delta}(x)=\zeta(\frac{x}{\delta})$ for $0<\delta<1$.  Assume  $\max\{\frac{d}{2k},1\}<p<\frac{d}{k}<q<\infty$, then for each $r\in \mathbb{N}$ there is some constant $C=C(p,q,d,k)$ and a $r$-differentiable approximation $V_{\varepsilon}\in C^{r}(\mathbb{R}^{d})$ of
$V$ such that 
\begin{itemize}
    \item[i.] $\widehat{V_{\varepsilon}}\geq0$.

    \item[ii.] $\left\Vert V_{\varepsilon}-V\right\Vert _{L^{p}+L^{q}}\leq\eta(\varepsilon)$ with $\eta(\varepsilon)$ satisfying $\eta(\varepsilon){\rightarrow}0$ as $\varepsilon\rightarrow0$.

    \item[iii.] $\left\Vert \zeta_{\delta}(V_{\varepsilon}-V)\right\Vert _{L^{p}+L^{q}}\leq C\left(\frac{\varepsilon}{\delta^{k+1-\frac{d}{q}}}+\delta^{\frac{d}{p}-k}\right).$

    \item[iv.] $V_{\varepsilon}(x)\leq V(x)+\varepsilon.$
    \item[v.] $\left\Vert V_{\varepsilon}\right\Vert_{\infty}\leq \frac{1}{\Theta(\varepsilon)}$ where $\Theta(\varepsilon){\rightarrow 0}$ as $\varepsilon\rightarrow0$.
\end{itemize}
\end{lem}

\textit{Proof.} \textbf{Step 0}. Note that $V\in L^{p}(\mathbb{R}^{d})+L^{q}(\mathbb{R}^{d})$ for $1<p<\frac{d}{k}<q<\infty$. We will need a further restriction  $p>\frac{d}{2k}$ later in the proof. In addition, we have $\widehat V= \mathbf{C}\left\vert x \right\vert^{-\kappa}$ with $\kappa=d-k$. 
Consider a kernel $K^{1}:\mathbb{R}^{d}\rightarrow \mathbb{R}$ such that 
\begin{itemize}
    \item $\int_{\mathbb{R}^{d}}K^{1}(z)dz=1$, $\widehat{K^{1}}\geq0$, and $\left|K^{1}(z)\right|\leq\exp\left(-\left|z\right|^{2}\right).$
    \item There is some $C>0$ and $r>0$ such that
\begin{align*}
   \frac{1}{C(1+\left\vert \xi \right\vert^{r})}\leq \widehat{K^{1}}(\xi)\leq \frac{C}{1+\left\vert \xi \right\vert^{r}}. 
\end{align*}
\end{itemize}
Let $K_{\delta}^{1}\coloneqq\frac{1}{\delta^{d}}K^{1}(\frac{x}{\delta})$. Fix $R=R(\delta)$ to be chosen later. Let $\chi\in C_{0}^{\infty}(\mathbb{R}^{d})$ 
such that $\chi\geq 0$, $\chi\equiv1$ on $B_{\frac{R}{2}}(0)$ and $\chi\equiv0$
on $\mathbb{R}^{d}\setminus B_{R}(0)$. Clearly, we get 
\begin{align}
    \left|K_{\delta}^{1}\star V(x)-V(x)\right|\leq&\,\int_{\mathbb{R}^{d}}\chi(z)\left\vert K^{1} \right\vert(z)\left|V(x-\delta z)-V(x)\right|dz\nonumber \\
    &+\int_{\mathbb{R}^{d}}(1-\chi(z))\left\vert K^{1} \right\vert(z)\left|V(x-\delta z)-V(x)\right|dz.  
\label{Eq}
\end{align}
Ultimately $V_{\varepsilon}$ will be defined by means of $K^{1}_{\delta(\varepsilon)}\star V$ for a well chosen function $\delta(\varepsilon)$. Estimating \eqref{Eq} will be done by distinguishing between possible intervals in which $\left\vert x\right\vert$ lies in. In this Lemma, $\lesssim$ stands for an inequality up to a constant depending only on $p,q,d,k$. \\  

\smallskip

\textbf{Step 1}.  The aim of this step is to prove there is some  $0<\lambda<1$ such that 
\begin{equation}
 K_{\delta}^{1}\star V(x)\leq V(x)+C\delta^{\lambda},\ \left|x\right|\geq \delta^{\lambda}
 \label{Eq 2Rdel}
 \end{equation}
 for all $\delta$ arbitrarily small. We assume that $\left\vert x\right\vert\geq 2R\delta$. 
First, note that $\left|x\right|\geq2\delta R$ and $\left|z\right|\leq R$ implies that $\left|x-\delta z\right|\geq\delta R$,  
so 
\begin{equation}
\int_{\mathbb{R}^{d}}\chi(z)\left\vert K^{1} \right\vert(z)\left|V(x-\delta z)-V(x)\right|dz\lesssim\frac{\delta R}{\left|x\right|^{k+1}}, \label{ine'}    
\end{equation}
due to $\left|\nabla V(x)\right|\lesssim\frac{1}{\left|x\right|^{k+1}}$ and the mean value theorem.
We split the second integral in (\ref{Eq}) as 
\begin{align*}
T_{1}+T_{2}\coloneqq
&\int_{\left|x-\delta z\right|\leq\frac{3\left|x\right|}{2}}(1-\chi(z))\left\vert K^{1} \right\vert(z)\left|V(x-\delta z)-V(x)\right|dz\\
&+\int_{\left|x-\delta z\right|\geq\frac{3\left|x\right|}{2}}(1-\chi(z))\left\vert K^{1} \right\vert(z)\left|V(x-\delta z)-V(x)\right|dz .
\end{align*}
The remaining of step 1 is occupied with estimating $T_{1}$ and $T_{2}$. The term $T_{2}$ is
\begin{align}
T_{2}&\leq\int_{\left|x-\delta z\right|\geq\frac{3\left|x\right|}{2}}(1-\chi(z))\left\vert K^{1} \right\vert(z)V(x-\delta z)dz+V(x)\int_{\mathbb{R}^{d}}(1-\chi(z))\left\vert K^{1} \right\vert(z)dz\notag\\&\lesssim (V(\tfrac{3x}{2})+V(x))\int_{\mathbb{R}^{d}}(1-\chi(z))\left\vert K^{1} \right\vert(z)dz.    \label{1st ine on T_{2}}
\end{align}
Observe that for any given $l>0$ we have the estimate 
\begin{align}
\int_{\mathbb{R}^{d}}(1-\chi(z))\left\vert K^{1}\right\vert(z)dz\leq \int_{\left\vert z \right\vert \geq \frac{R}{2}} e^{-\frac{\left\vert z \right\vert^{2}}{2}}dz \lesssim \frac{1}{R^{l}}.
\label{estimate on int of exp}\end{align}
Substituting \eqref{estimate on int of exp} into \eqref{1st ine on T_{2}} yields 
\begin{align}
T_{2}\lesssim \frac{V(x)}{R^{l}}.  \label{ine''}  
\end{align}
For the term $T_{1}$, we separate between two cases: $|x|\geq R$ first and then $|x|\leq R$. 
Start with $\left|x\right|\geq R$, where we have 
\begin{align*}
T_{1}\leq\int_{\mathbb{R}^{d}}(1-\chi(z))\left\vert K^{1}\right\vert(z)\left|V(x-\delta z)-V(x)\right|dz\lesssim V(x)+\int_{\mathbb{R}^{d}}(1-\chi(z))\left\vert K^{1}\right\vert(z)V(x-\delta z)dz.    
\end{align*}
We estimate 
\begin{align*}
\left|x\right|^{k}\int_{\mathbb{R}^{d}}(1-\chi(z))\left\vert K^{1}\right\vert(z)V(x-\delta z)dz\lesssim&\,\int_{\mathbb{R}^{d}}(1-\chi(\tfrac{\zeta}{\delta}))\left\vert K^{1}_\delta \right\vert(\zeta)(\left|x-\zeta\right|^{k}+\left|\zeta\right|^{k})V(x-\zeta)d\zeta\\
\lesssim&\, \frac{1}{R^{l}}+V\star K(x),
\end{align*}
where we have set $K(\zeta)\coloneqq \left\vert K^{1}_\delta\right\vert(\zeta)\left|\zeta\right|^{k}$. Thanks to Young's inequality it holds that 
\begin{align*}
\left\Vert V\star K\right\Vert _{\infty}\leq\left\Vert V\right\Vert _{L^{p}+L^{q}}\left(\left\Vert K\right\Vert _{p'}+\left\Vert K\right\Vert _{q'}\right)\lesssim &\,\delta^{k-\frac{d}{p}}\left(\int_{\mathbb{R}^{d}}\left\vert K^{1}\right\vert^{p'}(\xi)\left|\xi\right|^{kp'}d\xi\right)^{\frac{1}{p'}}\\
&+\delta^{k-\frac{d}{q}}\left(\int_{\mathbb{R}^{d}}\left\vert K^{1}\right\vert^{q'}(\xi)\left|\xi\right|^{kq'}d\xi\right)^{\frac{1}{q'}}\lesssim \delta^{k-\frac{d}{p}}+\delta^{k-\frac{d}{q}}.
\end{align*}
Therefore, we find 
\begin{equation*}
\left|x\right|^{k}\int_{\mathbb{R}^{d}}(1-\chi(z))\left\vert K^{1}\right\vert(z)V(x-\delta z)dz\lesssim \frac{1}{R^{l}}+\delta^{k-\frac{d}{p}}+\delta^{k-\frac{d}{q}}\lesssim \frac{1}{R^{l}}+\delta^{k-\frac{d}{p}},
\end{equation*}
using $p<\frac{d}{k}<q$, and hence for $\left|x\right|\geq R$, we have
\begin{align*}
\int_{\mathbb{R}^{d}}(1-\chi(z))\left\vert K^{1}\right\vert(z)V(x-\delta z)dz\lesssim \frac{V(x)}{R^{l}}+\frac{\delta^{k-\frac{d}{p}}}{\left|x\right|^{k}}\lesssim \frac{1}{R^{l+k}}+\frac{1}{\delta^{\frac{d}{p}-k}R^{k}}.    
\end{align*}
Therefore we conclude that
\begin{equation}
T_{1}\lesssim \frac{1}{R^{l+k}}+\frac{1}{\delta^{\frac{d}{p}-k}R^{k}}, \ \mbox{for } \left\vert x \right\vert \geq R. 
\label{eq:-36}
\end{equation}
If $\left|x\right|\leq R$, we estimate 
\begin{align*}
T_{1}&\lesssim e^{-\frac{R^{2}}{2}}\int_{\left|x-\delta z\right|\leq\frac{3\left|x\right|}{2}}V(x-\delta z)dz+\frac{V(x)}{R^{l}}=\frac{e^{-\frac{R^{2}}{2}}}{\delta^{d}}\int_{\left|x-\zeta\right|\leq\frac{3\left|x\right|}{2}}V(x-\zeta)d\zeta+\frac{V(x)}{R^{l}}\\    
&=\frac{e^{-\frac{R^{2}}{2}}}{\delta^{d}}\int_{\left|\xi\right|\leq\frac{3\left|x\right|}{2}}V(\xi)d\xi+\frac{V(x)}{R^{l}}\lesssim\frac{e^{-\frac{R^{2}}{2}}}{\delta^{d}}\left|x\right|^{d-k}+\frac{1}{R^{l}\left\vert x \right\vert^{k}},
\end{align*}
so that we conclude that
\begin{equation}
T_{1}\lesssim \frac{R^{d-k}e^{-\frac{R^{2}}{2}}}{\delta^{d}}+\frac{1}{R^{l}\left\vert x \right\vert^{k}}, \ \mbox{for } \left\vert x \right \vert\leq R.\label{eq:-27}     
\end{equation}
Note that in the proof \eqref{ine''},\eqref{eq:-36} and \eqref{eq:-27} we did not employ the assumption $\left\vert x \right\vert\geq 2R\delta$, a fact which will be useful later on. We now gather inequalities (\ref{ine'}), (\ref{ine''}), (\ref{eq:-36}) and (\ref{eq:-27}) to obtain  
\begin{align}
\left|K_{\delta}^{1}\star V(x)-V(x)\right|\lesssim \frac{\delta R}{\left|x\right|^{k+1}}+\frac{1}{R^{l+k}}+\frac{1}{\delta^{\frac{d}{p}-k}R^{k}}+\frac{R^{d-k}e^{-R^{2}}}{\delta^{d}}+\frac{1}{R^{l}\left\vert x \right\vert^{k}}. \label{prefinal step1}   
\end{align}
We can pick $\lambda'\in(0,1)$ small enough so that $2k-\frac{d}{p}-\lambda'>0>k-\frac{d}{p}-\lambda'$ assuming that $\tfrac{d}{2k}<p<\tfrac{d}{k}$. We observe that 
\[
\frac{1}{\delta^{\frac{d}{p}-k}R^{k}}=\delta^{\lambda'},\ \delta R=\delta^{2-\tfrac{d}{kp}-\tfrac{\lambda'}{k}} \quad \mbox{by setting } R=R(\delta)=\delta^{1-\tfrac{d}{kp}-\tfrac{\lambda'}{k}}.
\]
Notice that there is some $\beta'>0$ such that the function $\delta \mapsto \frac{R^{d-k}(\delta)e^{-R^{2}(\delta)}}{\delta^{d}}$ satisfies 
$\frac{R^{d-k}(\delta)e^{-R^{2}(\delta)}}{\delta^{d}}\leq \delta^{\beta'}$ when $\delta\to 0$, since both $\frac{R^{d-k}(\delta)}{\delta^{d}}$ and $R^{2}(\delta)$ are negative powers of $\delta$ under our assumptions.
Then, owing to (\ref{prefinal step1}) we get  
\[
\left|K_{\delta}^{1}\star V(x)-V(x)\right|\lesssim \frac{\delta^{\alpha}}{\left|x\right|^{k+1}}+\delta^{\beta}+\frac{\delta^{\gamma}}{\left\vert x \right\vert^{k}},
\]
for some $\alpha>0, \beta>0, \gamma>0$, which implies that for a suitable choice of $\lambda>0$ it holds that 
\begin{equation*}
K_{\delta}^{1}\star V(x)\leq V(x)+C\delta^{\lambda},\ \left|x\right|\geq \delta^{\lambda},
\end{equation*}
as desired. Notice that we can assume that $\lambda <1$  without loss of generality since $\delta$ will be chosen converging to zero. 

\vskip 12pt

\textbf{Step 2}.  In this step, we prove that 
\begin{align}
K_{\delta}^{1}\star V(x)\leq C\left(V(x)+\delta^{\lambda}\right),\ \mbox{for all } x\in\mathbb{R}^{d}. \label{step 2 main ine}     
\end{align}
Note that in \eqref{step 2 main ine} we allow the estimate to hold  up to a constant (unlike inequality \eqref{Eq 2Rdel}). Assume first that $\left\vert x \right\vert\leq 2R\delta$. The following inequality
follows by a direct calculation 
\begin{equation}
\int_{B(0,\delta)}V(y)dy\lesssim \int_{\frac{\delta}{2}\leq\left|y\right|\leq\delta}V(y)dy,  \label{int on annulus}   
\end{equation}
for $0<k<d-1$.
We now split $K^{1}_{\delta}\star V$ as 
\begin{align*}
 K^{1}_{\delta}\star V(x)=\int_{\mathbb{R}^{d}}\chi\left(\frac{x-z}{\delta}\right)K^{1}_\delta\left(x-z\right)V(z)dz+\int_{\mathbb{R}^{d}}\left(1-\chi\left(\frac{x-z}{\delta}\right)\right)K^{1}_\delta\left(x-z\right)V(z)dz.   
\end{align*}
The first term can be estimated as
\begin{align*}
\int_{\mathbb{R}^{d}}\chi\left(\frac{x-z}{\delta}\right)K^{1}_\delta\left(x-z\right)V(z)dz\leq\int_{\left|z\right|\leq3R\delta}\left\vert K^{1}_\delta\right\vert\left(x-z\right)V(z)dz
=\int_{A\cup B}\left\vert K^{1}_\delta\right\vert\left(x-z\right)V(z)dz,
\end{align*}
where $A=\{z\in\mathbb{R}^{d}: \left|z\right|\leq3R\delta,\left|z\right|\leq\frac{1}{2}\left|x\right|\}$ and $B=\{z\in\mathbb{R}^{d}: \frac{1}{2}\left|x\right|\leq\left|z\right|\leq3R\delta\}$.
If $\left|z\right|\leq\frac{1}{2}\left|x\right|$
then $\left|x-z\right|\geq\frac{1}{2}\left|x\right|$ so that 
\begin{align*}
\int_{A} \left\vert K^{1}_\delta\right\vert\left(x-z\right)V(z)dz&\leq\frac{\exp\left(-\frac{\left|x\right|^{2}}{2\delta^{2}}\right)}{\delta^{d}}\int_{A}V(z)dz \leq\frac{\exp(-\frac{\left|x\right|^{2}}{2\delta^{2}})}{\delta^{d}}\int_{\left|z\right|\leq\frac{1}{2}\left|x\right|}V(z)dz\\
&\lesssim \frac{\exp(-\frac{\left|x\right|^{2}}{2\delta^{2}})}{\delta^{d}}\int_{\frac{1}{4}\left|x\right|\leq\left|z\right|\leq\frac{1}{2}\left|x\right|}V(z)dz \lesssim \exp\left(-\frac{\left|x\right|^{2}}{2\delta^{2}}\right)\frac{\left|x\right|^{d}}{\delta^{d}}V(x),
\end{align*}
where we used \eqref{int on annulus}. 
Note that the function $x\mapsto\exp(-\frac{\left|x\right|^{2}}{2\delta^{2}})\frac{\left|x\right|^{d}}{\delta^{d}}$
is uniformly bounded in $\delta$. Therefore we conclude that 
\begin{equation*}
\int_{A}\left\vert K^{1}_\delta\right\vert\left(x-z\right)V(z)dz\lesssim V(x).
\end{equation*}
In addition, we obtain
\[
\int_{B}\left\vert K^{1}_\delta\right\vert\left(x-z\right)V(z)dz\lesssim V(x)\int_{\left|z\right|\leq3R\delta}\left\vert K^{1}_\delta\right\vert\left(x-z\right)dz\lesssim V(x).
\]
Thus, we have proved that 
\begin{equation}
\int_{\mathbb{R}^{d}}\chi\left(\frac{x-z}{\delta}\right)\left\vert K^{1}_\delta\right\vert\left(x-z\right)V(z)dz\lesssim V(x).\label{eq:-37}
\end{equation}
Next we handle the truncation far from the origin. By exactly the same argument leading to \eqref{ine''},\eqref{eq:-36} and \eqref{eq:-27} we
have 
\begin{align}
\int_{\mathbb{R}^{d}}\left(1-\chi\left(\frac{x-z}{\delta}\right)\right)K^{1}_\delta\left(x-z\right)V(z)dz &=\int_{\mathbb{R}^{d}}(1-\chi(z))K^{1}(z)V(x-\delta z)dz \notag\\
&\lesssim V(x)+\delta^{\lambda}+\delta^{\alpha}.\label{eq:-38}    
\end{align}
In view of (\ref{eq:-37})
and (\ref{eq:-38}), we deduce $K_{\delta}^{1}\star V(x)\lesssim V(x)$ for $\left\vert x \right\vert\leq 2R\delta$.
We are left to treat the range $2R\delta\leq \left\vert x\right\vert\leq \delta^{\lambda}$.  If $\left|x\right|\geq2R\delta$
and $\left|z\right|\leq R$ then clearly 
\[
\frac{\left|x\right|}{2}=\left|x\right|-\frac{\left|x\right|}{2}\leq\left|\left|x\right|-R\delta\right|\leq\left|x-\delta z\right|\leq2\left|x\right|.
\]
Thus, we infer that
\begin{equation*}
\int_{\mathbb{R}^{d}}\chi(z)\left\vert K^{1}\right\vert(z)V(x-\delta z)dz\lesssim V(\tfrac{x}{2})\int_{\left|z\right|\leq R}\left\vert K^{1}\right\vert(z)dz\lesssim V(x).
\end{equation*}
In addition, by step 2, we have 
\begin{align*}
 \left\vert\int_{\mathbb{R}^{d}}(1-\chi(z))K^{1}(z)V(x-\delta z)dz\right\vert\lesssim V(x)+\delta^{\lambda}+\delta^{\alpha},    
\end{align*}
for $\left\vert x \right\vert\geq 2R\delta$.
Together with inequality \eqref{Eq 2Rdel} in step 1, this establishes \eqref{step 2 main ine} by redefining $\lambda$ if needed since $\lambda <1$ and $\delta$ will be chosen converging to zero.

\vskip 12pt

\textbf{Step 3}. We claim that there is a non-decreasing function $f$ such that 
\begin{align}
K^{1}_{\delta}\star V(x) \leq V(x), \ \left\vert x \right\vert\leq f(\delta). \label{x less than f}   
\end{align} 
In order to prove the claim, we first estimate it as
\begin{align*}
 K^{1}_{\delta}\star V(x)\leq \int_{\left\vert z\right\vert\leq 1}V(x-\delta z)\left\vert K^{1} \right\vert(z)dz+\int_{\left\vert z\right\vert\geq 1}V(x-\delta z)\left\vert K^{1} \right\vert(z)dz\coloneqq I+J.
 \end{align*}
 By change of variables, we have 
 \begin{align*}
I\leq \frac{1}{\delta^{d}}\int_{B_{\delta }
(x)} V(z)dz,    
 \end{align*}
and since the function $x \mapsto \int_{B_{\delta}(x)}V(z)dz$ attains its maximum at $x=0$\footnote{ For each $x,y$ with $\left\vert x\right\vert< \left\vert y\right\vert$  it holds that $\left\vert x-z\right\vert\leq \left\vert y-z\right\vert $ for all $z\in B_{\delta}(0)$ and $\delta>0$ sufficiently small. Consequently, given a radially decreasing function $V(z)=v(\left\vert z\right\vert)$ for each $\left\vert x\right\vert< \left\vert y\right\vert $ and $\delta>0$ small enough it holds that 
\begin{align*}
\int_{B_{\delta}(x)}V(z)dz-\int_{B_{\delta}(y)}V(z)dz=\int_{B_{\delta}(0)}\left(v(\left\vert x-z\right\vert)-v(\left\vert y-z\right\vert)\right)dz\geq 0.         
\end{align*}}, it follows that 
\begin{align*}
I\leq \frac{1}{\delta^{d}}\int_{B_{\delta }(0)}V(z)dz= \frac{1}{(d-k)\delta^{k}}.
\end{align*}
Hence, for all $\left\vert x \right\vert\leq \left(\tfrac{d-k}{2}\right)^{\frac{1}{k}}\delta$, we get
\begin{align*}
 I\leq \frac{1}{2\left\vert x\right\vert^{k}}=\frac{1}{2}V(x).       
\end{align*}
In addition, we obtain
\begin{align*}
J\leq \frac{1}{( \delta-\left\vert x\right\vert)^{k}}\leq \frac{1}{2\left\vert x \right\vert^{k}},   
\end{align*}
for all $\left\vert x \right\vert\leq \frac{\delta}{1+2^{\frac{1}{k}}}$,
so that $J\leq \tfrac{1}{2}V(x)$.

To conclude, the claim follows by taking $f(\delta)=\min\left\{\frac{1}{1+2^{\frac{1}{k}}},\left(\frac{d-k}{2}\right)^{\frac{1}{k}}\right\}\delta$. 

\vskip 12pt

\textbf{Step 4}. \textbf{ Construction of} $V_{\varepsilon}$. In this step we will construct $V_{\varepsilon}$ and then show that it satisfies the requested properties i.-iv. We construct our regularized kernel following the same strategy as in \cite[Lemma 4.1]{bresch2019modulated}.  

We now need to interpolate between this inequality near the origin and the far-field inequality \eqref{Eq 2Rdel} by means of the inequality \eqref{step 2 main ine}. This is crucially needed in order to keep the constant 1 in front of $V(x)$ on the right-hand side of our statement in point iv. As mentioned the construction is motivated by \cite[Lemma 4.1]{bresch2019modulated} proved in the torus.

Given $\varepsilon>0$ pick $0<\delta_{1}\leq \frac{\varepsilon^{\frac{2}{\lambda}}}{C}$, where $C>0$ is the constant in \eqref{step 2 main ine}, and recursively pick $0<\delta_{i+1}$ such that 
\[
\delta_{{i+1}}\leq\min\left\{ f(\delta_{{i}}),\delta_{{i}}^{\frac{1}{\lambda}}\right\} .
\]
Let $M\coloneqq\left\lfloor \frac{1}{\varepsilon}\right\rfloor $ where $\lfloor x\rfloor$ denotes the closest lower integer to $x$,
and set 
\begin{equation}\label{aux55}
W_{\varepsilon}\coloneqq\frac{1}{M}\stackrel[i=1]{M}{\sum}K^{1}_{\delta_{{i}}}\star V.
\end{equation}
Obviously $\delta_{{i}}$ is decreasing since $\lambda<1$ and  $\underset{1\leq i\leq M}{\max}\delta_{{i}}\underset{\varepsilon\rightarrow0}{\rightarrow}0$ by construction. 

If $\left|x\right|\geq\delta_{1}^\lambda$, \eqref{Eq 2Rdel} implies that 
\begin{align*}
W_{\varepsilon}(x)\leq V(x)+\varepsilon.     
\end{align*}
\indent If $\left|x\right|\leq f(\delta_{{M}})$, \eqref{x less than f} implies that $W_{\varepsilon}(x)\leq V(x)$. 

If $f(\delta_{{M}})\leq\left|x\right|\leq\delta_{1}^{\lambda}$, fix $1\leq i\leq M$ such that $\delta_{{i+1}}\leq\left|x\right|\leq\delta_{{i}}$, then:
\begin{itemize}
    \item If $j>i+1$ then $\left|x\right|\geq\delta_{{i+1}}\geq\delta_{{j}}^{\lambda}$
so that using again \eqref{Eq 2Rdel} we get 
$
K_{\delta_{{j}}}^{1}\star V(x)\leq  V(x)+\varepsilon.
$
\item If $j<i$ then $\left|x\right|\leq\delta_{{i}}\leq f(\delta_{{j}})$
so that again by \eqref{x less than f} we have 
$
K_{\delta_{{j}}}^{1}\star V(x)\leq V(x).
$

\item If $j=i$ or $j=i+1$ then according to  \eqref{step 2 main ine}
$
K_{\delta_{{j}}}^{1}\star V(x)\leq C\left(V(x)+\varepsilon \right).
$
\end{itemize}
We finally define 
\begin{equation}\label{aux56}
V_{\varepsilon}(x)\coloneqq\frac{1}{1+2C\varepsilon}W_{\varepsilon}(x)
\end{equation} 
and conclude that 
$V_{\varepsilon}(x)\leq  V(x)+\varepsilon$,
which establishes point iv. Notice that the definition of $M$ is crucially used for this estimate.

That $\widehat{V_{\varepsilon}}\geq0$ is immediate by construction by the choice of $K^1$ and establishes point i. We are left to prove point ii,iii and v. If we split $V=V_1+V_2$ where $V_1\in L^{p}(\mathbb{R}^{d})$ and $V_2\in L^{q}(\mathbb{R}^{d})$ then 
\begin{align*}
\left\Vert W_{\varepsilon}-V\right\Vert _{L^{p}+L^{q}(\mathbb{R}^{d})}&\leq\underset{1\leq i\leq M}{\max}\left\Vert K_{\delta_{{i}}}\star V_{1}-V_{1}\right\Vert _{L^{p}(\mathbb{R}^{d})}+\underset{1\leq i\leq M}{\max}\left\Vert K_{\delta_{{i}}}\star V_{2}-V_{2}\right\Vert _{L^{q}(\mathbb{R}^{d})}\\
&\leq2\underset{1\leq i\leq M}{\max}\delta_{{i}}\underset{\varepsilon\rightarrow0}{\longrightarrow}0,
\end{align*}
which is assertion ii. 
Note further that for a given $\delta>0$ one has 
\begin{align*}
\left\Vert \zeta_{\delta}(K^{1}_{\delta_{{i}}}\star V)-\zeta_{\delta}V\right\Vert _{L^{p}+L^{q}(\mathbb{R}^{d})}\leq&\left\Vert K^{1}_{\delta_{{i}}}\star(\zeta_{\delta}V)-\zeta_{\delta}V\right\Vert _{L^{p}+L^{q}(\mathbb{R}^{d})}+\left\Vert \zeta_{\delta}(K^{1}_{\delta_{{i}}}\star((1-\zeta_{\delta})V))\right\Vert _{L^{p}+L^{q}(\mathbb{R}^{d})}\\&+\left\Vert (\zeta_{\delta}-1)(K^{1}_{\delta_{{i}}}\star(\zeta_{\delta}V))\right\Vert _{L^{p}+L^{q}(\mathbb{R}^{d})}\\\coloneqq& I+II+III.    
\end{align*}
It is not difficult to check that $\left\Vert K^{1}_{\delta}\star f-f\right\Vert_{q}\lesssim_{d}\delta\left\Vert  \nabla f\right\Vert_{q}$ by using that the first moment of $K^{1}_{\delta}$ is of order $\delta$. Therefore, the first term is 
\begin{align*}
I\leq\left\Vert K^{1}_{\delta_{i}}\star\zeta_{\delta}V-\zeta_{\delta}V\right\Vert _{L^{q}(\mathbb{R}^{d})}\lesssim\underset{1\leq i\leq M}{\max}\delta_{i}\left\Vert \nabla(\zeta_{\delta}V)\right\Vert _{q}\leq\underset{1\leq i\leq M}{\max}\delta_{i}\left(\left\Vert \zeta_{\delta}\nabla V\right\Vert _{q}+\left\Vert V\nabla\zeta_{\delta}\right\Vert _{q}\right).    
\end{align*}
Observe that 
\begin{align*}
\left\Vert \zeta_{\delta}\nabla V\right\Vert _{q}\lesssim \left(\int_{\left|\cdot\right|\geq\delta}\frac{1}{\left|x\right|^{q(k+1)}}dx\right)^{\frac{1}{q}}&=\left(\int_{\delta}^{\infty}r^{d-1-q(k+1)}dr\right)^{\frac{1}{q}}\lesssim \frac{\delta^{\frac{d}{q}}}{\delta^{k+1}}.
\end{align*}
Also, a similar calculation reveals that 
\begin{equation}\label{comphomog}
    \left\Vert V\nabla\zeta_{\delta}\right\Vert _{q}\leq\left\Vert \nabla\zeta_{\delta}\right\Vert _{\infty}\left\Vert V\right\Vert _{L^{q}(B_{2\delta}\setminus B_{\delta})}\lesssim\left\Vert \nabla\zeta_{\delta}\right\Vert _{\infty}\frac{\delta^{\frac{d}{q}}}{\delta^{k}}\leq\frac{\delta^{\frac{d}{q}}}{\delta^{k+1}},
\end{equation}
hence $I\lesssim {\varepsilon}{\delta^{\frac{d}{q}-1-k}}$.
Moreover 
\begin{align*}
II=\left\Vert \zeta_{\delta}(K^{1}_{\delta_{i}}\star((1-\zeta_{\delta})V))\right\Vert _{L^{p}+L^{q}(\mathbb{R}^{d})}&\leq\left\Vert K^{1}_{\delta_{i}}\star((1-\zeta_{\delta})V)\right\Vert _{L^{p}(\mathbb{R}^{d})}\\
&\leq\left\Vert K^{1}_{\delta_{i}}\right\Vert _{L^{1}(\mathbb{R}^{d})}\left\Vert (1-\zeta_{\delta})V\right\Vert _{L^p(\mathbb{R}^{d})}\lesssim \delta^{\frac{d}{p}-k}.    
\end{align*}
As for $III$, we estimate 
\begin{align*}
    III\leq \left\Vert (\zeta_{\delta}-1)(K^{1}_{\delta_{i}}\star\zeta_{\delta}V)\right\Vert_{L^{p}(\mathbb{R}^{d})}&\leq   \left\Vert \zeta_{\delta}-1\right\Vert_{L^{p}(\mathbb{R}^{d})} \left\Vert K^{1}_{\delta_{i}}\star\zeta_{\delta}V\right\Vert_{L^{\infty}(\mathbb{R}^{d})}\\
    &\lesssim \delta^{\frac{d}{p}}\left\Vert \zeta_{\delta}V\right\Vert_{L^{\infty}(\mathbb{R}^{d})}\leq  \delta^{\frac{d}{p}-k}.
\end{align*}
Therefore, we have proved 
\begin{align*}
 \left\Vert \zeta_{\delta}K_{\delta_{n_{i}}}\star V-\zeta_{\delta}V\right\Vert _{L^{p}+L^{q}(\mathbb{R}^{d})}\lesssim\frac{\varepsilon}{\delta^{k+1-\frac{d}{q}}}+\delta^{\frac{d}{p}-k},  
\end{align*}
which concludes the proof of item iii. We finally have 
\begin{align*}
\left\Vert K^{1}_{\delta}\star V\right\Vert_{\infty}\leq \left\Vert K^{1}_{\delta}\right\Vert_{p'} \left\Vert \mathbf{1}_{\left\vert \cdot\right\vert\leq 1}V\right\Vert_{p}+\left\Vert K^{1}_{\delta}\right\Vert_{q'} \left\Vert \mathbf{1}_{\left\vert \cdot\right\vert\geq 1}V\right\Vert_{q}   \leq \delta^{-\frac{d}{p}}+ \delta^{-\frac{d}{q}}
\end{align*}
so that by defining
\begin{align*}
\left\Vert V_{\varepsilon}\right\Vert_{\infty}\leq \left(\min_{1\leq i\leq M} \delta_{i}\right)^{-\frac{d}{p}}+ \left(\min_{1\leq i\leq M} \delta_{i}\right)^{-\frac{d}{q}}:=\frac{1}{\Theta(\varepsilon)},
\end{align*}
the statement v. follows.
\qed

\

In the following Lemma we show that given a function in Fourier variable $s(\xi)$ satisfying some growth assumptions we can bound  $s\star\widehat{V_{\varepsilon}}$  pointwise in absolute value by means of $\widehat{V_{\varepsilon}}$, where $V_{\varepsilon}$ is the approximating kernel constructed in Lemma \ref{Regularization }.

\begin{lem}
\label{bound on convolution} Let $V_{\varepsilon}$ be as in Lemma \ref{Regularization }. Assume that there are some $S_{0}>0$ and some $\sigma>2d$ such that 
\begin{align}\label{aux3}
  \left\vert s(\xi)\right\vert\leq \frac{S_{0}}{(1+\left\vert \xi \right\vert)^{\sigma}} .
\end{align}
Given $r>\frac{d}{d-k}$ set $S_{0}'\coloneqq \max\left\{\left\Vert s\right\Vert_{1},\left\Vert \left\vert \cdot\right\vert^{d-k}s\right\Vert_{r'},\left\Vert (1+\left\vert \cdot \right\vert^{d-k})s\right\Vert_{L^{\infty}}\right\}$.
Then, it holds that 
\[
\left\vert (s\star\widehat{V_{\varepsilon}})(\xi)\right\vert\leq C_{1}\widehat{V_{\varepsilon}}(\xi)+C_{2}\frac{o_{\varepsilon}(1)}{(1+\left\vert \xi \right\vert)^{\frac{\sigma}{2}}},
\]
where $C_{1}=c_{1}(k,d)S_{0}'$ and 
$C_{2}=c_{2}(k,\sigma,d)S_{0}$. 
\end{lem}
\begin{proof}
We separate between two regions. Put $\overline{\delta}\coloneqq\underset{1\leq i\leq M}{\max}\delta_{i}\leq \varepsilon$.

\textbf{Step 1}. Suppose that $\left|\xi\right|<\frac{1}{\overline{\delta}}$.
Note that for all $\xi \in \mathbb{R}^d$
\begin{equation}
\widehat{V_{\varepsilon}}(\xi)=\frac{1}{M(1+2C\varepsilon)}\stackrel[i=1]{M}{\sum}\widehat{K^{1}_{\delta_{{i}}}}(\xi)\widehat{V}(\xi)\leq\left\Vert \widehat{K^{1}}\right\Vert _{\infty}\widehat{V}(\xi).\label{eq:-33}
\end{equation}
Fix some $\chi\in C_{0}^{\infty}(\mathbb{R}^{d})$ with $0\leq \chi \leq 1$ and  $\chi\equiv1$
on $B_{1}(0)$ and $\chi\equiv0$ on $\mathbb{R}^{d}\setminus B_{2}(0)$.
In view of  \eqref{eq:-33} we have 
\begin{align}
&\left\vert \int_{\mathbb{R}^{d}}\widehat{V_{\varepsilon}}(\xi-\zeta)s(\zeta)d\zeta\right\vert \notag \\
&\leq\left\Vert \widehat{K^{1}}\right\Vert _{\infty}\int_{\mathbb{R}^{d}}\widehat{V}(\xi-\zeta)\left\vert s (\zeta)\right\vert d\zeta   \notag\\&=\left\Vert \widehat{K^{1}}\right\Vert _{\infty}\left(\int_{\mathbb{R}^{d}}\chi(\xi-\zeta)\widehat{V}(\xi-\zeta)\left\vert s(\zeta)\right\vert d\zeta 
+\int_{\mathbb{R}^{d}}(1-\chi(\xi-\zeta))\widehat{V}(\xi-\zeta)\left\vert s (\zeta)\right\vert d\zeta\right). \label{eq:-32}
\end{align}
We control each of the integrals in \eqref{eq:-32} separately.  Thanks to \eqref{aux3} one has 
\begin{align*}
\left|\xi\right|^{d-k}\int_{\mathbb{R}^{d}}\chi(\zeta)\widehat{V}(\zeta)\left\vert s (\xi-\zeta)\right\vert d\zeta\lesssim& \int_{\mathbb{R}^{d}}\chi(\zeta)\widehat{V}(\zeta)\left|\xi-\zeta\right|^{d-k}\left\vert s (\xi-\zeta)\right\vert d\zeta\\
&+\int_{\mathbb{R}^{d}}\chi(\zeta)\widehat{V}(\zeta)\left|\zeta\right|^{d-k}\left\vert s (\xi-\zeta)\right\vert d\zeta\\
\lesssim&\, \left\Vert (1+\left\vert \cdot \right\vert^{d-k})s\right\Vert_{\infty}\int_{\mathbb{R}^{d}}(1+\left|\zeta\right|^{d-k})\chi(\zeta)\widehat{V}(\zeta)d\zeta\\
\lesssim& \left\Vert (1+\left\vert \cdot \right\vert^{d-k})s\right\Vert_{\infty}\left\Vert \left(1+\left|\cdot\right|^{d-k}\right)\widehat{V}(\cdot)\right\Vert _{L^{1}(B_{2}(0))}\lesssim \left\Vert (1+\left\vert \cdot \right\vert^{d-k})s\right\Vert_{\infty},
\end{align*}
which shows that 
\begin{equation}
\int_{\mathbb{R}^{d}}\chi(\xi-\zeta)\widehat{V}(\xi-\zeta)\left\vert s(\zeta)\right\vert d\zeta\lesssim\frac{\left\Vert (1+\left\vert \cdot \right\vert^{d-k})s\right\Vert_{\infty}}{\left|\xi\right|^{d-k}}\lesssim \left\Vert (1+\left\vert \cdot \right\vert^{d-k})s\right\Vert_{\infty}\widehat{V}(\xi).\label{eq:-34}
\end{equation}
Secondly, we can estimate the truncation far from origin as follows 
\begin{align}
\left|\xi\right|^{d-k}\int_{\mathbb{R}^{d}}(1-\chi(\xi-\zeta))\widehat{V}(\xi-\zeta)\left\vert s(\zeta)\right\vert d\zeta\lesssim&\int_{\mathbb{R}^{d}}\left|\xi-\zeta\right|^{d-k}(1-\chi(\xi-\zeta))\widehat{V}(\xi-\zeta)\left\vert s(\zeta)\right\vert d\zeta \notag\\
&+\int_{\mathbb{R}^{d}}\left|\zeta\right|^{d-k}(1-\chi(\xi-\zeta))\widehat{V}(\xi-\zeta)\left\vert s(\zeta)\right\vert d\zeta \notag\\
\lesssim&\int_{\mathbb{R}^{d}}\left\vert s(\zeta)\right\vert d\zeta+\int_{\mathbb{R}^{d}}\left|\xi-\zeta\right|^{d-k}\left\vert s(\xi-\zeta)\right\vert (1-\chi(\zeta))\widehat{V}(\zeta)d\zeta
\notag\\
\lesssim& 
\left\Vert s\right\Vert_{1}+\left\Vert (1-\chi)\widehat{V}\right\Vert_{r} \left\Vert \left\vert \cdot\right\vert^{d-k}s\right\Vert_{r'}. \label{eq:-35}
\end{align}
Choosing $r>\frac{d}{d-k}$ ensures that $\left\Vert (1-\chi)\widehat{V}\right\Vert_{r}<\infty$. Therefore we conclude from  \eqref{eq:-34}, \eqref{eq:-35}
that 
\begin{align*}
\int_{\mathbb{R}^{d}}\widehat{V_{\varepsilon}}(\xi-\zeta)\left\vert s \right\vert(\zeta)d\zeta\lesssim S'_{0}\widehat{V}(\xi). 
\end{align*}
Now, utilizing the assumption that  $\left|\xi\right|<\frac{1}{\overline{\delta}}$ we see that $\widehat{K^{1}_{\delta_{{i}}}}(\xi)\gtrsim 1 $ for all $1\leq i \leq M$, so that we find 
$\widehat{V}(\xi)\lesssim \widehat{V}_{\varepsilon}(\xi)$,
for $\left|\xi\right|<\frac{1}{\overline{\delta}}$, and thus we get
\begin{align} 
\left\vert s\star\widehat{V_{\varepsilon}}(\xi) \right\vert\lesssim S'_{0}\widehat{V_{\varepsilon}}(\xi), \mbox{ for } \left|\xi\right|<\frac{1}{\overline{\delta}}. \label{xi less than 1/delta}   
\end{align}
\textbf{Step 2}. Suppose now $\left|\xi\right|\geq\frac{1}{\overline{\delta}}$. 
In this case we split the integral as follows 
\begin{align*}
\int_{\mathbb{R}^{d}}\widehat{V}_{\varepsilon}(\xi-\zeta)s(\zeta)d\zeta 
= \int_{\left\vert \zeta \right\vert\geq \frac{\left\vert \xi \right\vert}{2}}\widehat{V}_{\varepsilon}(\xi-\zeta)s(\zeta)d\zeta+\int_{\left\vert \zeta \right\vert < \frac{\left\vert \xi \right\vert}{2}}\widehat{V}_{\varepsilon}(\xi-\zeta)s(\zeta)d\zeta\coloneqq I(\xi)+J(\xi).  
\end{align*}
Pick a function $0\leq \chi \leq 1$ such that $\chi\equiv1$ on $\left|\zeta\right|\leq\frac{1}{4\overline{\delta}}$
and $\chi\equiv0$ on $\left|\zeta\right|\geq\frac{1}{2\overline{\delta}}$.
We start with $I(\xi)$. We can write 
\begin{align*}
I(\xi)&=\int_{\left\vert \zeta \right\vert \geq \frac{\left\vert \xi \right\vert}{2} }\widehat{V_{\varepsilon}}(\xi-\zeta)(1-\chi(\zeta))s(\zeta)d\zeta 
\\&=\int_{\left\vert \zeta \right\vert \geq \frac{\left\vert \xi \right\vert}{2} }\mathbf{1}_{\left\vert \cdot \right\vert\leq 1}(\xi-\zeta)\widehat{V_{\varepsilon}}(\xi-\zeta)(1-\chi(\zeta))s(\zeta)d\zeta+\int_{\left\vert \zeta \right\vert \geq \frac{\left\vert \xi \right\vert}{2} }\mathbf{1}_{\left\vert \cdot \right\vert\geq 1}(\xi-\zeta)\widehat{V_{\varepsilon}}(\xi-\zeta)(1-\chi(\zeta))s(\zeta)d\zeta \\ &\coloneqq I_{1}(\xi)+I_{2}(\xi). 
\end{align*}
First, we have 
\begin{align}
\left\vert I_{2}(\xi)\right\vert \lesssim \int_{\left\vert \zeta \right\vert \geq \frac{\left\vert \xi \right\vert}{2}} \left\vert (1-\chi(\zeta))s(\zeta)\right\vert d\zeta \lesssim  \frac{S_{0}}{(1+\left\vert \xi \right\vert)^{\frac{\sigma}{2}}}\int_{\left\vert \zeta \right\vert \geq \frac{1}{2\overline{\delta}}}\frac{1}{(1+\left\vert \zeta \right\vert)^{\frac{\sigma}{2}}} d\zeta\lesssim \frac{S_{0}o_{\varepsilon}(1)}{(1+\left\vert \xi \right\vert)^{\frac{\sigma}{2}}}.    \label{I_2(xi)} 
\end{align}
Secondly, we can deduce that 
\begin{align*}
\left(1+\left\vert \xi \right\vert^{\frac{\sigma}{2}}\right)\left\vert I_{1}(\xi) \right\vert \leq & \left(1+\left\vert \xi \right\vert^{\frac{\sigma}{2}}\right)\int_{\mathbb{R}^{d}}\mathbf{1}_{\left\vert \cdot \right\vert \leq 1}(\xi-\zeta)\widehat{V_{\varepsilon}}(\xi-\zeta)(1-\chi(\zeta))\left\vert s (\zeta)\right\vert d\zeta \notag\\ 
\lesssim& \int_{\mathbb{R}^{d}}\left\vert \xi-\zeta \right\vert^{\frac{\sigma}{2}}\mathbf{1}_{\left\vert\cdot \right\vert \leq 1}(\xi-\zeta)\widehat{V_{\varepsilon}}(\xi-\zeta)(1-\chi(\zeta))\left\vert s (\zeta)\right\vert d\zeta \notag\\&+\int_{\mathbb{R}^{d}}\mathbf{1}_{\left\vert\cdot \right\vert \leq 1}(\xi-\zeta)\widehat{V_{\varepsilon}}(\xi-\zeta)(1-\chi(\zeta))\left(1+\left\vert \zeta \right\vert^{\frac{\sigma}{2}}\right)\left\vert s (\zeta)\right\vert d\zeta \notag\\
\lesssim & \int_{\mathbb{R}^{d}}\mathbf{1}_{\left\vert \cdot \right\vert\leq 1}(\xi-\zeta)\left\vert \xi-\zeta \right\vert^{\frac{\sigma}{2}}\widehat{V}(\xi-\zeta)\left|(1-\chi(\zeta))s(\zeta)\right|d\zeta \notag\\
&+\int_{\mathbb{R}^{d}}\mathbf{1}_{\left\vert \cdot \right\vert\leq 1}(\xi-\zeta)\widehat{V}(\xi-\zeta)\left|(1-\chi(\zeta))\left(1+\left\vert \zeta \right\vert^{\frac{\sigma}{2}}\right)s(\zeta)\right|d\zeta.
\end{align*}
The right-hand side of the last equation is bounded by 
\begin{align}
&\int_{\left\vert \zeta\right\vert \geq \frac{1}{4\overline{\delta}}}\left\vert s(\zeta) \right\vert d\zeta+\int_{\mathbb{R}^{d}}\mathbf{1}_{\left\vert \cdot \right\vert\leq 1}(\xi-\zeta)\widehat{V}(\xi-\zeta)\left|(1-\chi(\zeta))\left(1+\left\vert \zeta \right\vert^{\frac{\sigma}{2}}\right)s(\zeta)\right| d\zeta \notag\\
&\lesssim 
 S_{0}o_{\varepsilon}(1)+\underset{\left\vert \zeta \right\vert \geq \frac{1}{4\overline{\delta}}}{\sup}\left\vert \left(1+\left\vert \zeta \right\vert^{\frac{\sigma}{2}}\right)s(\zeta)\right\vert\lesssim S_{0}o_{\varepsilon}(1).     
\label{first est for I_1}    
\end{align}
The combination of \eqref{I_2(xi)} and \eqref{first est for I_1} reveal that 
\begin{align}
\left\vert I(\xi) \right\vert \lesssim \frac{S_{0}o_{\varepsilon}(1)}{(1+\left\vert \xi \right\vert)^{\frac{\sigma}{2}}}.    \label{est for I}        
\end{align}
To bound $J(\xi)$, recall that by assumption, there is some $C>0$ and $r>0$ such that 
\begin{align*}
   \frac{1}{C(1+\left\vert \xi \right\vert^{r})}\leq \widehat{K^{1}}(\xi)\leq \frac{C}{1+\left\vert \xi \right\vert^{r}}. 
\end{align*}
As a result, for any $\zeta, \xi$ such that 
$\left\vert \zeta \right\vert < \frac{\left\vert \xi\right\vert}{2} $ it holds that 
\begin{align*}
\widehat{K^{1}_{\delta_i}}(\xi-\zeta)=\widehat{K^{1}}(\delta_i(\xi-\zeta))\lesssim \frac{1}{1+\delta_i^r\left\vert \xi -\zeta\right\vert^r}\lesssim \frac{1}{1+\delta_i^r\left\vert \xi \right\vert^r}\lesssim \widehat{K^{1}_{\delta_i}}(\xi).    
\end{align*}
Consequently, we infer that
\begin{align*}
\widehat{V_{\varepsilon}}(\xi-\zeta)=\frac{1}{M(1+2C\varepsilon)} \stackrel[i=1]{M}{\sum}\widehat{K^{1}}(\delta_{i}(\xi-\zeta))\widehat{V}(\xi-\zeta)\lesssim 
\frac{1}{M(1+2C\varepsilon)} \stackrel[i=1]{M}{\sum}\widehat{K^{1}}(\delta_{i}\xi)\widehat{V}(\xi)=\widehat{V}_{\varepsilon}(\xi),
\end{align*}
for $\left\vert \zeta \right\vert<\frac{\left\vert \xi\right\vert}{2}$.
Hence we obtain 
\begin{align}
\left\vert J(\xi) \right\vert \lesssim \widehat{V_{\varepsilon}}(\xi)\int_{\mathbb{R}^{d}}\left\vert s \right\vert(\zeta)d\zeta \lesssim S'_{0}\widehat{V}_{\varepsilon}(\xi). \label{est on J}   \end{align}
Gathering Inequalities \eqref{xi less than 1/delta}, \eqref{est for I} and \eqref{est on J} yields the announced result. 
\end{proof}

As a corollary we can estimate the Fourier transform of the $\delta$-truncation of $V_{\varepsilon}$ as follows.  

\begin{cor}
Under the same assumptions and notations of Lemma \ref{Regularization }, it holds that 
\[
\left\vert\widehat{\zeta_{\delta}V_{\varepsilon}}(\xi)\right\vert\lesssim \widehat{V_{\varepsilon}}(\xi)+\frac{\delta^{d-k-\frac{\sigma}{2}}o_{\varepsilon\delta^{-1}}(1)}{(1+\left|\xi\right|)^{\frac{\sigma}{2}}}
\] \label{cor FT}
and 
\[
\left\vert\widehat{(1-\zeta_{\delta})V_{\varepsilon}}(\xi)\right\vert\lesssim \widehat{V_{\varepsilon}}(\xi)+\frac{\delta^{d-k-\frac{\sigma}{2}}o_{\varepsilon\delta^{-1}}(1)}{(1+\left|\xi\right|)^{\frac{\sigma}{2}}}.
\]

\end{cor}
\begin{proof}
Set $\chi_{\delta}=1-\zeta_{\delta}$. Observe that  
$\widehat{\zeta_{\delta}V_{\varepsilon}}=\widehat{\zeta_{\delta}}\star\widehat{V_{\varepsilon}}=\widehat{V_{\varepsilon}}-\widehat{\chi_{\delta}}\star\widehat{V_{\varepsilon}}$,
which shows that the first inequality would follow from the second inequality. To prove the second inequality, notice that using the construction of $V_\varepsilon$ in \eqref{aux55}-\eqref{aux56}, we can deduce for each $\delta> 0$ that 
\begin{align*}
\widehat V_{\varepsilon}(\delta\xi)= \frac{1}{M(1+2C\varepsilon)} \sum_{i=1}^{d} \widehat {K^{1}}(\delta_i \delta \xi) \widehat {V}(\delta\xi)=\frac{\delta^{k-d}}{M(1+2C\varepsilon)} \sum_{i=1}^{d} \widehat {{K}_{\delta_i \delta}^{1}} (\xi) \widehat V(\xi)=\delta^{k-d} \frac{1+2C\delta\varepsilon}{1+2C\varepsilon}\widehat {V}_{\delta\varepsilon}(\xi).   
\end{align*}
Note that $\widehat{\chi_{\delta}}(x)=\delta^{d}\widehat{\chi}(\delta x)$. Therefore, we get
\begin{align*}
(\widehat{\chi_{\delta}}\star \widehat{V}_{\varepsilon})(\delta \xi)&=\int_{\mathbb{R}^{d}}\widehat{\chi}(\delta \xi -\zeta)\widehat{V}_{\varepsilon}\left(\delta^{-1}\zeta\right)d\zeta=\delta^{d-k}\frac{1+2C\delta^{-1}\varepsilon}{1+2C\varepsilon}\int_{\mathbb{R}^{d}}\widehat{\chi}(\delta \xi -\zeta)\widehat{V}_{\delta^{-1}\varepsilon}(\zeta)\,d\zeta\\&= \delta^{d-k} \frac{1+2C\delta^{-1}\varepsilon}{1+2C\varepsilon} (\widehat{\chi} \star \widehat{V}_{\delta^{-1}\varepsilon})(\delta \xi). 
\end{align*}
By using Lemma \ref{bound on convolution} with $s=\widehat{\chi}$, we conclude that 
\begin{align*}
\left\vert (\widehat{\chi_{\delta}}\star \widehat{V}_{\varepsilon})(\delta \xi)\right\vert&\leq \frac{1+2C\delta^{-1}\varepsilon}{1+2C\varepsilon} \left( C_{1}\delta^{d-k}\widehat{V}_{\delta^{-1}\varepsilon}(\delta \xi)+C_{2}\delta^{d-k}\frac{o_{\varepsilon \delta^{-1}}(1)}{(1+\left\vert \delta\xi \right\vert)^{\frac{\sigma}{2}}} \right)\\
&\leq C_{1}\widehat{V}_{\varepsilon}(\xi)+C_{2}\frac{\delta^{d-k-\frac{\sigma}{2}}o_{\varepsilon \delta^{-1}}(1)}{(1+\left\vert \xi \right\vert)^{\frac{\sigma}{2}}},    
\end{align*}
where we made use again of \eqref{aux55}-\eqref{aux56}.
\end{proof}

The following Lemma is the deterministic version of \cite[Lemma 5.1]{bresch2019modulated},
and it is meant to provide a bound on the truncated off-diagonal interaction
part in terms of the modulated energy. 

\begin{lem}\label{Bound on the non diag terms} 
Under the same assumptions and notations of Lemma \ref{Regularization }, let $\mu\in L^{1}\cap L^{\infty}(\mathbb{R}^d)$ be a probability density and let $\mu_N$ be an empirical measure defined by \eqref{empirical intro} associated to any $\mathbf{x}_{N}\in \mathbb{R}^{dN}\setminus \triangle_{N}$ and $\mathbf{m}_{N}\in\mathbb{M}^{N}$  such that $m_{i}^{N}\leq\overline{M}$ for some $\overline{M}>0$. Then, there is a constant $C=C(\left\Vert \mu\right\Vert _{\infty},d,k,p,q)$  such that the renormalized energy \eqref{renomenergyNconf} can be estimated from below as 
\begin{align}
\mathcal{E}_{N}(\mu,\mu_{N})=\int_{x\neq y}V(x-y)\left(\mu_{N}-\mu\right)^{\otimes 2}(dxdy)\geq-\frac{\overline{M}\left\Vert V_{\varepsilon}\right\Vert _{\infty}}{N}-C\left(\frac{\varepsilon}{\delta^{k+1-\frac{d}{q}}}+\delta^{\frac{d}{p}-k}+o_{\varepsilon}(1)\right) 
\label{1st identity}
\end{align}
and
\begin{align}
\frac{1}{N^{2}}\underset{i\neq j}{\sum}m_{i}m_{j}(1-\zeta_{\delta})V(x_{i}-x_{j})\leq C\left(\mathcal{E}_{N}(\mu,\mu_{N})+\mathcal{O}(\varepsilon,\delta,N)\right) \label{2nd identity}    
\end{align}
where $\mathcal{O}$ has the form 
\begin{align}
\mathcal{O}(\varepsilon,\delta,N)=\frac{\varepsilon}{\delta^{k+1-\frac{d}{q}}}+\delta^{d-k-\frac{\sigma}{2}}o_{\varepsilon\delta^{-1}}(1)+\frac{\overline{M}\left\Vert V_{\varepsilon} \right\Vert_{\infty}}{N}+\delta^{\frac{d}{p}-k}.   
\label{aux57}
\end{align}  
\end{lem}
\textit{Proof}. \textbf{Step 1}. In this step we prove \eqref{1st identity}. We denote by $\lesssim$ inequality up to a constant depending only on $\left\Vert \mu\right\Vert _{\infty},d,k,p,q$. We have 
\begin{align}
\int_{x\neq y}V(x-y)\left(\mu_{N}-\mu\right)^{\otimes2}(dxdy)=&\int_{\mathbb{R}^{d}\times\mathbb{R}^{d}}\zeta_{\delta}(x-y)V(x-y)\left(\mu_{N}-\mu\right)^{\otimes2}(dxdy)\notag\\&+\int_{x\neq y}\left(1-\zeta_{\delta}(x-y)\right)V(x-y)\left(\mu_{N}-\mu\right)^{\otimes2}(dxdy). \label{first identity in lemma 4.3} 
\end{align}
Observing that 
\begin{align}
\left\Vert ((1-\zeta_{\delta})V)\star \mu \right\Vert_{\infty} \lesssim \delta^{\frac{d-pk}{p}} \label{INFINITY NORM OF TRUN}    
\end{align}
by H\"older's inequality, we see that the second integral in the right-hand side is 
\begin{align*}
\int_{x\neq y}\left(1-\zeta_{\delta}(x-y)\right)V(x-y)\left(\mu_{N}-\mu\right)^{\otimes2}(dxdy)
=&\,\frac{1}{N^{2}}\underset{i\neq j}{\sum}m_{i}m_{j}V(x_{i}-x_{j})\left(1-\zeta_{\delta}(x_{i}-x_{j})\right)\\&-\frac{2}{N}\stackrel[i=1]{N}{\sum}m_{i}((1-\zeta_{\delta})V)\star\mu(x_{i})\\&-\int_{\mathbb{R}^{d}}\mu(x)((1-\zeta_{\delta})V)\star\mu(x)dx
\\\geq&\,\frac{1}{N^{2}}\underset{i\neq j}{\sum}m_{i}m_{j}V(x_{i}-x_{j})\left(1-\zeta_{\delta}(x_{i}-x_{j})\right)\\&-C(\left\Vert \mu\right\Vert _{\infty},d,k,p)\delta^{\frac{d-pk}{p}}. 
\end{align*}
Hence, in view of Lemma \ref{Regularization }-iv,
we get 
\begin{align}
\int_{x\neq y}(1-\zeta_{\delta}(x-y))&V(x-y)\left(\mu_{N}-\mu\right)^{\otimes2}(dxdy)  \nonumber\\
&\geq\frac{1}{N^{2}}\underset{i\neq j}{\sum}m_{i}m_{j}V_{\varepsilon}(x_{i}-x_{j})\left(1-\zeta_{\delta}(x_{i}-x_{j})\right)-C(\left\Vert \mu\right\Vert _{\infty},d,k,p)\delta^{\frac{d-pk}{p}}-\varepsilon. 
\label{eq:-5}   
\end{align}
Moreover, Lemma \ref{Regularization }-iii entails 
\begin{align*}
\left\Vert(\zeta_{\delta}V)\star\mu-(\zeta_{\delta}V_{\varepsilon})\star\mu\right\Vert _{\infty}&\leq \left\Vert (\zeta_{\delta}(V_{\varepsilon}-V))\star\mu\right\Vert _{\infty}\\&\lesssim \left\Vert \zeta_{\delta}(V_{\varepsilon}-V)\right\Vert _{L^{p}+L^{q}}\lesssim \frac{\varepsilon}{\delta^{k+1-\frac{d}{q}}}+\delta^{\frac{d}{p}-k},     
\end{align*}
so that we can bound from below  the first integral in \eqref{first identity in lemma 4.3} as 
\begin{align}
\int_{\mathbb{R}^{d}\times \mathbb{R}^{d}}&\zeta_{\delta}(x-y)V(x-y)\left(\mu_{N}-\mu\right)^{\otimes2}(dxdy) \nonumber\\
&\geq\int_{\mathbb{R}^{d}\times \mathbb{R}^{d}}\zeta_{\delta}(x-y)V_{\varepsilon}(x-y)\left(\mu_{N}-\mu\right)^{\otimes2}(dxdy)-C(\left\Vert \mu\right\Vert _{\infty},d,k,p)\left(\frac{\varepsilon}{\delta^{k+1-\frac{d}{q}}}+\delta^{\frac{d}{p}-k}\right)-\varepsilon,\label{eq:-6}
\end{align}
using again Lemma \ref{Regularization }-iv. Hence summing up inequalities \eqref{eq:-5} and \eqref{eq:-6} we get 
\begin{align}
\int_{x\neq y}V(x-y)\left(\mu_{N}-\mu\right)^{\otimes2}(dxdy)\geq&\int_{\mathbb{R}^{d}\times\mathbb{R}^{d}}\zeta_{\delta}(x-y)V_{\varepsilon}(x-y)\left(\mu_{N}-\mu\right)^{\otimes2}(dxdy) \notag
\\
&+\frac{1}{N^{2}}\underset{i\neq j}{\sum}m_{i}m_{j}V_{\varepsilon}(x_{i}-x_{j})\left(1-\zeta_{\delta}(x_{i}-x_{j})\right) \notag\\&-C(\left\Vert \mu\right\Vert _{\infty},d,k,p)\left(\frac{\varepsilon}{\delta^{k+1-\frac{d}{q}}}+\delta^{\frac{d}{p}-k}+\varepsilon \right). \label{aux5}
\end{align}
Notice that the second term in the right-hand side can be rewritten as
\begin{align}
\frac{1}{N^{2}}\underset{i\neq j}{\sum}m_{i}m_{j}V_{\varepsilon}(x_{i}-x_{j})\left(1-\zeta_{\delta}(x_{i}-x_{j})\right)=&\int_{\mathbb{R}^{d}\times\mathbb{R}^{d}}(1-\zeta_{\delta}(x-y))V_{\varepsilon}(x-y) (\mu_{N}-\mu)^{\otimes2}(dxdy)\notag\\
&-\frac{1}{N^{2}}\sum_{j=1}^{N}m_{j}^{2}V_{\varepsilon}(0)-\int_{\mathbb{R}^{d}}((1-\zeta_{\delta})V_{\varepsilon})\star\mu(x) \mu(x)dx\notag\\
&+\frac{2}{N}\sum_{j=1}^{N}m_{j}((1-\zeta_{\delta})V_{\varepsilon})\star\mu(x_{j}).\label{aux6}
\end{align}
Using Lemma 5.1-ii and \eqref{INFINITY NORM OF TRUN}, we have
\begin{align*}
\left\Vert ((1-\zeta_{\delta}) V_{\varepsilon})\star \mu\right\Vert_{\infty}&\leq \left\Vert ((1-\zeta_{\delta})( V_{\varepsilon}-V))\star \mu\right\Vert_{\infty}+\left\Vert ((1-\zeta_{\delta})V)\star \mu\right\Vert_{\infty}\\
&\lesssim \left\Vert V_{\varepsilon}-V\right\Vert_{L^{p}+L^{q}}+\delta^{\frac{d}{p}-k}\lesssim \eta(\varepsilon)+\delta^{\frac{d}{p}-k}.
\end{align*}
Inserting this into \eqref{aux6} together with a trivial estimate on the second term of the right-hand side of \eqref{aux6} we obtain
\begin{align}
\frac{1}{N^{2}}\underset{i\neq j}{\sum}m_{i}m_{j}V_{\varepsilon}(x_{i}-x_{j})\left(1-\zeta_{\delta}(x_{i}-x_{j})\right)\geq&\int_{\mathbb{R}^{d}\times\mathbb{R}^{d}}(1-\zeta_{\delta}(x-y))V_{\varepsilon}(x-y) (\mu_{N}-\mu)^{\otimes2}(dxdy)\notag\\
&-\frac{\overline{M}\left\Vert V_{\varepsilon}\right\Vert _{\infty}}{N}-C(\left\Vert \mu\right\Vert _{\infty},d,k,p)\left(\delta^{\frac{d}{p}-k}+\eta(\varepsilon) \right). \label{aux4}
\end{align}
We finally substitute \eqref{aux4} into \eqref{aux5} to conclude that
\begin{align}
\int_{x\neq y}V(x-y)\left(\mu_{N}-\mu\right)^{\otimes2}(dxdy)\geq&\int_{\mathbb{R}^{d}\times\mathbb{R}^{d}}V_{\varepsilon}(x-y)\left(\mu_{N}-\mu\right)^{\otimes2}(dxdy) -\frac{\overline{M}\left\Vert V_{\varepsilon}\right\Vert _{\infty}}{N}\notag\\
&-C(\left\Vert \mu\right\Vert _{\infty},d,k,p)\left(\frac{\varepsilon}{\delta^{k+1-\frac{d}{q}}}+\delta^{\frac{d}{p}-k}+\eta(\varepsilon) \right).\label{eq:-7}
\end{align}
By Lemma \ref{Regularization } we have $\widehat{V_{\varepsilon}}\geq0$ so that in view of (\ref{eq:-7}) we obtain 
\[
\int_{x\neq y}V(x-y)\left(\mu_{N}-\mu\right)^{\otimes2}(dxdy)\geq-\frac{\overline{M}\left\Vert V_{\varepsilon}\right\Vert _{\infty}}{N}-C(\left\Vert \mu\right\Vert _{\infty},d,k,p)\left(\frac{\varepsilon}{\delta^{k+1-\frac{d}{q}}}+\delta^{\frac{d}{p}-k}+\eta(\varepsilon)\right),
\]
which is the desired inequality \eqref{1st identity}. 

\vskip 12pt

 \textbf{Step 2}. We now prove   \eqref{2nd identity}. We follow the same procedure as in \eqref{aux6} reconstructing the tensor measure $\left(\mu_{N}-\mu\right)^{\otimes2}$ together with \eqref{INFINITY NORM OF TRUN} to deduce 
\begin{align}
\frac{1}{N^{2}}\underset{i\neq j}{\sum}m_{i}m_{j}(1-\zeta_{\delta})(x_{i}-x_{j})V(x_{i}-x_{j})\leq&\int_{x\neq y}(1-\zeta_{\delta})(x-y)V(x-y)\left(\mu_{N}-\mu\right)^{\otimes2}(dxdy) \notag \\&+C(k,d,\left\Vert \mu\right\Vert_{\infty})\delta^{\frac{d-kp}{p}} 
\notag\\=&\int_{x\neq y}V(x-y)\left(\mu_{N}-\mu\right)^{\otimes2}(dxdy)\notag \\
&-\int_{x\neq y}\zeta_{\delta}(x-y)V(x-y)\left(\mu_{N}-\mu\right)^{\otimes2}(dxdy) \notag\\
&+C(k,d,\left\Vert \mu\right\Vert_{\infty})\delta^{\frac{d-kp}{p}}, \label{ineaa}
\end{align}
Notice that there is no diagonal term with respect to \eqref{aux6} since the integration domain avoids the diagonal.
Inequality \eqref{eq:-6} can be written as
\begin{align}
-\int_{x\neq y}\zeta_{\delta}(x-y)V(x-y)\left(\mu_{N}-\mu\right)^{\otimes2}(dxdy)
\leq&\,-\int_{\mathbb{R}^{d}\times\mathbb{R}^{d}}\zeta_{\delta}(x-y)V_{\varepsilon}(x-y)\left(\mu_{N}-\mu\right)^{\otimes2}(dxdy)\notag\\
&+C(\left\Vert \mu\right\Vert _{\infty},d,k,p)\left(\frac{\varepsilon}{\delta^{k+1-\frac{d}{q}}}+\delta^{\frac{d}{p}-k}\right)+\varepsilon\notag\\
\leq &\,\int_{\mathbb{R}^{d}\times\mathbb{R}^{d}}(1-\zeta_{\delta})(x-y)V_{\varepsilon}(x-y)\left(\mu_{N}-\mu\right)^{\otimes2}(dxdy)\notag\\
&+C(\left\Vert \mu\right\Vert _{\infty},d,k,p)\left(\frac{\varepsilon}{\delta^{k+1-\frac{d}{q}}}+\delta^{\frac{d}{p}-k}\right)+\varepsilon.    \label{ina}    
\end{align}
By Corollary \ref{cor FT} we have  $\widehat{(1-\zeta_{\delta})V_{\varepsilon}}\lesssim \widehat{V_{\varepsilon}}+\frac{\delta^{d-k-\frac{\sigma}{2}}o_{\varepsilon \delta^{-1}}(1)}{(1+\left|\xi\right|)^{\frac{\sigma}{2}}}$
so that 
\begin{align}
\int_{\mathbb{R}^{d}\times\mathbb{R}^{d}}(1-\zeta_{\delta})(x-y)V_{\varepsilon}(x-y)\left(\mu_{N}-\mu\right)^{\otimes2}(dxdy)\lesssim &\,\int_{\mathbb{R}^{d}\times\mathbb{R}^{d}} V_{\varepsilon}(x-y)\left(\mu_{N}-\mu\right)^{\otimes2}(dxdy)\notag\\
&+\delta^{d-k-\frac{\sigma}{2}}o_{\varepsilon\delta^{-1}}(1)\int_{\mathbb{R}^{d}}\frac{\left\vert 1-\widehat{\mu}\right\vert^{2}(\xi)}{(1+\left\vert \xi\right\vert)^{\frac{\sigma}{2}}}d\xi
\notag\\\lesssim&\, \int_{\mathbb{R}^{d}\times\mathbb{R}^{d}} V_{\varepsilon}(x-y)\left(\mu_{N}-\mu\right)^{\otimes2}(dxdy) +\delta^{d-k-\frac{\sigma}{2}}o_{\varepsilon\delta^{-1}}(1)\notag\\
\lesssim&\, \int_{x \neq y} V_{\varepsilon}(x-y)\left(\mu_{N}-\mu\right)^{\otimes2}(dxdy)\notag\\
&+\frac{\overline{M}\left\Vert V_{\varepsilon} \right\Vert_{\infty}}{N}+\delta^{d-k-\frac{\sigma}{2}}o_{\varepsilon\delta^{-1}}(1). \label{inee}
\end{align}
Notice that this passage to Fourier variables is justified since we work with the regular interaction potential $V_{\varepsilon}$ instead of $V$. By Lemma \ref{Regularization }, we obtain 
\begin{equation}
 \left\Vert V_{\varepsilon}\star\mu-V\star\mu\right\Vert_{\infty}\lesssim \left\Vert V_{\varepsilon}-V\right\Vert_{L^p+L^q(\mathbb{R}^d)}=\eta(\varepsilon). \label{ineee}   
\end{equation}
Gathering  \eqref{ina}, \eqref{inee}, and \eqref{ineee} we find  
\begin{align*}
-\int_{x\neq y}&\zeta_{\delta}(x-y)V(x-y)\left(\mu_{N}-\mu\right)^{\otimes2}(dxdy)\\\lesssim&
 \int_{x \neq y} V_{\varepsilon}(x-y)\left(\mu_{N}-\mu\right)^{\otimes2}(dxdy)+\frac{\overline{M}\left\Vert V_{\varepsilon} \right\Vert_{\infty}}{N}\\
 &+C(\left\Vert \mu\right\Vert _{\infty},d,k,p)\left(\frac{\varepsilon}{\delta^{k+1-\frac{d}{q}}}+\delta^{\frac{d}{p}-k}+\delta^{d-k-\frac{\sigma}{2}}o_{\varepsilon\delta^{-1}}(1)\right)\\
 \lesssim& \int_{x \neq y} (V_{\varepsilon}-V)(x-y)\left(\mu_{N}-\mu\right)^{\otimes2}(dxdy)+\mathcal{E}_{N}(\mu,\mu_{N})+\frac{\overline{M}\left\Vert V_{\varepsilon} \right\Vert_{\infty}}{N}\\
 &+C(\left\Vert \mu\right\Vert _{\infty},d,k,p)\left(\frac{\varepsilon}{\delta^{k+1-\frac{d}{q}}}+\delta^{\frac{d}{p}-k}+\delta^{d-k-\frac{\sigma}{2}}o_{\varepsilon\delta^{-1}}(1)\right).
\end{align*} 
Moreover, Lemma \ref{Regularization }-ii and Lemma \ref{Regularization }-iv using \eqref{ineee} implies that 
 \begin{align*}
 \int_{x \neq y} (V_{\varepsilon}-V)(x-y)\left(\mu_{N}-\mu\right)^{\otimes2}(dxdy)\lesssim o_{\varepsilon}(1)+\int_{x \neq y} (V_{\varepsilon}-V)(x-y)\mu_{N}^{\otimes2}(dxdy)\lesssim o_{\varepsilon}(1)+\varepsilon.
 \end{align*}
Therefore, in view of \eqref{ineaa} we obtain 
\begin{align*}
\frac{1}{N^{2}}\underset{i\neq j}{\sum}m_{i}m_{j}(1-\zeta_{\delta})V(x_{i}-x_{j})\leq C(k,d,p,q,\left\Vert \mu\right\Vert_{\infty})\left(\mathcal{E}_{N}(\mu,\mu_{N})+\mathcal{O}(\varepsilon,\delta,N)\right)     
\end{align*}
with 
\begin{align*}
\mathcal{O}(\varepsilon,\delta,N)=\frac{\varepsilon}{\delta^{k+1-\frac{d}{q}}}+\delta^{d-k-\frac{\sigma}{2}}o_{\varepsilon\delta^{-1}}(1)+\frac{\overline{M}\left\Vert V_{\varepsilon} \right\Vert_{\infty}}{N}+\delta^{\frac{d}{p}-k},     
\end{align*} 
as desired.
\qed

\

We are now ready to estimate $\mathcal{D}^{2}_{N}$ in \eqref{sketch sec 2} by means of $\mathcal{E}_{N}(\mu,\mu_{N})$ in \eqref{renomenergyN}, which is the main novelty of this section. 
\begin{lem}\label{commutator estimate1}
Under the assumptions of Lemma  \ref{Bound on the non diag terms}, it holds that 
\begin{align*}
\mathcal{D}^2_{N}(\mu)&=\int_{x \neq y}V(x-y)\left(\mu_{N}-\mu\right)(x)\left(h\left[\mu_{N}\right]-h\left[\mu\right]\right)(y)(dxdy)\leq C\left(\mathcal{E}_{N}(\mu,\mu_{N})+\eta(\varepsilon)+\mathcal{O}(\varepsilon,\delta,N)
\right), 
\end{align*} 
where $\mathcal{O}(\varepsilon,\delta,N)$ is given by \eqref{aux57}. 
\end{lem}
\textit{Proof}. Let $\zeta_{\delta}$ be a function as in Lemma \ref{Bound on the non diag terms}. We split the integral into a near and far parts from
the origin as follows 
\begin{align*}
\mathcal{D}^2_{N}(\mu)
=&\int_{\mathbb{R}^{d}\times\mathbb{R}^{d}}(\zeta_{\delta}V)(x-y)\left(\mu_{N}-\mu\right)(x)\left(h\left[\mu_{N}\right]-h\left[\mu\right]\right)(y)dxdy\\
&+\int_{x\neq y}((1-\zeta_{\delta})V)(x-y)\left(\mu_{N}-\mu\right)(x)\left(h\left[\mu_{N}\right]-h\left[\mu\right]\right)(y)dxdy\coloneqq I_{\delta}+J_{\delta}.
\end{align*}
We denote by $\lesssim$ inequality up to a constant depending on $d,k,p,q, \left\Vert S \right\Vert_{\infty},\left\Vert \mu \right\Vert_{\infty}$. 

\textbf{Step 1}\textit{. The integral} $J_{\delta}$. We expand the
integral as 
\begin{align*}
J_{\delta}=&\int_{x\neq y} ((1-\zeta_{\delta})V)(x-y)\mu_{N}(x)h\left[\mu_{N}\right](y)dxdy-\frac{1}{N}\stackrel[i=1]{N}{\sum}m_{i}(\left((1-\zeta_{\delta})V)\right)\star\mu)(x_{i})\,(S\star\mu_{N})(x_{i})\\
&-\frac{1}{N}\stackrel[i=1]{N}{\sum}m_{i}(\left((1-\zeta_{\delta})V)\right)\star h\left[\mu\right])(x_{i})+\int_{\mathbb{R}^d}\mu(x)(((1-\zeta_{\delta})V)\star h\left[\mu\right])(x)dx\coloneqq\stackrel[i=1]{4}{\sum}J_{\delta}^{i}.
\end{align*}
Owing to Lemma \ref{Bound on the non diag terms}, we have 
\begin{equation*}
\left|J_{\delta}^{1}\right|\leq\left\Vert S\right\Vert _{\infty}\frac{1}{N^{2}}\underset{i\neq j}{\sum}m_{i}m_{j} (1-\zeta_{\delta})(x_{i}-x_{j})V(x_{i}-x_{j})\leq C\left(\mathcal{E}_{N}(\mu,\mu_{N})+\mathcal{O}(\varepsilon,\delta,N)\right).
\end{equation*}
Next, using similar computations as in \eqref{comphomog}, we deduce the following elementary bound
\begin{equation}
\left\Vert ((1-\zeta_{\delta})V)\star\mu\right\Vert _{\infty}\leq\left\Vert (1-\zeta_{\delta})V\right\Vert _{1}\left\Vert \mu\right\Vert _{\infty}\lesssim \delta^{d-k},\label{eq:-2}
\end{equation}
which together with the assumption (\textbf{H1}) directly implies that
$\left|J_{\delta}^{2}\right|\lesssim \delta^{d-k}.$
Similarly, we have that
\begin{equation*}
\left\Vert ((1-\zeta_{\delta})V)\star h\left[\mu\right]\right\Vert _{\infty}\leq\left\Vert (1-\zeta_{\delta})V\right\Vert _{1}\left\Vert h\left[\mu\right]\right\Vert _{\infty}\lesssim \delta^{d-k},
\end{equation*}
yielding that 
$\left|J_{\delta}^{3}\right|\lesssim \delta^{d-k}$ and $\left|J_{\delta}^{4}\right|\lesssim \delta^{d-k}$.
Collecting these estimates, we find 
\begin{equation}
J_{\delta}\lesssim \mathcal{E}_{N}(\mu,\mu_{N})+\mathcal{O}(\varepsilon,\delta,N)+\delta^{d-k}. \label{Estimate J_{delta}}    
\end{equation}
To estimate $I_{\delta}$, we need to use the regularization $V_{\varepsilon}$ constructed in Lemma \ref{Regularization }. Note the identity 
\begin{align*}
I_\delta
=&\int_{\mathbb{R}^{d}\times\mathbb{R}^{d}}\zeta_{\delta}(x-y)(V-V_{\varepsilon})(x-y)\left(\mu_{N}-\mu\right)(x)\left((h\left[\mu_{N}\right]-\left\Vert S\right\Vert _{\infty}\mu_{N})-(h\left[\mu\right]-\left\Vert S\right\Vert _{\infty}\mu)\right)(y)dxdy\\
&+\int_{\mathbb{R}^{d}\times\mathbb{R}^{d}}\zeta_{\delta}(x-y)V_{\varepsilon}(x-y)\left(\mu_{N}-\mu\right)(x)\left(h\left[\mu_{N}\right]-h\left[\mu\right]\right)(y)dxdy\\
&+\left\Vert S\right\Vert _{\infty}\int_{\mathbb{R}^{d}\times\mathbb{R}^{d}}\zeta_{\delta}(x-y)(V-V_{\varepsilon})(x-y)\left(\mu_{N}-\mu\right)^{\otimes2}(dxdy)\coloneqq I_{\delta,\varepsilon}^{1}+I_{\delta,\varepsilon}^{2}+I_{\delta,\varepsilon}^{3}.
\end{align*}
In what follows, we separately estimate $I_{\delta,\varepsilon}^{1},I_{\delta,\varepsilon}^{2},I_{\delta,\varepsilon}^{3}$. 

\vskip 12pt

\textbf{Step 2}. \textit{The integral $I^{1}_{\delta,\varepsilon}$}. The term $I^1_{\delta,\varepsilon}$ is recast as 
\begin{equation}
\begin{aligned}
I_{\delta,\varepsilon}^{1}=&\int_{\mathbb{R}^{d}\times\mathbb{R}^{d}}\zeta_{\delta}(x-y)(V_{\varepsilon}-V)(x-y)\mu_{N}(x)(\left\Vert S\right\Vert _{\infty}\mu_{N}-h\left[\mu_{N}\right])(y)dxdy\\
&+\frac{1}{N}\stackrel[i=1]{N}{\sum}m_{i}((\zeta_{\delta}(V_{\varepsilon}-V))\star(h\left[\mu\right]-\left\Vert S\right\Vert _{\infty}\mu))(x_{i})\\
&+\frac{1}{N}\stackrel[i=1]{N}{\sum}m_{i}((\zeta_{\delta}(V_{\varepsilon}-V))\star\mu)(x_{i})(\left\Vert S\right\Vert _{\infty}-(S\star\mu_{N})(x_{i}))\\  
&+\int_{\mathbb{R}^{d}}((\zeta_{\delta}(V_{\varepsilon}-V))\star\mu)(x)(\left\Vert S\right\Vert _{\infty}\mu-h\left[\mu\right])(x)dx. \   
\end{aligned}
\label{I^1 decom}
\end{equation}
As $\left\Vert S \right\Vert_{\infty}\mu_{N}-h[\mu_{N}]$ is a non-negative measure, by Lemma \ref{Regularization }-iv, the 1st term in the
right-hand side of \eqref{I^1 decom} is bounded by
\[
\varepsilon\int_{\mathbb{R}^d}\zeta_{\delta}(x-y)\mu_{N}(x)(\left\Vert S\right\Vert _{\infty}\mu_{N}-h\left[\mu_{N}\right])dxdy\leq2\left\Vert S\right\Vert _{\infty}\varepsilon.
\]
Utilizing Lemma \ref{Regularization }-iii. we see that the 2nd term in the right-hand side of (\ref{I^1 decom}) is less than
\begin{align*}
&\left\Vert (\zeta_{\delta}(V_{\varepsilon}-V))\star(h\left[\mu\right]-\left\Vert S\right\Vert _{\infty}\mu)\right\Vert _{\infty}\\
&\leq\left\Vert \zeta_{\delta}(V_{\varepsilon}-V)\right\Vert _{L^{p}+L^{q}}\left\Vert h\left[\mu\right]-\left\Vert S\right\Vert _{\infty}\mu\right\Vert _{\infty}\leq 2C\left\Vert \mu\right\Vert _{\infty}\left\Vert S\right\Vert _{\infty}\left(\frac{\varepsilon}{\delta^{k+1-\frac{d}{q}}}+\delta^{\frac{d}{p}-k}\right).
\end{align*}
By the same token, the 3rd term in the right hand side of (\ref{I^1 decom}) is bounded by
\[
2\left\Vert S\right\Vert _{\infty}\left\Vert \zeta_{\delta}(V_{\varepsilon}-V)\star\mu\right\Vert _{\infty}\leq2\left\Vert S\right\Vert _{\infty}\left\Vert \mu\right\Vert _{\infty}\left\Vert \zeta_{\delta}(V_{\varepsilon}-V)\right\Vert _{L^{p}+L^{q}}\leq 2C\left\Vert \mu\right\Vert _{\infty}\left\Vert S\right\Vert _{\infty}\left(\frac{\varepsilon}{\delta^{k+1-\frac{d}{q}}}+\delta^{\frac{d}{p}-k}\right),
\]
and the 4th term is dominated by
\[
2\left\Vert S\right\Vert _{\infty}\left\Vert \mu\right\Vert _{\infty}\left\Vert \zeta_{\delta}(V_{\varepsilon}-V)\right\Vert _{L^{p}+L^{q}}\leq 2C\left\Vert \mu\right\Vert _{\infty}\left\Vert S\right\Vert _{\infty}\left(\frac{\varepsilon}{\delta^{k+1-\frac{d}{q}}}+\delta^{\frac{d}{p}-k}\right).
\]
Gathering our inequalities, we find that 
\begin{equation*}
I_{\delta,\varepsilon}^{1}\leq C\left(\frac{\varepsilon}{\delta^{k+1-\frac{d}{q}}}+\delta^{\frac{d}{p}-k}\right)   
\end{equation*}
where $C= C\left(\left\Vert S\right\Vert _{\infty},\left\Vert \mu\right\Vert _{\infty},d,k,p,q\right)$.

\vskip 12pt

\textbf{Step 3}. \textit{The integral $I^{2}_{\delta,\varepsilon}$}. Note first that 
\begin{align*}
I_{\delta,\varepsilon}^{2}&=\int_{\mathbb{R}^{d}}\zeta_{\delta}V_{\varepsilon}(x-y)\left(\mu_{N}-\mu\right)(x)\left(h\left[\mu_{N}\right]-h\left[\mu\right]\right)(y)dxdy\\
&=\int_{\mathbb{R}^{d}}\widehat{\zeta_{\delta}V_{\varepsilon}}(x)\left(\widehat{\mu_{N}-\mu}\right)(x)\left(\widehat{h\left[\mu_{N}\right]-h\left[\mu\right]}\right)(x)dx.
\end{align*}
Therefore, the Cauchy-Schwarz inequality  yields   
\begin{align}
I_{\delta,\varepsilon}^{2}&\leq\int_{\mathbb{R}^{d}}\left\vert \widehat{\zeta_{\delta}V_{\varepsilon}}\right\vert \left|\widehat{\mu_{N}-\mu}\right|^{2}(x)dx+\int_{\mathbb{R}^{d}}\left\vert \widehat{\zeta_{\delta}V_{\varepsilon}}\right\vert \left|\widehat{h\left[\mu_{N}\right]-h\left[\mu\right]}\right|^{2}(x)dx . \label{h term} 
\end{align}
Now, we want to apply Lemma \ref{h Lip estimate } to the second term on the right-hand side of \eqref{h term}. The first two assumptions are satisfied by the choice of $S$ in {\bf (H1)}. We are reduced to check the third assumption in Lemma \ref{h Lip estimate }. By Corollary \ref{cor FT} it holds that 
\begin{equation}
\left\vert \widehat{\zeta_{\delta}V_{\varepsilon}}\right\vert(\xi) 
\lesssim \widehat{V_{\varepsilon}}(\xi)+\frac{\delta^{d-k-\frac{\sigma}{2}}o_{\varepsilon\delta^{-1}}(1)}{(1+\left|\xi\right|)^{\frac{\sigma}{2}}}.    \label{reminder of corollary 4.1}    
\end{equation}
for $\sigma > 2d$. 
By Lemma \ref{bound on convolution} we have 
\begin{align*}
\left\vert \widehat{S}\right\vert\star \widehat{V}_{\varepsilon}(x)\lesssim \widehat{V}_{\varepsilon}(x)+\frac{o_{\varepsilon}(1)}{(1+\left\vert x \right\vert)^{\frac{\sigma}{2}}},   
\end{align*}
while by the assumption that $S\in \mathcal{S}(\mathbb{R}^{d})$ we have 
\begin{align*}
\left\vert \widehat{S}\right\vert\star \left(\frac{1}{(1+\left\vert \cdot\right\vert)^{\frac{\sigma}{2}}}\right)\lesssim \frac{1}{(1+\left\vert x \right\vert)^{\frac{\sigma}{2}}}.     
\end{align*}
Thus, we conclude that
\begin{align*}
\left\vert \widehat{S}\right\vert \star \left(\widehat{V}_{\varepsilon}+\frac{o_{\varepsilon \delta^{-1}}(1)}{\left(1+\left\vert \cdot\right\vert \right)^{\frac{\sigma}{2}}}\right)(x)\lesssim \widehat{V}_{\varepsilon}(x)+\frac{\delta^{d-k-\frac{\sigma}{2}}o_{\varepsilon\delta^{-1}}(1)}{\left(1+\left\vert x\right\vert \right)^{\frac{\sigma}{2}}}
\end{align*}
leading to the desired convolution inequality in Lemma \ref{h Lip estimate }. So, we are now entitled to apply Lemma \ref{h Lip estimate } with $W(x)=\widehat{V}_{\varepsilon}(x)+\frac{\delta^{d-k-\frac{\sigma}{2}}o_{\varepsilon\delta^{-1}}(1)}{\left(1+\left\vert x\right\vert \right)^{\frac{\sigma}{2}}}$ to obtain that 
\begin{align*}
\int_{\mathbb{R}^{d}}\left\vert \widehat{\zeta_{\delta}V_{\varepsilon}}\right\vert \left|\widehat{h\left[\mu_{N}\right]-h\left[\mu\right]}\right|^{2}(x)dx\leq C\int_{\mathbb{R}^{d}}\left(\widehat{V}_{\varepsilon}(x)+\frac{\delta^{d-k-\frac{\sigma}{2}}o_{\varepsilon\delta^{-1}}(1)}{(1+\left\vert x\right\vert)^{\frac{\sigma}{2}}}\right)\left\vert \widehat{\mu_{N}-\mu}\right\vert^{2}(x)dx,     
\end{align*}
which proves that 
\begin{align*}
I_{\delta,\varepsilon}^{2}\lesssim \int_{\mathbb{R}^{d}}\left(\widehat{V}_{\varepsilon}(x)+\frac{\delta^{d-k-\frac{\sigma}{2}}o_{\varepsilon\delta^{-1}}(1)}{(1+\left\vert x\right\vert)^{\frac{\sigma}{2}}}\right)\left\vert \widehat{\mu_{N}-\mu}\right\vert^{2}(x)dx. \end{align*}
Hence, similar to Step 2 in Lemma \ref{Bound on the non diag terms}, we arrive at the estimate 
\begin{align}
I_{\delta,\varepsilon}^{2}&\lesssim \int_{x\neq y}V_{\varepsilon}(x-y)\left(\mu_{N}-\mu\right)^{\otimes2}(dxdy)+\frac{\overline{M}\left\Vert V_{\varepsilon}\right\Vert _{\infty}}{N}+\delta^{d-k-\frac{\sigma}{2}}o_{\varepsilon\delta^{-1}}(1) \notag \\
&\lesssim\int_{x\neq y}\left(V_{\varepsilon}(x-y)-V(x-y)\right)\left(\mu_{N}-\mu\right)^{\otimes2}(dxdy)+\mathcal{E}_{N}(\mu,\mu_{N})+\frac{\overline{M}\left\Vert V_{\varepsilon}\right\Vert _{\infty}}{N}+\delta^{d-k-\frac{\sigma}{2}}o_{\varepsilon\delta^{-1}}(1).    \label{inter I^2}
\end{align}
To finish, we estimate the first integral in the right-hand side of  \eqref{inter I^2}, which
is achieved by similar considerations to those in Step 2. Expanding the terms $\left(\mu_{N}-\mu\right)^{\otimes2}(dxdy)$ and using the fourth and the second statements in Lemma \ref{Regularization }, we have 
$$
\frac{1}{N^{2}}\underset{i\neq j}{\sum}m_{i}m_{j}\left(V_{\varepsilon}(x_{i}-x_{j})-V(x_{i}-x_{j})\right)\leq\varepsilon,    
$$
$$
\left|\frac{2}{N}\stackrel[i=1]{N}{\sum}m_{i}(V_{\varepsilon}-V)\star\mu(x_{i})\right|\leq\frac{2}{N}\stackrel[i=1]{N}{\sum}\left\Vert (V_{\varepsilon}-V)\star\mu\right\Vert _{\infty}\leq \left\Vert \mu \right\Vert_{\infty}\left\Vert V_{\varepsilon}-V\right\Vert _{L^{p}+L^{q}}\lesssim\eta(\varepsilon), 
$$
and 
\[
\int_{\mathbb{R}^{d}}(V_{\varepsilon}-V)\star\mu(x)\mu(x)dx\leq\left\Vert (V_{\varepsilon}-V)\star\mu\right\Vert _{\infty}\leq\left\Vert \mu \right\Vert_{\infty}\left\Vert V_{\varepsilon}-V\right\Vert _{L^{p}+L^{q}}\lesssim\eta(\varepsilon).
\]
Collecting terms implies that
\begin{equation}
\int_{x\neq y}(V_{\varepsilon}-V)(x-y)\left(\mu_{N}-\mu\right)^{\otimes2}(dxdy)\lesssim \varepsilon+\eta(\varepsilon).\label{eq:-11}
\end{equation}
Thus, \eqref{inter I^2} and \eqref{eq:-11} entail 
\begin{equation*}
I_{\delta,\varepsilon}^{2}\leq C\left(\mathcal{E}_{N}(\mu,\mu_{N})+\frac{\overline{M}\left\Vert V_{\varepsilon}\right\Vert _{\infty}}{N}+\delta^{d-k-\frac{\sigma}{2}}o_{\varepsilon\delta^{-1}}(1)\right), 
\end{equation*}
where 
$
C= C\left(\left\Vert S\right\Vert _{\infty},\left\Vert \mu\right\Vert _{\infty},d,k,p,q\right).    
$

\vskip 12pt

\textbf{Step 4}. We immediately have that 
\begin{align*}
I_{\delta,\varepsilon}^{3}=&\left\Vert S\right\Vert _{\infty}\int_{\mathbb{R}^{d}\times\mathbb{R}^{d}}\zeta_{\delta}(x-y)V(x-y)\left(\mu_{N}-\mu\right)^{\otimes2}(dxdy)\\
&-\left\Vert S\right\Vert _{\infty}\int_{\mathbb{R}^{d}\times\mathbb{R}^{d}}\zeta_{\delta}(x-y)V_{\varepsilon}(x-y)\left(\mu_{N}-\mu\right)^{\otimes2}(dxdy).
\end{align*}
First, note that 
\begin{align*}
\int_{\mathbb{R}^{d}\times\mathbb{R}^{d}}\zeta_{\delta}(x-y)V(x-y)\left(\mu_{N}-\mu\right)^{\otimes2}(dxdy)=& \int_{x\neq y}\zeta_{\delta}(x-y)V(x-y)\left(\mu_{N}-\mu\right)^{\otimes2}(dxdy) \notag\\
=&\ \mathcal{E}_{N}(\mu,\mu_{N})-\frac{1}{N^{2}}\underset{i\neq j}{\sum}m_{i}m_{j}((1-\zeta_{\delta})V)(x_{i}-x_{j}) \notag\\
&-\int_{\mathbb{R}^{d}}\mu(x)(((1-\zeta_{\delta})V)\star\mu)(x)dx \notag\\
&-\frac{1}{N}\stackrel[i=1]{N}{\sum}m_{i}(((1-\zeta_{\delta})V)\star\mu)(x_{i}) \notag\\
\leq& C\left(\mathcal{E}_{N}(\mu,\mu_{N})+\mathcal{O}(\varepsilon,\delta,N)+\delta^{d-k}\right), 
\end{align*}
where the last inequality uses Lemma \ref{Bound on the non diag terms} to estimate the second term and the observation that $\left\Vert (1-\zeta_{\delta}) V\right\Vert_{1}\lesssim \delta^{d-k}$ as in \eqref{eq:-2}. Second, by Corollary \ref{cor FT} it holds that 
\begin{align*}
\left\vert \widehat{\zeta_{\delta}V_{\varepsilon}}\right\vert 
\lesssim \widehat{V_{\varepsilon}}+\frac{\delta^{d-k-\frac{\sigma}{2}}o_{\varepsilon \delta^{-1}}(1)}{(1+\left|\xi\right|)^{\frac{\sigma}{2}}},    
\end{align*}
so that we get 
\begin{align*}
\left|\int_{\mathbb{R}^d}\zeta_{\delta}V_{\varepsilon}(x-y)\left(\mu_{N}-\mu\right)^{\otimes2}(dxdy)\right|&=\left|\int_{\mathbb{R}^{d}}\widehat{\zeta_{\delta}V_{\varepsilon}}(x)\left|\widehat{\mu_{N}-\mu}\right|^{2}(x)dx\right|\\
&\lesssim \int_{\mathbb{R}^{d}}\widehat{V_{\varepsilon}}(x)\left|\widehat{\mu_{N}-\mu}\right|^{2}(x)dx+\delta^{d-k-\frac{\sigma}{2}}o_{\varepsilon\delta^{-1}}(1)    .
\end{align*}
By \eqref{eq:-11} of step 3, it follows that 
\begin{align*}
\left|\int_{\mathbb{R}^d}\zeta_{\delta}V_{\varepsilon}(x-y)\left(\mu_{N}-\mu\right)^{\otimes2}(dxdy)\right|\lesssim &\int_{x \neq y}(V_{\varepsilon}-V)(x-y)(\mu_{N}-\mu)^{\otimes 2}(dxdy)\\
&+\frac{\overline{M}\left\Vert V_{\varepsilon}\right\Vert _{\infty}}{N}+\mathcal{E}_{N}(\mu,\mu_{N})+\delta^{d-k-\frac{\sigma}{2}}o_{\varepsilon\delta^{-1}}(1)\\
\lesssim&  
 \mathcal{E}_{N}(\mu,\mu_{N})+\varepsilon+\eta(\varepsilon)+\frac{\overline{M}\left\Vert V_{\varepsilon}\right\Vert _{\infty}}{N}+\delta^{d-k-\frac{\sigma}{2}}o_{\varepsilon \delta^{-1}}(1).    
\end{align*}
This proves that 
\begin{equation*}
I_{\delta,\varepsilon}^{3}\lesssim \mathcal{\mathcal{E}}_{N}(\mu,\mu_{N})+\varepsilon+\eta(\varepsilon)+\frac{\overline{M}\left\Vert V_{\varepsilon}\right\Vert _{\infty}}{N}+\delta^{d-k-\frac{\sigma}{2}}o_{\varepsilon \delta^{-1}}(1)+\delta^{d-k}+\mathcal{O}(\varepsilon,\delta,N).
\end{equation*}
Collecting all previous terms $I_{\delta,\varepsilon}^{1}$, $I_{\delta,\varepsilon}^{2}$ and $I_{\delta,\varepsilon}^{3}$, we find that 
\begin{align*}
I_{\delta}&\lesssim \mathcal{\mathcal{E}}_{N}(\mu,\mu_{N})+\varepsilon+\eta(\varepsilon)+\frac{\varepsilon}{\delta^{k+1-\frac{d}{q}}}+\delta^{\frac{d}{p}-k}+\frac{\overline{M}\left\Vert V_{\varepsilon}\right\Vert _{\infty}}{N}+\delta^{d-k-\frac{\sigma}{2}}o_{\varepsilon \delta^{-1}}(1)+\delta^{d-k}+\mathcal{O}(\varepsilon,\delta,N).
\end{align*}
Combining estimate \eqref{Estimate on I_{delta}} with \eqref{Estimate J_{delta}} and taking into account the definition of $\mathcal{O}(\varepsilon,\delta,N)$ in \eqref{aux57}, we deduce that
\begin{align}
\mathcal{D}^2_{N}(\mu)=I_{\delta}+J_{\delta}&\lesssim \mathcal{\mathcal{E}}_{N}(\mu,\mu_{N})+\eta(\varepsilon)+\mathcal{O}(\varepsilon,\delta,N),  
\label{Estimate on I_{delta}} \end{align}
since we will eventually take $\delta\to 0$, $p>1$ and $k-\frac{d}{q}>0$, and thus we conclude the announced result. 
\qed

\vskip 12pt

In order to prove the mean field limit, we need the asymptotic positivity, coercivity  inequalities and the commutator estimate, which are the fundamental discoveries in \cite{duerinckx2020mean}.

\begin{prop} \textup{(\cite[Corollary 3.5]{duerinckx2020mean})}
Let $V(x)=\frac{1}{\left\vert x \right\vert^{k}}$ where $0<k<d-1$ and let $\mu\in L^{\infty}(\mathbb{R}^{d})\cap \mathcal{P}(\mathbb{R}^{d})$. For any weights $\mathbf{m}_{N}\in \mathbb{M}^{N}$ with $m_{i}^{N}\leq \overline{M}, i=1,\dots,N$ and configurations $\mathbf{x}_{N}\in \mathbb{R}^{dN}\setminus \triangle_{N}$, there exist a constant $C=C(\overline{M},\left\Vert \mu\right\Vert_{\infty},k,d)$ such that the modulated energy is bounded below as
\begin{align*}
\mathcal{E}_{N}(\mu,\mu_{N}) \geq -C(1+\left\Vert \mu \right\Vert_{\infty})N^{\frac{k}{d}-1} 
\end{align*}
and thus, asymptotically non-negative.
\label{non positivity}
\end{prop}

\begin{prop}
\textup{\cite[Proposition 3.6]{duerinckx2020mean}}
 Let  $V(x)=\frac{1}{\left\vert x \right\vert^{k}}$ where $0<k<d-1$ and let $\mu\in L^{\infty}(\mathbb{R}^{d})\cap \mathcal{P}(\mathbb{R}^{d})$. For any weights $\mathbf{m}_{N}\in \mathbb{M}^{N}$ with $m_{i}^{N}\leq \overline{M}, i=1,\dots,N$ and configurations $\mathbf{x}_{N}\in \mathbb{R}^{dN}\setminus \triangle_{N}$ and $0<\alpha\leq 1$ there is some $C=C(d,k,\overline{M})>0$ and $\lambda=\lambda(d,k)>0$ such that for any $\varphi\in C^{\infty}(\mathbb{R}^{d})$ it holds that 
 \begin{align*}
\int_{\mathbb{R}^{d}}\varphi(x)\left(\mu_{N}-\mu\right)(x)(dx)\leq CN^{-\lambda}\left\Vert \varphi\right\Vert_{C^{0,\alpha}} +C\left\Vert \varphi\right\Vert_{\dot{H}^{\frac{d-k}{2}}} \left(\mathcal{E}_N(\mu,\mu_{N})+C(1+\left\Vert \mu \right\Vert_{\infty})N^{\frac{k}{d}-1}\right)^\frac{1}{2}.
 \end{align*}
 \label{coercivity inequality}
\end{prop}

\begin{thm} \textup{(\cite[Proposition 1.1]{duerinckx2020mean})}\label{commutator estimate} Let the assumptions of Proposition \ref{coercivity inequality} hold and assume further that $\mu\in C^{0,\alpha}(\mathbb{R}^{d})$ for some $\alpha \in (0,1)$ and that $u:\mathbb{R}^{d}\rightarrow \mathbb{R}^{d}$ is a Lipschitz vector field. Then
 \begin{align*}
\left\vert \int_{x\neq y}\!\!\!\!\!\!\!\left(u(x)-u(y)\right)\nabla V(x-y)\left(\mu_{N}-\mu\right)^{\otimes2}(dxdy)\right\vert\leq C\left(\int_{x\neq y}\!\!\!\!\!\!\!V(x-y)\left(\mu_{N}-\mu\right)^{\otimes2}(dxdy)+N^{-\beta}\right)
\end{align*}   
where $C=C\left(\overline{M},\left\Vert u \right\Vert_{W^{1,\infty}},\left\Vert \mu \right\Vert_{C^{0,\alpha}},d,k\right)$ and $\beta=\beta(k,d,\alpha)>0$.
\end{thm}
We are now in a good position to prove our main theorem.

\

\textit{Proof of Theorem \ref{main thm intro}}. By Lemma \ref{commutator estimate1} we have 
\begin{align*}
\mathcal{D}_{N}^{2}(t)
\lesssim \mathcal{E}_{N}(t)+\eta(\varepsilon)+\mathcal{O}(\varepsilon,\delta,N)   
\end{align*}
where $\mathcal{O}$ has the form 
\begin{align*}
\mathcal{O}(\varepsilon,\delta,N)=\frac{\varepsilon}{\delta^{k+1-\frac{d}{q}}}+\delta^{d-k-\frac{\sigma}{2}}o_{\varepsilon\delta^{-1}}(1)+\frac{\overline{M}\left\Vert V_{\varepsilon} \right\Vert_{\infty}}{N}+\delta^{\frac{d}{p}-k}.   
\end{align*} 
Note that $\left\Vert V_{\varepsilon}\right\Vert_{\infty}\leq \frac{1}{\Theta(\varepsilon)}$ for some $\Theta(\varepsilon)\rightarrow0$ as $\varepsilon \rightarrow 0$ due to Lemma \ref{Regularization }-v.
Choose $\varepsilon=\varepsilon(N)\to 0$ as $N\to\infty$ and choose $\delta=\delta(\varepsilon)$ such that 
$$
\max\left\{\eta(\varepsilon(N)), \frac1{\Theta(\varepsilon(N))N}, \frac{\varepsilon(N)}{\delta(\varepsilon(N))^{k+1-\frac{d}{q}}}\right\}\to 0\mbox{ as }N\rightarrow \infty \,.
$$
Then, we get 
\begin{align*}
\mathcal{D}_{N}^{2}(t)
\lesssim \mathcal{E}_{N}(t)+o_{N}(1), 
\end{align*}
since $p>1$ $k-\frac{d}{q}>0$.
By Theorem \ref{thm final existence} and the assumption $\mu_{0}\in W^{1,\infty}(\mathbb{R}^{d})$, we have $\mathbb{J}\nabla V \star \mu \in L^{\infty}([0,T],W^{1,\infty}(\mathbb{R}^d))$. Applying Theorem \ref{commutator estimate} with $u=\mathbb{J}\nabla V\star\mu$, we infer 
\begin{align*}
\mathcal{D}_{N}^{1}(t)\lesssim  \mathcal{E}_{N}(t)+N^{-\beta}
\end{align*}
where $C=C(\left\Vert\mu_{0}\right\Vert_{W^{1,\infty}},d,k)$. By Proposition \ref{Time derivative of energy} we conclude
\begin{align*}
\frac{d}{dt}\mathcal{E}_{N}(t)\leq \mathcal{D}_{N}^{1}(t)+\mathcal{D}_{N}^{2}(t)\leq C\mathcal{E}_{N}(t)+o_{N}(1).
\end{align*}
By Gr\"onwall's inequality and Proposition \ref{non positivity}, it follows that $\underset{t\in[0,T]} \sup{\mathcal{E}_{N}(t)}\underset{N\rightarrow \infty}{ \rightarrow} 0$, which by Proposition \ref{coercivity inequality} concludes the proof. 
\begin{flushright}
$\square$
\par\end{flushright}

\begin{rem}
Proposition \ref{non positivity}, Proposition \ref{coercivity inequality} and Theorem \ref{commutator estimate} in \cite{duerinckx2020mean} are all stated for the specific choice $\mathbf{m}_{N}=\left(1,\dots,1\right)$, but in fact hold for arbitrary convex combinations of Diracs, as can be deduced from a careful examination of the proof. Alternatively, this can be seen from   Lemma \ref{Bound on the non diag terms} which we proved for general weights:  Proposition \ref{non positivity} follows directly from this Lemma by choosing $\varepsilon$ and $\delta$ appropriately and Theorem \ref{commutator estimate} follows by this Lemma by exactly the same argument outlined in \cite[Corollary 5.1]{bresch2019modulated}. Those relations form a critical example of a larger family of functional inequalities controlling various quantities by the modulated energy. Those inequalities could often be proved in a straightforward if the empirical measure was replaced by a smooth function but cannot make for general measures, which is why we constantly need to remove the diagonal here. Roughly speaking Lemma~\ref{Bound on the non diag terms} allows to control a neighborhood of the diagonal in the modulated energy, and this is the key step in extending functional inequalities, such as given by Proposition \ref{non positivity}, Proposition \ref{coercivity inequality} and Theorem \ref{commutator estimate}, 
to the empirical measure.
\end{rem}

\section*{Acknowledgments}
IBP and JAC were supported by EPSRC grant number EP/V051121/1.
The research of JAC was supported by the Advanced Grant Nonlocal\--CPD (Non\-local PDEs for Complex Particle Dynamics: Phase Transitions, Patterns and Synchronization) of the European Research Council Executive Agency (ERC) under the European Union's Horizon 2020 research and innovation programme (grant agreement No. 883363).
JAC was also partially supported by the ''Maria de Maeztu'' Excellence Unit IMAG, reference CEX2020-001105-M, funded by MCIN/AEI/10.13039/501100011033/.


\end{document}